\documentclass[a4paper,12pt]{article}

\usepackage[utf8x]{inputenc}
\usepackage{amsmath}
\usepackage{amsfonts}
\usepackage{latexsym}
\usepackage{amsthm}
\usepackage{amssymb,amscd}
\usepackage{xargs}
\usepackage{graphicx}
\usepackage{stmaryrd}
\usepackage{color}
\usepackage{enumitem}
\newtheorem{thm}{Theorem}[section]

\newtheorem{lem}[thm]{Lemma}

\theoremstyle{definition}

\newtheorem{rmk}[thm]{Remark}
\newcommand{\be}{\begin{eqnarray}}
\newcommand{\ee}{\end{eqnarray}}
\newcommand{\beal}{\begin{aligned}}
\newcommand{\enal}{\end{aligned}}

\newcommand{\eps}{\varepsilon}
\newcommand{\tet}{\theta}

\newcommand{\sg}{\sigma}
\newcommand{\lb}{\lambda}
\newcommand{\E}{\mathbb{E}}
\newcommand{\Ev}{\mathbb{E}v}
\newcommand{\Eu}{\mathbb{E}u}
\newcommand{\Ef}{\mathbb{E}f}
\newcommand{\Prob}{\mathbb{P}}
\newcommand{\T}{\mathbb{T}}
\newcommand{\R}{\mathbb{R}}
\newcommand{\A}{\mathbb{A}}
\newcommand{\Z}{\mathbb{Z}}
\newcommand{\Lb}{\Lambda}
\newcommand{\om}{\omega}

\newcommand{\hatEv}{\hat{\mathbb{E}}v}

	\def\textb{\textcolor{blue}}
	\def\textr{\textcolor{red}}

\title{Random Iteration of Maps on a Cylinder and diffusive behavior}
\author{
O. Castej\'on\footnote{Universitat Polit\`ecnica de Catalunya,
oriol.castejon@gmail.com},\ \ \ 
V. Kaloshin\footnote{University of Maryland at College Park,
vadim.kaloshin@gmail.com}}

\begin{document}
\maketitle

\begin{abstract}
In this paper we propose a model of random 
compositions of maps of a cylinder, which in the simplified 
form is as follows: $(\theta,r)\in \T\times \R=\mathbb A$ and 
\begin{eqnarray} \nonumber
f_{\pm 1}:
\left(\begin{array}{c}\theta\\r\end{array}\right) & 
\longmapsto &  
\left(\begin{array}{c}\theta+r+\eps u_{\pm 1}(\theta,r).
\\
r+\eps v_{\pm 1}(\theta,r).
\end{array}\right),
\end{eqnarray} 
where $u_\pm$ and $v_\pm$ are smooth and $v_\pm$ 
are trigonometric polynomials in $\theta$ such that 
$\int v_\pm(\theta,r)\,d\theta=0$ for each $r$. 
We study the random compositions  
$$
(\theta_n,r_n)=f_{\om_{n-1}}\circ \dots \circ f_{\om_0}(\theta_0,r_0)
$$
with $\om_k \in \{-1,1\}$ with equal probabilities. 
We show that under natural non-degeneracy hypothesis
for $n\sim \eps^{-2}$ the distributions 
of $r_n-r_0$ weakly converge to a diffusion process 
with explicitly computable drift and variance. 

In the case of random iteration of the standard maps 
\begin{eqnarray} \nonumber
f_{\pm 1}:
\left(\begin{array}{c}\theta\\r\end{array}\right) & 
\longmapsto &  
\left(\begin{array}{c}\theta+r+\eps v_{\pm 1}(\theta).
\\
r+\eps v_{\pm 1}(\theta)
\end{array}\right),
\end{eqnarray}
where $v_\pm$ are trigonometric polynomials 
such that $\int v_\pm(\theta)\,d\theta=0$
we prove a vertical central limit theorem. 
Namely, for $n\sim \eps^{-2}$ the distributions 
of $r_n-r_0$ weakly converge to a normal 
distribution $\mathcal N(0,\sigma^2)$ for 
$\sigma^2=\frac14\int (v_+(\theta)-v_-(\theta))^2\,d\theta$. 

Such random models arise as a restrictions to a 
Normally Hyperbolic Invariant Lamination for 
a Hamiltonian flow of the generalized example of Arnold.
We hope that this mechanism of stochasticity sheds some 
light on formation of diffusive behaviour
at resonances of nearly integrable Hamiltonian systems. 
\end{abstract}

\tableofcontents

\section{Introduction}

\subsection{Motivation: Arnold diffusion and instabilities}

By Arnold-Louiville theorem a completely integrable Hamiltonian system 
can be written in action-angle coordinates, namely, for action 
$p$ in an open set $U\subset \R^n$ and angle $\theta$ on 
an $n$-dimensional torus $\T^n$ there  is a function $H_0(p)$ 
such that equations of motion have the form 
\[
\dot \theta=\om(p),\quad \dot p=0, \qquad \text{ where }\ \om(p):=\partial_p H_0(p).
\]
The phase space is foliated by invariant $n$-dimensional tori $\{p=p_0\}$ 
with either periodic or quasi-periodic motions 
$\theta(t)=\theta_0+t\,\om (p_0)$ (mod 1). There are many different 
examples of integrable systems (see e.g. wikipedia).

It is natural to consider small Hamiltonian perturbations 
\[
H_\eps (\theta,p)=H_0(p)+\eps H_1(\theta,p),\qquad \theta\in\T^n,\ p\in U
\]
where $\eps$ is small.  The new equations of motion become  
\[
\dot \theta=\om(p)+\eps \partial_pH_1,\quad \dot p=-\eps \partial_\theta H_1. \qquad \qquad 
\]

In the sixties, Arnold \cite{Arn64} (see also \cite{Arn89, Arn94}) 
conjectured that {\it for a generic analytic perturbation there 
are orbits $(\theta,p)(t)$ for which the variation of the actions 
is of order one, i.e. $\|p(t)-p(0)\|$ that is bounded from 
below independently of $\eps$ for all $\eps$ sufficiently small.}

See \cite{BKZ,Ch,KZ12,KZ14a,KZ14b} about recent 
progress proving this conjecture for convex Hamiltonians. 

\subsection{KAM stability}

Obstructions to Arnold diffusion, and to any form of instability 
in general, are widely known, following the works of Kolmogorov, Arnold, and Moser called nowadays KAM theory. 
The fundamental result says that for a properly 
non-degenerate $H_0$ and for all sufficiently regular perturbations 
$\eps H_1$, the system defined by $H_\eps$ still has many 
invariant $n$-dimensional tori. These tori are small deformation 
of unperturbed tori and measure of the union of these invariant 
tori  tends to the full measure as $\eps$ goes to zero. 

One consequence of KAM theory is that for $n=2$ there are 
no instabilities. Indeed, generic energy surfaces $S_E=\{H_\eps =E\}$ are $3$-dimensional manifolds, KAM tori are $2$-dimensional. 
Thus, KAM tori separate surfaces $S_E$ and prevent orbits from diffusing. 

\subsection{A priori unstable systems}
As an interesting model \cite{Arn64} Arnold proposed to study 
the following example 
\be \label{eq:Hamiltonian}
\beal 
H_\eps (p,q,I,\varphi,t)=\dfrac{I^2}{2}+H_0(p,q)
+\eps H_1(p,q,I,\varphi,t):=
\qquad \qquad  \qquad \qquad 
\\
=\underbrace{\dfrac{\ \ \ \ \ \ \ I^2\ \ \ \ \ \ \ }{2}}_{harmonic\ oscillator}+
\underbrace{\dfrac{p^2}{2}+(\cos q-1)}_{pendulum}+
\eps H_1 (p,q,I,\varphi,t),
\enal
\ee
where $q,\varphi,t\in \T$ are angles, $p,I\in \R$ are actions 
(see Fig. \ref{fig:rotor-pendulum}) and $H_1=(\cos q-1)(\cos \varphi+\cos t)$.

\begin{figure}[h]
  \begin{center}
  \includegraphics[width=10cm]{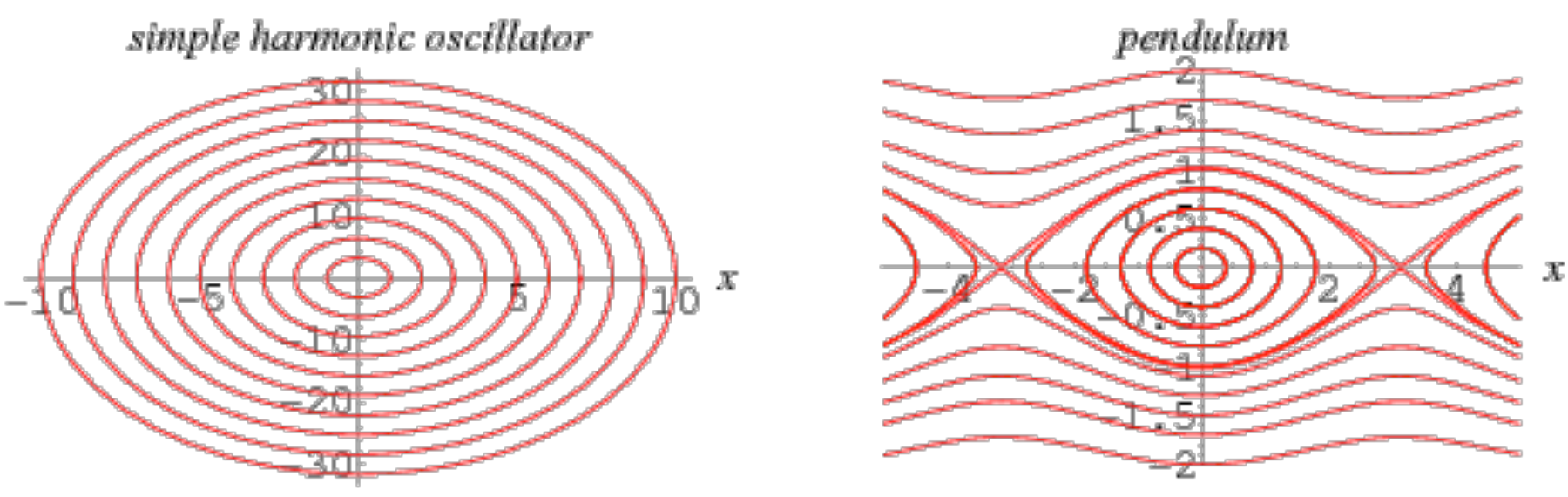}
  \end{center}
  \caption{The rotor times the pendulum }
  \label{fig:rotor-pendulum}
\end{figure}

For $\eps=0$ the system is a direct product of the harmonic 
oscillator  $\ddot \varphi=0$ and the pendulum $\ddot q=\sin q$. 
Instabilities occur when the $(p,q)$-component follows 
the separatrices $H_0(p,q)=0$ and passes near the saddle
$(p,q)=(0,0)$. Equations of motion for $H_\eps$ have 
a (normally hyperbolic) invariant cylinder $\Lb_\eps$ 
which is $\mathcal{C}^1$ close to $\Lb_0=\{p=q=0\}$.  
Systems having an invariant cylinder with a family of 
separatrix loops are called {\it an apriori unstable}. 
Since they were introduced by Arnold \cite{Arn64}, they 
received a lot of attention both in mathematics and physics 
community see e.g. \cite{Be,CY,Ch,CV,DLS,GL,T2,T3}. 

Chirikov \cite{Ch} and his followers made extensive 
numerical studies for the Arnold example. It indicates that 
{\it the $I$-displacement  behaves  randomly, 
where randomness is due to choice of initial conditions
near $H_0(p,q)=0$}.

More exactly, integration of solutions whose ``initial conditions''
randomly chosen $\eps$-close to $H_0(p,q)=0$ and integrated 
over time $\sim -\eps^{-2}\ln \eps$\,-time. This leads to the $I$-
displacement being of order of one and having some 
distribution. This coined the name for this phenomenon:
{\it Arnold diffusion}.

Let $\eps=0.01$ and $T=-\eps^{-2}\ln \eps$. 
On Fig. \ref{fig:histograms} we present several 
histograms plotting displacement of the $I$-component 
after time $T, 2T, 4T,8T$ with 6 different groups of initial 
conditions,  and histograms of $10^6$ points. In each 
group we start with a large set of initial conditions close 
to $p=q=0,\ I=I^*$.\footnote{These histograms are part of 
the forthcoming paper of the second author with P. Roldan 
with extensive numerical analysis of dynamics of the Arnold's 
example.} One of the distinct features is that only one distribution 
(a) is close symmetric, while in all others have a drift. 

\begin{figure}[h]
  \begin{center}
  \includegraphics[width=12cm]{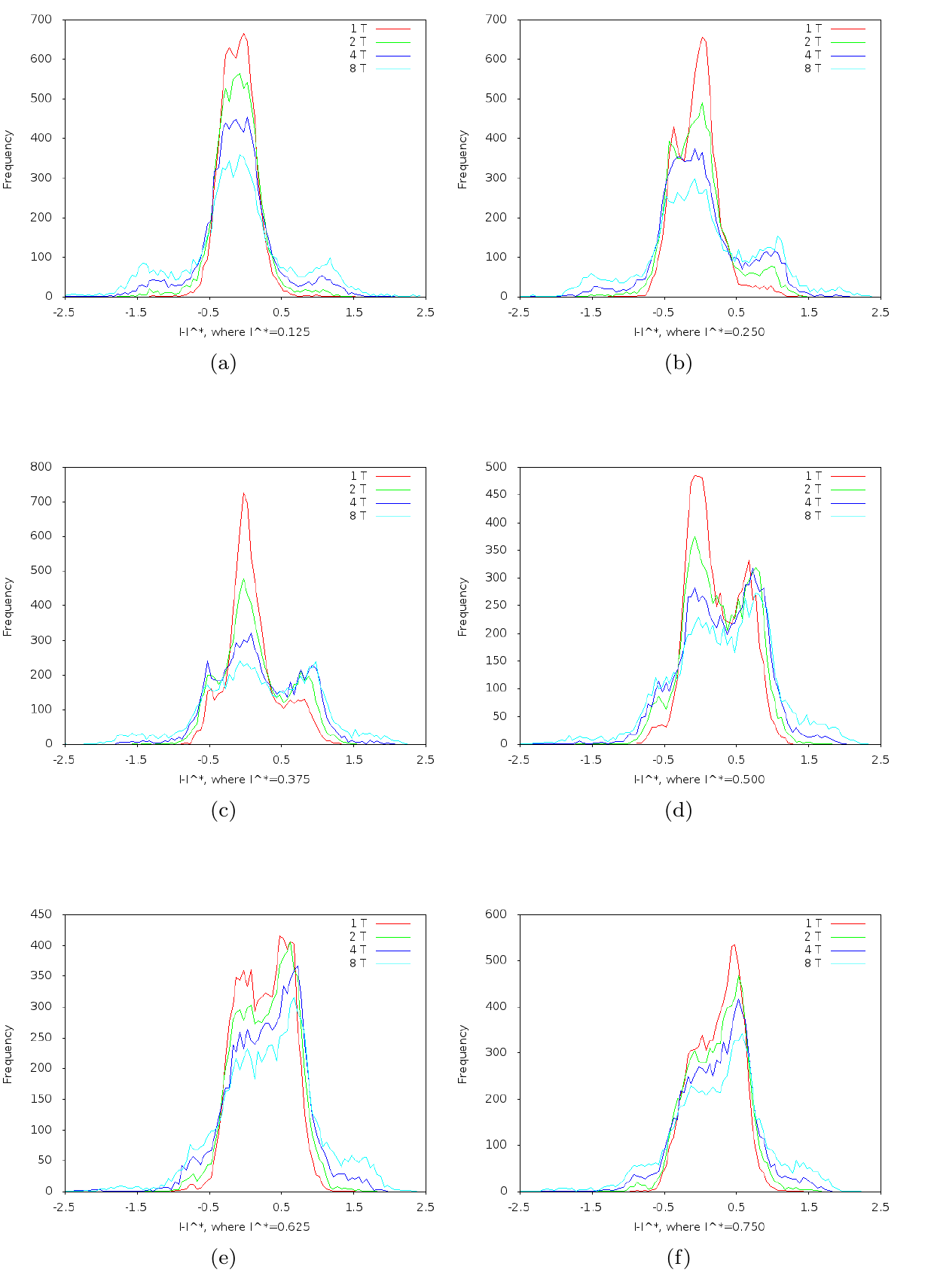}
  \end{center}
  \label{fig:histograms}
\end{figure}

A similar stochastic behaviour was observed numerically in 
many other nearly integrable problems (\cite{Ch} pg. 370, 
\cite{DL, La}, see also \cite{SLSZ}). To give another illustrative 
example consider motion of asteroids in the asteroid belt.

\subsection{Random fluctuations of eccentricity in 
Kirkwood gaps in the asteroid belt}
The asteroid belt is located between orbits of Mars and 
Jupiter and has around one million  asteroids of diameter 
of at least one kilometer. When astronoters build 
a histogram based on orbital period of asteroids there are 
well known gaps in distribution called {\it Kirkwood gaps}
(see Figure below).

\begin{figure}[h]
  \begin{center}
  \includegraphics[width=8cm]{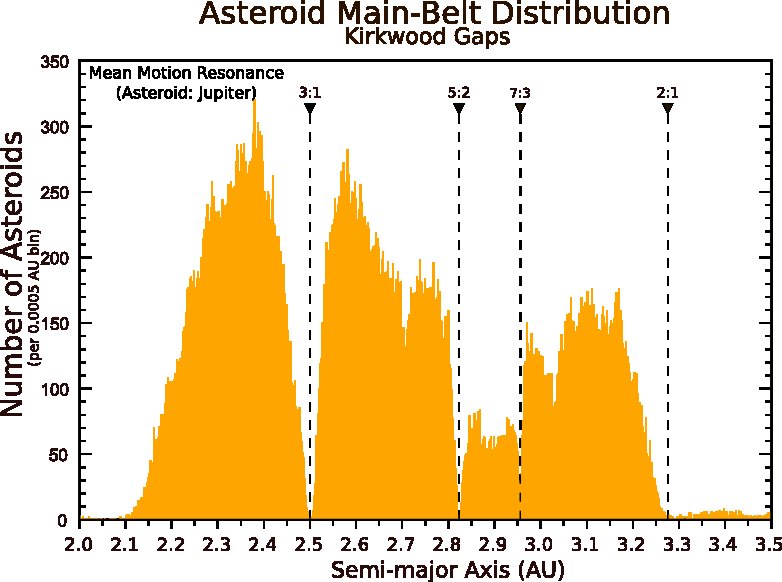}
  \end{center}
  \label{fig:Kirkwood-gaps}
\end{figure}
These gaps occur when ratio of of an asteroid and of 
Jupiter is a rational with small denominator: $1/3,2/5,3/7$.  
This correspond to so called {\it mean motion resonances
for the three body problem}. Wisdom \cite{Wi} made 
a numerical analysis of dynamics at mean motion resonance 
and observed {\it random fluctuations of eccentricity} 
of asteroids. As these fluctuations grow and eccentricity 
reaches a certain critical value an orbit of a hypothetic asteroid 
starts to cross the orbit of Mars. This eventually leads either 
to a collision of the asteroid with Mars or a close encounter. 
The latter changes the orbit so drastically that almost certainly 
it disappears from the asteroid belt. In \cite{FGKR} we only 
managed to prove existence of certain orbits whose eccentricity 
change by $O(1)$ for the restricted planar three body problem. 
Outside of these resonances one could argue that KAM 
theory provides stability \cite{Mo}.

\subsection{Random iteration of cylinder maps}
Consider the time one map of $H_\eps$, denoted 
$$
F_\eps:(p,q,I,\varphi)\to (p',q',I',\varphi').
$$ 
It turns out that for initial conditions $\eps$-close 
to $H_0(p,q)=0$, except of a hypersurface, one can define 
a return map to an $O(\eps)$-neighborhood of $(p,q)=0$. 
Often such a map is called {\it a separatrix map} and in 
the $2$-dimensional case was introduced by physicists 
Filonenko-Zaslavskii \cite{FZ}. 
In multidimensional setting such a map was defined and 
studied by Treschev \cite{PT,T1,T2,T3}. 

It turns starting near $(p,q)=0$ and iterating $F_\eps$ until 
the orbit comes back $(p,q)=0$ leads to a family of maps 
of a cylinder 
$$
f_{\eps,p,q}:(I,\varphi) \to (I',\varphi'), \qquad 
(I,\varphi)\in \A=\R\times \T
$$ 
which are close to integrable. Since at $(p,q)=0$ 
the $(p,q)$-component has a saddle, there is a sensitive 
dependence on initial condition in $(p,q)$ and returns do 
have some randomness in $(p,q)$. The precise nature 
of this randomness at the moment is not clear. There are 
several coexisting behaviours, including unstable diffusive,
stable quasi-periodic, orbits can stick to KAM tori, and
which one is dominant is to be understood. May be 
mechanism of capture into resonances \cite{Do}
is also relevant in this setting. 

In \cite{GK} we construct a normally hyperbolic lamination 
(NHL) for an open class of trigonometric perturbations 
of the form 
$$
H_1=(\cos q-1)P(\exp(i\varphi),\exp(i t)).
$$ 
Constructing unstable orbits along NHL is also discussed in 
\cite{dlL}. In general, NHL give rise to a skew shift. 
For example, let $\Sigma=\{-1,1\}^\Z$ be the space of 
infinite sequences of $-1$'s and $1$'s and 
$\sigma:\Sigma \to \Sigma$ be the standard shift. 

\vskip 0.1in 

{\it Consider a skew product of cylinder maps 
$$F:\mathbb A \times \Sigma \to \mathbb A \times \Sigma,
\qquad F(r,\theta;\om)=(f_\om(r,\theta),\sigma \om),  
$$
where each $f_\om(r,\theta)$ is a nearly integrable cylinder 
maps, in the sense that it almost preserves the $r$-component
\footnote{The reason we switch 
from the $(I,\varphi)$-coordinates on  the cylinder to 
$(r,\theta)$ is because we perform a coordinate change.}.} 

\vskip 0.1in 

The goal of the present paper is to study a wide enough 
class of skew products so that they arise in Arnold's example
with a trigonometric perturbation of the above type 
(see \cite{GK}).
 
Now we formalize our model and present the main result.

\subsection{Diffusion processes and infinitesimal generators}
\label{sec:diffusion-generators}

In order to formalise the statement about diffusive behaviour 
we need to recall some basic probabilistic notions. 
Consider a Brownian motion
 $\{B_t,\, t\ge 0\}$. 


A Brownian motion is a properly chosen limit of the standard 
random walk. A generalisation of a Brownian motion is 
{\it a diffusion process} or {\it an Ito diffusion}. To define it 
let $(\Omega,\Sigma,P)$ be a probability space. 
Let $R:[0,+\infty) \times \Omega \to \R$. It is called an Ito diffusion 
if it satisfies {\it a stochastic differential equation} of the form
\begin{equation}\label{eq:diffusion}
\mathrm{d} R_{t} = b(R_{t}) \, \mathrm{d} t + 
\sigma (R_{t}) \, \mathrm{d} B_{t},
\end{equation}
where B is an Brownian motion, $b : \R \to \R$ and 
$\sigma : \R \to \R$ are Lipschitz functions called 
the drift and the variance respectively. For a point 
$r \in \R$, let $\mathbb{P}_r$ denote the law of $X$ 
given initial data $R_0 = r$, and let $\E_r$ denote 
expectation with respect to $\mathbb{P}_r$.

The {\it infinitesimal generator} of $R$ is the operator $A$, 
which is defined to act  on suitable functions $f :\R\to \R$ by
\[
A f (r) = \lim_{t \downarrow 0} \dfrac{\E_{r} [f(R_{t})] - f(r)}{t}.
\]
The set of all functions $f$ for which this limit exists at 
a point $r$ is denoted $D_A(r)$, while $D_A$ denotes 
the set of all $f$'s for which 
the limit exists for all $r\in \R$. One can show that any 
compactly-supported $\mathcal{C}^2$  function $f$ 
lies in $D_A$ and that
\be \label{eq:diffusion-generator}
Af(r)=b(r) \dfrac{\partial f}{\partial r}+ \dfrac 12 \sigma^2(r)
\dfrac{\partial^2 f}{\partial r \partial r}.
\ee
The distribution of a diffusion process is characterise  by 
the drift $b(r)$ and the variance $\sigma^2(r)$.

\section{The model and statement of the main result}
Let $\eps>0$ be a small parameter and $l\ge 12$
be an integer. Denote by $\mathcal O_l(\eps)$  
a $\mathcal C^l$ function whose $\mathcal C^l$ norm is 
bounded by $C\eps$ with $C$ independent of $\eps$. 
Similar definition applies for a power of $\eps$. As before 
$\Sigma$ denotes $\{0,1\}^\Z$ and 
$\om=(\dots,\om_0,\dots)\in \Sigma$. 

Consider two nearly integrable maps:
\begin{eqnarray} \label{mapthetar}
f_\om:\mathbb{T}
\times \mathbb{R}
& \longrightarrow & 
\mathbb{T} \times \mathbb{R} \qquad \qquad 
\qquad \qquad \qquad \qquad
\nonumber\\
f_\om:
\left(\begin{array}{c}\theta\\r\end{array}\right) & 
\longmapsto &  
\left(\begin{array}{c}\theta+r+\eps u_{\om_0}(\theta,r)+
\mathcal O_l(\eps^{1+a},\om)
\\
r+\eps v_{\om_0}(\theta,r)+
\eps^2 w_{\om_0}(\theta,r)+\mathcal O_l(\eps^{2+a},\om)
\end{array}\right).
\end{eqnarray} 
for $\om_0\in \{-1,1\}$, where $u_{\om_0},\ v_{\om_0},$
and $w_{\om_0}$ are bounded $\mathcal{C}^l$ functions, 
$1$-periodic in $\theta$, $\mathcal O_l(\eps^{1+a},\om)$ 
and $\mathcal O_l(\eps^{2+a},\om)$ denote  remainders 
depending on $\om$ and uniformly $C^l$ bounded in $\om$, 
and $0<a\le 1/6$. Assume 
$$
\max |v_i(\theta,r)|\le 1,
$$
where maximum is taken over $i=-1,1$ and all 
$(\theta,r)\in \A$, otherwise, renormalize $\eps$. 

We study random iterations of the maps $f_1$ and $f_{-1}$, 
such that at each step the probability of performing either 
map is $1/2$. Importance of understanding iterations of 
several maps for problems of diffusion is well known
(see e.g. \cite{K,Mo}). 

Denote the expected potential and 
the difference of potentials by 
$$\Eu(\theta,r):=
\frac 12 (u_1(\theta,r)+u_{-1}(\theta,r)),\ \ \  \Ev(\theta,r):=\frac 12 (v_1(\theta,r)+v_{-1}(\theta,r)),$$
$$u(\theta,r):=\frac 12 (u_1(\theta,r)-u_{-1}(\theta,r)),\ \ \ 
v(\theta,r):=\frac 12 (v_1(\theta,r)-v_{-1}(\theta,r)).$$

Suppose the following assumptions hold:
\begin{itemize}
\item[{\bf [H0]}] ({\it zero average})
Let for each $r\in \R$ and $i=\pm 1$ we have  
$\int v_i(\theta,r)\,d\theta=0$.

\item[{\bf [H1]}] ({\it no common zeroes}) 
For each integer $n\in \Z$ potentials $v_{1}(\theta,n)$ and 
$v_{-1}(\theta,n)$ have no common zeroes and, equivalently, 
$f_1$ and $f_{-1}$ have no fixed points;

\item[{\bf [H2]}] 
for each $r\in \R$ we have $\int_0^1\ v^2(\theta,r)d\theta=:\sigma(r) \neq0$;

\item[{\bf [H3]}] The functions $v_i(\theta,r)$ are trigonometric polynomials in $\theta$, i.e. for some positive integer $d$ we have 
$$
v_i(\theta,r)=\sum_{k\in \Z,\ |k|\le d} v^{(k)}(r)\exp 2\pi ik\theta. 
$$
\end{itemize}

For $\omega\in\{-1,1\}^\Z$ we 
can rewrite the maps $f_{\om}$ in the following form:
\begin{equation*}
f_{\omega}
\left(\begin{array}{c}\theta\\r\end{array}\right)\longmapsto
\left(\begin{array}{c}\theta+r+\eps \Eu(\theta,r)+ 
\eps\omega_0 u(\theta,r)+\mathcal O_l(\eps^{1+a},\om)
\\
r+\eps \Ev(\theta,r)
+\eps\omega_0 v(\theta,r)+\eps^2 w_{\om_0}(\theta,r)
+\mathcal O_l(\eps^{2+a},\om)
\end{array}\right).
\end{equation*}

Let $n$ be positive integer and $\omega_k\in\{-1,1\}$, $k=0,\dots,n-1$, be independent random variables with $\mathbb{P}\{\omega_k=\pm1\}=1/2$ and $\Omega_n=\{\omega_0,\dots,\omega_{n-1}\}$. 
Given an initial condition $(\theta_0,r_0)$ we denote:
\be \label{eq:random-seq} 
(\theta_n,r_n):=f^n_{\Omega_n}(\theta_0,r_0)=
f_{\omega_{n-1}}\circ f_{\omega_{n-2}}\circ \cdots
\circ f_{\omega_0}(\theta_0,r_0).
\ee

\begin{itemize}
\item[{\bf [H4]}]  ({\it no common periodic orbits}) 
Suppose for any rational $r=p/q\in\mathbb Q$ 
with $p,q$ relatively prime, $1\le |q|\le 2d$ and any $\theta\in \T$   
$$
\sum_{k=1}^q\left[v_{-1}(\theta+\frac kq,r)- v_1(\theta+\frac kq,r)\right]^2\ne 0.
$$
This prohibits $f_1$ and $f_{-1}$ to have common periodic 
orbits of period $|q|$. 

\item[{\bf [H5]}] ({\it no degenerate periodic points}) 
Suppose for any rational $r=p/q\in\mathbb Q$ 
with $p,q$ relatively prime, $1\le |q|\le d$, the function:
$$\Ev_{p,q}(\theta,r)=\sum_{\substack{k\in \mathbb{Z}\\0<|kq|<d}}\Ev^{kq}(r)e^{2\pi ikq\theta}$$
has distinct non-degenerate zeroes, where $\Ev^{j}(r)$ 
denotes the $j$--th Fourier coefficient of $\Ev(\theta,r)$.
\end{itemize}

A straightforward calculation shows that:
\begin{equation}\label{mapthetanrn}
\begin{array}{rcl}
\theta_n&=&\displaystyle\theta_0+nr_0+\eps \sum_{k=0}^{n-1}
\left(\Eu(\theta_k,r_k)+ \Ev(\theta_k,r_k)\right)
\\
&&\displaystyle+
\eps\sum_{k=0}^{n-1} \omega_k
\left(u(\theta_k,r_k)+v(\theta_k,r_k)\right) 
+\mathcal O_l(n\eps^{1+a})
\bigskip\\
r_n&=&\displaystyle r_0+\eps\sum_{k=0}^{n-1}
\Ev(\theta_k,r_k)+\eps
\sum_{k=0}^{n-1}\omega_kv(\theta_k,r_k)
+\mathcal O_l(n\eps^{2+a})
\end{array}
\end{equation}
Even though these maps might not be area-preserving, 
using normal forms 
we will simplify these maps significantly on a large domain of the cylinder.



\begin{thm}\label{maintheorem}
Assume that in the notations above conditions {\bf [H0-H5]} 
hold. Let $n_\eps \eps^2 \to s>0$ as $\eps\to 0$ for some 
$s>0$. Then as $\eps \to 0$ the distribution of $r_{n_\eps}-r_0$ 
converges weakly to $R_s$, where $R_\bullet$ is 
a diffusion process of the form \eqref{eq:diffusion}, 
with the drift and the variance 
$$
b(R)=\int_0^1E_2(\theta,R)\,d\theta,
\qquad \sigma^2(R)=\int_0^1v^2(\theta,R)\,d\theta. 
$$
for some function $E_2$, defined in (\ref{eq:drift}).
\end{thm}

\begin{itemize}
\item In the case that $u_{\pm1}=v_{\pm 1}$ and they are 
independent of $r$ we have two area-preserving standard 
maps. In this case the assumptions become 
\begin{itemize}
\item{\bf [H0]} $\int v_i(\tet)d\tet=0$ for $i=\pm 1$;
\item{\bf [H1]} $v_1$ and $v_{-1}$ have no common zeroes;
\item{\bf [H2]} $v$ is not identically zero. 
\item{\bf [H3]} the functions $v_i$ are trigonometric
polynomials;
\item{\bf [H4]} the same condition as above without 
dependence on $r$;
\item{\bf [H5]} the same condition as above without 
dependence on $r$;
\end{itemize}
A good example is $u_1(\tet)=v_1(\tet)=\cos \tet$ and 
$u_{-1}(\tet)=v_{-1}(\tet)=\sin \tet$. In this case 
$$
\int_0^1E_2(\tet,r)d\tet\equiv 0,\qquad 
\sigma^2=\int_0^1 v^2(\theta)\,d\theta
$$
and for $n\le \eps^{-2}$ 
the distribution  $r_n-r_0$ converges to the zero mean 
variance $\eps n^2 \sg^2$ normal distribution, 
denoted $\mathcal N(0,\eps n^2 \sg^2)$. More generally,  
we have the following ``vertical central limit theorem'':

\begin{thm}\label{submaintheorem}
Assume that in the notations above conditions {\bf [H0-H5]} 
hold. Let $n_\eps \eps^2 \to s>0$ as $\eps\to 0$ for some 
$s>0$. Then as $\eps \to 0$ the distribution of $r_{n_\eps}-r_0$ 
converges weakly to a normal random variable 
$\mathcal{N}(0,s^2 \sigma^2).$
\end{thm}

\item Numerical experiments of Mockel \cite{Moe1} show 
that no common fixed points [H1] (resp. [H4]) is not neccessary 
for Theorem \ref{maintheorem} to hold. One could probably 
replaced by a weaker non-degeneracy condition, e.g. that 
the linearisation of maps $f_{\pm 1}$ at the common fixed 
point (resp. periodic points) are different. 

\item In \cite{Sa} Sauzin studies random iterations of the standard 
maps  $(\theta,r)\to (\theta+r+\lb \phi(\theta),r+\lb \phi(\theta)),$
where $\lb$ is chosen randomly from $\{-1,0,1\}$ and proves
the vertical central limit theorem; 
In \cite{MS,Sa2} Marco-Sauzin present examples of nearly 
integrable systems having a set of initial conditions 
exhibiting the vertical central limit theorem.

\item The condition [H3] that the functions $v_i$ are
trigonometric polynomials in $\theta$ seems redundant too,
however, removing it leads to considerable technical 
difficulties (see Section \ref{sec:different-strips} and 
Remark \ref{H3-assumption}). In  short, for perturbations  
by a trigonometric polynomial there are finitely many resonant 
zones. This finiteness considerably simplifies the analysis.


\item One can replace $\Sigma=\{0,1\}^\Z$ with 
$\Sigma_N=\{0,1,\dots,N-1\}^\Z$, consider any finite 
number of maps of the form (\ref{mapthetar}) and 
a transitive Markov chain with some transition probabilities. 
If conditions [H2--H4] are satisfied for the proper averages 
$\Ev$ of $v$, then Theorem \ref{maintheorem} holds. 
\end{itemize}

\section{Strategy of the proof}

\subsection{Strip decomposition}
The main idea of the proof is to divide the cylinder $\A$ in strips 
$\mathbb{T}\times I^i_\beta$, where $I^j_\beta\subset\mathbb{R},\ 
j\in \Z$ are intervals of size $\eps^\beta$, for any $0<\beta\le 1/5$. 
Then we will study how the random variable $r_n-r_0$ behaves 
in each strip. More precisely, 
decompose the process $r_n(\om), n\in\Z_+$  into infinitely 
many time intervals defined by stopping times 
\be \label{eq:stopping-time}
0<n_1<n_2<\dots,
\ee 
where 
\begin{itemize} 
\item $r_{n_i}(\om)$ is $\eps$-close to 
the boundary between $I^j_\beta$ and $I^{j+1}_\beta$ for 
some $j\in \Z$ 
\item $r_{n_{i+1}}(\om)$ is $\eps$-close to the other 
boundary of either $I^j_\beta$ or of $I^{j+1}_\beta$ and 
$n_{i+1}>n_i$ is the smallest integer with this property. 
\end{itemize}
Since $\eps\ll \eps^\beta$, being $\eps$-close to the boundary
of $I^j_\beta$ with a negligible error means jump from $I^j_\beta$ 
to the neighbour interval $I^{j\pm 1}_\beta$. In what follows for brevity 
we drop dependence of $r_n(\om)$'s on $\om$.

\subsection{Subdivision of the cylinder into domains with 
different quantitative behaviour}\label{sec:different-strips}

Fix $b>0$ such that $0<\beta-2b<0.04,$ small 
$\gamma>0$, and $K_i:=K_i(u_1,v_1,u_2,v_2),$ $i=1,2,$ 
depending on functions $u_j,v_j,\ j=1,2$, such that 
$K_1<K_2$ and  all are independent of $\eps$. 
Consider the $\eps^\beta$-grid in $\R$. 
Denote by $I_\beta$ a segment whose end points are in the grid. 
We distinguish among three types of strips $I_\beta$. 
We will have strips of three types as well as transition 
zones from one to another. We define:

\begin{itemize}
\item\textbf{The Real Rational (RR) case:} 
A strip $I_\beta$ is called {\it real rational} 
if there exists a rational 
$p/q\in I_\beta$, with $\gcd(p,q)=1$ and $|q|\le d$. 
Clearly, there are just finitely many strips of this kind. 
However, this case is the most complicated one and 
requires a detailed study.

\item\textbf{The Imaginary Rational (IR) case:} 
A strip $I_\beta$ is called {\it imaginary rational} if 
there exists a rational $p/q\in I_\beta$, with $\gcd(p,q)=1$ 
with $d<|q|<\eps^{-b}$. 

\medskip 

The reason we call these strips are imaginary rational, 
because the leading term of the angular dynamics
is a rational rotation, however, averaged systems 
appearing in the previous case are vanishing 
(see the next section). 

We show that the imaginary rational strips occupy
an $\mathcal O(\eps^\rho)$-fraction of the cylinder
(see Lm. \ref{lem:im-rational} in Sect. \ref{sec:measure-IR-RR}). 
We can show that orbits spend small fraction of the total time 
in these strips and global behaviour is determined 
by behaviours in the complement, which we call totally 
irrational.

\item\textbf{The Totally Irrational (TI) case:} 
A strip $I_\beta$ is called {\it totally irrational} if
$r \in I_\beta$ and $|r-p/q|<\eps^\beta$, with 
$\gcd(p,q)=1$, then $|q|>\eps^{-b}$. 

In this case, we show that there is a good ``ergodization'' 
and 
$$
\sum_{k=0}^{n-1}\omega_kv\left(\theta_0+k\frac{p}{q}\right)\approx \sum_{k=0}^{n-1}\omega_kv\left(\theta_0+kr_0^*\right).
$$
Loosely speaking, any $r_0^*\in I_\beta\cap(\mathbb{R} \setminus\mathbb{Q})$  can be treated as an irrational. 
These strips cover most of the cylinder and give 
the dominant contribution to the behaviour of $r_n-r_0$.
Eventually it will lead to the desired weak convergence to 
a diffusion process (Theorem \ref{maintheorem}).

\item\textbf{Transition zones, type I:} 
A zone is a transition zone if there is $p/q$ such 
that $\gcd(p,q)=1$ and $|q|\le d$ and it is defined 
by the corresponding annuli 
$K_1\eps^{1/2}\leq|r-p/q|\leq K_2 \eps^{1/6}$.

Analysis in these zones needs to be adapted as
``influence'' of real resonances is strong. 

\item\textbf{Transition zones, type II:} 
A zone is a transition zone if there is $p/q$ such 
that $\gcd(p,q)=1$ and $|q|\le d$ and it is defined 
by the corresponding annuli 
$K_2\eps^{1/6}\leq|r-p/q|\leq \gamma$.

Analysis in these zones requires an adjusted coordinates, 
otherwise, we still study the Totally Irrational and 
the Imaginary Rational strips inside of the type II
Transition Zones. 
\end{itemize}

\begin{rmk}\label{H3-assumption}
Notice that  finiteness of Real Rational strips follows 
from assumption [H3]. If the expected potential is not 
a trigonometric polynomial in $\theta$ this is not true.   
\end{rmk}

\subsection{The normal form}
The first step is to find a normal form, so that 
the deterministic part of map \eqref{mapthetanrn} 
is as simple as possible. In short, we shall see that the 
 deterministic system in both the Totally Irrational case and 
the Imaginary rational case are 
a small perturbation of the perfect twist map:
\begin{equation*}
\left(\begin{array}{c}\theta\\r\end{array}\right)
\longmapsto
\left(\begin{array}{c}\theta+r\\r\end{array}\right).
\end{equation*}
On the contrary, in the Real Rational case, the deterministic 
system will be close to a pendulum-like system:
\begin{equation*}
\left(\begin{array}{c}\theta\\r\end{array}\right)
\longmapsto
\left(\begin{array}{c}\theta+r\\r+\eps E(\theta,r)\end{array}\right),
\end{equation*}
for an ``averaged'' potential $E(\theta,r)$ (see e.g. Thm. 
\ref{thm:normal-form}, (\ref{NFfar})). 
We note that this system has the following approximate first integral:
\be \label{eq:perfect-first-int-prelim}
H(\theta,r)=\frac{r^2}{2}-\eps\int_0^\theta E(s,r)ds,
\ee 
so that indeed it is close to a pendulum-like system. 
This will lead to different qualitative behaviours when considering 
the random system. Inside the Real Rational strips as well as 
the transition zones we use $H$ as one of the coordinates.

The rigorous statement of these results about the normal forms 
is given in  
Theorem \ref{thm:normal-form}, Sect. \ref{sec:NormalForm}. 

\subsection{Analsys of the Martingale problem in each kind of strip}
The next step is to study the behaviour of the random system 
respectively in Totally Irrational, Imaginary Rational and 
Real Rational strips, as well as in the Transition Zones. 
This is done in Sections \ref{sec:TI-case}--\ref{sec:TZ-case}. 
More precisely, we use a discrete version of the scheme 
by Freidlin and Wentzell \cite{FW}, giving a sufficient condition 
to have weak convergence to a diffusion process  as $\eps\to0$ 
in terms of the associated Martingale problem 
(see Lemma \ref{lem:FW-suff-cond}). Now using 
the results proved below we derive the main result --- 
Theorem \ref{maintheorem}. This is done in two steps. 
First, we describe local behaviour in each strip and 
then we combine the information. Fix $s>0$. 

By the discrete version of Lemma \ref{lem:FW-suff-cond} is 
sufficient to prove that as $\eps\to 0$ any time 
$n\le s\eps^{-2}$ 
and any $(\theta_0,r_0)$ we have 
{\small 
\be\label{eq:suff-condition}
\beal 
 &\E\left(e^{-\lambda\eps^2n}f(r_{n})+  \right.& \\ 
 & \quad \ \  \eps^2
\sum_{k=0}^{n-1}e^{-\lambda\eps^2k}& \left.\left[\lambda f(r_k)-\left(b(r_k)f'(r_k)+\frac{\sigma^2(r_k)}{2}f''(r_k)\right)\right]\right)-f(r_0)\to 0, 
\enal
\ee}

We define Markov times $0=n_0<n_1<n_2 <\dots <n_{m-1}<n_m<n$
for some random $m=m(\om)$ such that each $n_k$ is the stopping 
time as in (\ref{eq:stopping-time}). Almost surely $m(\om)$ is 
finite.  We decompose the above sum 
\begin{eqnarray*}
\sum_{k=0}^m &&
\E\left(e^{-\lambda\eps^2n_{k+1}}f(r_{n_{k+1}})-
e^{-\lambda\eps^2n_{k}}f(r_{n_{k}})+ \right. \qquad \qquad \\
&& \left. \eps^2
\sum_{s=n_k}^{n_{k+1}}e^{-\lambda\eps^2s}\left[\lambda f(r_s)-\left(b(r_s)f'(r_s)+\frac{\sigma^2(r_s)}{2}f''(r_s)\right)\right]\right)
\end{eqnarray*}
and show that it converges to $f(r_0)$.

\subsubsection{A Totally Irrational Strip}
\label{sec:TI-prelim}
Let the drift and the variance be 
$$
b(r)=\int_0^1E_2(\theta,r)\,d\theta
\quad \text{ and }\quad \sigma^2(r)=\int_0^1v^2(\theta,r)\,d\theta,
$$
where the function $E_2$ is defined in (\ref{eq:drift}).
Let $r_0$ be $\eps$-close to the boundary of two totally 
irrational strips and let $n_\beta$ be stopping of hitting 
$\eps$-neighbourhoods of the adjacent boundaries. 
In Lemma \ref{lemmaexpectation} 
we prove that
\begin{eqnarray*}
&&\E\left(e^{-\lambda\eps^2n_\beta}f(r_{n_\beta})+ \right. \\
&& \left. \eps^2
\sum_{k=0}^{n_\beta-1}e^{-\lambda\eps^2k}\left[\lambda f(r_k)-\left(b(r_k)f'(r_k)+\frac{\sigma^2(r_k)}{2}f''(r_k)\right)\right]\right)\\
&& \qquad \qquad \qquad \qquad \qquad  \qquad \qquad 
- f(r_0)=\mathcal{O}(\eps^{2\beta+d}),
\end{eqnarray*}
for some $d>0$. 

\subsubsection{An Imaginary Rational Strip}
\label{sec:IR-prelim}

Let the drift and the variance be 
$$
b_{IR}(\theta,r)=\frac{1}{q}\sum_{k=0}^{q-1}E_2(\theta+kr,r)\quad 
\text{ and } \quad \sigma^2_{IR}(\theta,r)=\frac{1}{q} \sum_{k=0}^{q-1}v^2(\theta+kr,r).
$$
Let $r_0$ be $\eps$-close to the boundary of an imaginary 
rational strip and let $n_\beta$ be stopping of hitting 
$\eps$-neighbourhoods of the adjacent boundaries. 
In Lemma \ref{lemmaexpectation-IR} we prove that
\begin{eqnarray*}
 &&\E\left(e^{-\lambda\eps^2n_\beta}f(r_{n_\beta})+\right. \\
 && \left. \eps^2
\sum_{k=0}^{n_\beta-1}e^{-\lambda\eps^2k}\left[\lambda f(r_k)-\left(b(\theta_k,r_k)f'(r_k)+\frac{\sigma^2(\theta_k,r_k)}{2}f''(r_k)\right)\right]\right)\\
&&\qquad \qquad \qquad \qquad \qquad  \qquad \qquad 
\qquad \qquad -f(r_0)=\mathcal{O}(\eps^{2\beta+d}),
\end{eqnarray*}
 As one can see, the limiting process does not take place on 
 a line, since the drift and diffusion coefficient depend also on 
 the variable $\theta$.  

Notice that the drift $b(\theta,r)$ and the variance 
$\sigma(\theta,r)$ both are $\theta$-dependent functions. 
In Section \ref{sec:IR-cyl-to-line} we show that time 
spent in these strips is too small to affect the drift and the 
variance of the limiting process. 

\subsubsection{A Real Rational Strip}
\label{sec:RR-prelim}

Let in the rescaled variable $r-p/q=R\sqrt \eps$  
the drift and the variance be 
$$
b_{RR}(\theta,R)=F(\theta,R), \qquad 
\sigma^2_{RR}(\theta,R)=(R\,p/q)^2
\sum_{k=0}^{q-1}v^2(\theta+kR,R),
$$
where $F$ is some function to be defined in 
(\ref{eq:RR-drift-variance}).
Consider the Real Rational case assuming that 
$$
|r-p/q|\leq K_1\eps^{1/2} \qquad \Longleftrightarrow \qquad 
|R |\le K_1 
$$ 
that is, that $r$ is close to the ``pendulum'' domain. In 
this case, we study the process $(\theta_{qn},H_n)$ with 
$H_n:=H^{p/q}(\theta_{qn},R_{qn})$, where $H^{p/q}(\theta,R)$ is 
an approximate first integral of the deterministic system 
(\ref{eq:perfect-first-int-prelim}). In the rescaled variables 
it has the form 
$$
H^{p/q}(\theta,R)=\frac{R^2}{2}-V^{p/q}(\theta),
$$
where 
$$
V^{p/q}(\theta)=\int_0^\theta \E v_{p,q}(s,p/q)\,ds
$$
for a properly defined averaged potential 
(see Thm. \ref{thm:normal-form}, (\ref{NFnear})).
In Lemma \ref{lemmaexpectation-RR} we prove that, 
$H_n-H_0$ converges weakly to a diffusion process 
$R_t$ with $t=\eps^2n$. 

Notice that the limiting process does not take place on a line. 
In this case it takes place on a graph, similarly as in \cite{FW}. 
More precisely, consider the level sets of the function 
$H^{p/q}(\theta,R)$. The critical points of the potential 
$V^{p/q}(\theta)$ give rise to critical points of the associated 
Hamiltonian system. Moreover, if the critical point is a local minimum 
of $V$, then it corresponds to a focus of the Hamiltonian system, 
while if it is a local maximum of $V^{p/q}$, then it corresponds to 
a saddle of the Hamiltonian system. Now, if for every value 
$H\in \mathbb{R}$ we identify all the points $(\theta,R)$ in the same 
connected component of the curve $\{H^{p/q}(\theta,R)=H\}$, 
we obtain a graph $\Gamma$ (see Figure \ref{fig:potential-graph} 
for an example). The interior vertices of this graph represent 
the saddle points of the underlying Hamiltonian system jointly with 
their separatrices, while the exterior vertices represent the focuses 
of the underlying Hamiltonian system. Finally, the edges of the graph 
represent the domains that have the separatrices as boundaries. 
The process $H_n$ takes places on this graph, and so it is 
a diffusion process on a graph.

\begin{figure}[h]
  \begin{center} 
  \includegraphics[width=7.8cm]{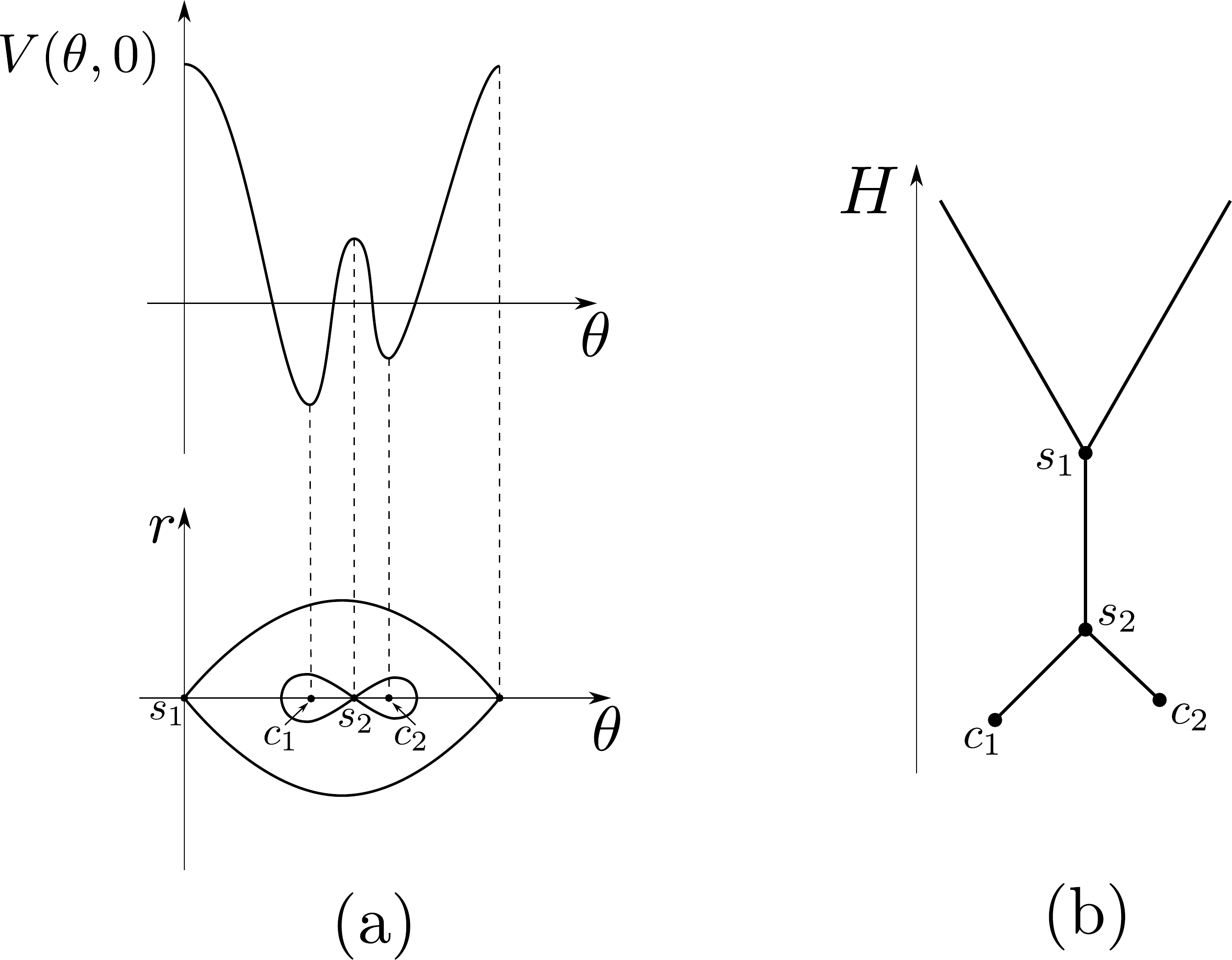}
  \end{center}
  \caption{(a) A potential and the phase portrait of its corresponding Hamiltonian system. (b) The associated graph $\Gamma$.}
  \label{fig:potential-graph}
\end{figure}

\subsubsection{A transition zone}
\label{sec:TZ-prelim}
Finally, in Lemma \ref{lemmaexpectation-TZ} we deal with 
the Transition Zones of Type I and Type II, that is the zones 
in the Real Rational strips such that 
$K_1\leq|R|\leq K_2\eps^{-1/3}$ and 
$K_2\eps^{-1/3} \leq|R|\leq \gamma\eps^{-1/2}$. 
In these strips we study the process 
$(\theta_{nq},H_{nq})=(\theta_{nq},H(\theta_{qn},R_{qn}))$. 
In this regime we fix small $\rho>0$ and subdivide each zone 
in sub-strips 
$$
I_\rho(R_0)=\{H\in\mathbb{R}\,:\,|H-H_0|\leq
|R_0|\,\eps^{1/2-\rho}\}.
$$ 
We prove that, inside each of one these sub-strips, 
as $\eps\to0$ the process $H_n-H_0$ converges weakly to 
a diffusion process $R_t$ with $t=\eps^2n$, zero drift and 
the variance:
$$
\sigma^2_{TZ}(\theta,R)=|R|^2\sum_{k=0}^{q-1}
v^2(\theta+kR,R).
$$

\subsection{From the local diffusion in the rational 
strips to the global diffusion on the line}
\label{sec:IR-cyl-to-line}

In this section we resolve the following problem. In order to 
combine all the previous results, which characterise the local 
behaviour of the process inside of infinitesimally small strips, 
to determine the global behaviour of the process in a 
$\mathcal{O}(1)-$strip. First, we prove that the Imaginary 
Rational and Real Rational strips cover a negligibly small part of 
any $\mathcal{O}(1)-$strip (see Section \ref{sec:measure-IR-RR}). 
Then, one can argue that the process is determined by 
the process in the Totally Irrational strips.

Notice both the drift $b_{IR}(\theta,r)$ and the variance 
$b_{IR}(\theta,r)$ at any Imaginary Rational strip is given 
by  $\theta$-dependent functions. Our main result
(Theorem \ref{maintheorem}), however, is a diffusion 
process on a line. To prove that this dependence does 
not enter into the global diffusion process we show that 
the process spends infinitesimal amount of time inside
of those strips as follows. 

\begin{lem} Let $(\theta_k,r_k)=f^k_{\Omega_k}(\theta_0,r_0),
k\ge 1$ be a random orbit defined by (\ref{eq:random-seq})
for some random sequence $\{\om_k\}_{k\in \Z_+}$. 
Let $n\le \eps^{-2}$ and 
$$
T_R(n)=\#\{0\le k\le n: r_k \text{ belongs to either }\qquad 
\qquad \qquad \qquad \qquad \qquad 
$$
$$
\text{an Imaginary Rational or a Real Rational strip}
\}. 
$$
Then for any $\rho>0$ and $\eps>0$ small enough
\[
\mathbb P\{ T_R(n)\ge \rho n\}\le \rho. 
\]
\end{lem} 

\begin{proof}
Define 
$$
b_{IR}(r):=\min_{\theta\in[0,1)}b_{IR}(\theta,r)
\quad \text{ and } \quad 
\sigma^2_{IR}(r):=\min_{\theta\in[0,1)}\sigma^2_{IR}(\theta,r)
$$
Consider the process $R^{IR}_t$ with the drift $b_{IR}(r)$ 
and the variance $\sigma^2_{IR}(r)$. By definition this process 
spends more time in $I_\beta$ that the process with the drift 
$b_{IR}(\theta,r)$ and the variance $\sigma^2_{IR}(\theta,r)$. 
Moreover, it is a diffusion process on a line. Then, using a local 
time argument, it can be seen that the time spent on 
a given domain is proportional to the size of this domain
up to a uniform constant. Hence, the time the original process 
spends in all the Imaginary Rational strips is infinitesimally 
small compared to the time it spends on the Totally Irrational 
ones. However, the time spent in the Imaginary Rational 
strip could be infinite and the argument would 
not be valid. This cannot happen, since if $r$ belongs 
to an Imaginary Rational strip one has that 
$\sigma^2(\theta,r)\neq 0$. Thus, it is enough to prove 
that for all imaginary rational $p/q$ one has 
$\sigma^2(\theta,p/q) \neq0$. Indeed, if this is true, then 
for $|r-p/q| \leq\eps^\beta$ and $\eps$ is sufficiently small, 
one has that $\sigma^2(\theta,r)\neq0$ by Lemma 
\ref{lem:sg-nonvanish}.

Finally, in the Real Rational case one can use a result from 
\cite{FS} that diffusion processes on a graph have well-defined 
local time. Thus, the time spent in all the Real Rational strips 
is infinitesimally small compared to the time spent in the Totally 
Irrational ones. Now one can have $\sigma^2_{RR}=0$, but 
it happens just when $r=p/q$, which follows directly from 
assumption \textbf{[H2]}. In this case, one can see that 
$b_{RR}(\theta,r)\neq0$, so that the process is non-degenerate 
and thus the fraction of time spend in the Real Rational strips 
is less than any ahead given fraction. 
\end{proof}

\begin{lem} \label{lem:sg-nonvanish}
$\sigma^2(\theta,p/q)\neq0$ if $p/q$ is 
any Imaginary Rational. 
\end{lem}

\begin{proof}
On the one hand, if $d<|q|\leq 2d$ 
this is ensured by hypothesis \textbf{[H4]}. On the other 
hand, if $|q|>2d$ then $\sigma^2(\theta,p/q)=0$ implies:
 \begin{equation}\label{vthetaeq0}
v(\theta+kp/q,p/q)=0,\qquad\quad k=0,\cdots,q-1.
 \end{equation}
Now, since $v(\theta,p/q)$ is a trigonometric polynomial 
in $\theta$ of degree $d$, it can have at most $2d$ zeros, 
or else be identically equal to zero. The latter case cannot 
occur, since by assumption \textbf{[H2]} we know that 
$$
\int_0^1v^2(\theta,r)\neq0 \quad \text{ for all }\quad r\in\mathbb{R},
$$ 
so that $v(\theta,p/q)\not\equiv0$. Thus, $v(\theta,p/q)$ has 
at most $2d$ zeros. Consequently equation \eqref{vthetaeq0} 
cannot be satisfied for all $k=0,\cdots,q-1$, since $|q|>2d$, 
so that $\sigma^2(\theta,p/q)\neq0$. The same argument 
applies to the Transition Zones.
\end{proof}

Combining these facts one can apply the arguments from 
\cite{FW}, sect. 8 and prove that the limiting diffusion process 
has the drift $b(r)$ and the variance $\sigma^2(r)$ 
corresponding to Totally Irrational strips.

\subsection{Plan of the rest of the paper}
 
In Section \ref{sec:NormalForm} we state and prove 
the normal form theorem for the expected cylinder map 
$\E f$. Main difference with a typical normal form is that 
we need to have not only the leading term in $\eps$,
but also $\eps^2$-terms. The latter terms give information
about the drift $b(r)$ (see (\ref{eq:drift})).

In Section \ref{sec:TI-case} we analyse the Totally Irrational 
case and prove approximation for the expectation from 
Section \ref{sec:TI-prelim}. 

In Section \ref{sec:IR-case} we analyse the Imaginary
Rational case and prove an analogous formula from 
Section \ref{sec:IR-prelim}. 

In Section \ref{sec:RR-case} we analyse the Real
Rational case and prove an analogous formula from 
Section \ref{sec:RR-prelim}. 

In Section \ref{sec:TZ-case} we study the Transition Zones and 
prove an analogous formula from Section \ref{sec:TZ-prelim}. 

In Section \ref{sec:measure-IR-RR} we estimate measure 
of the complement to the Totally Irrational strips and 
the Transition Zones of type II. 

In Section \ref{sec:auxiliaries} we present 
several auxiliary lemmas used in the proof.

\section{The Normal Form Theorem}\label{sec:NormalForm}
In this section we shall prove the Normal Form Theorem, 
which will allow us to deal with the simplest possible 
deterministic system. To this end, we shall enunciate 
a technical lemma which we will need in the proof of the theorem. 
This is a simplified version (sufficient for our purposes) of Lemma 3.1 
in \cite{BKZ}. 
\begin{lem}\label{lemClnorms}
Let $g(\theta,r)\in\mathcal{C}^l\left(\mathbb{T}\times B\right)$, 
where $B\subset\mathbb{R}$. Then:
\begin{enumerate}
	\item If $l_0\leq l$ and $k\neq0$, $\|g_k(r)e^{2\pi i k \theta}\|_{\mathcal{C}^{l_0}}\leq |k|^{l_0-l}\|g\|_{\mathcal{C}^{l_0}}$.
	\item Let $g_k(r)$ be some functions that satisfy 
	$\|\partial_{ r^\alpha}g_k\|_{\mathcal{C}^0}\leq M|k|^{-\alpha-2}$ 
	for all $\alpha\leq l_0$ and some $M>0$. Then:
	$$\left\|\sum_{\substack{k\in\mathbb{Z}\\0<k\leq d}}
	g_k(r)e^{2\pi i k \theta}\right\|_{\mathcal{C}^{l_0}}\leq cM,$$
	for some constant $c$ depending on $l_0$.
\end{enumerate}
\end{lem}

Let $\mathcal{R}$ be the finite set of resonances of this map, 
namely, 
$$
\mathcal{R}=\{p/q\in\mathbb{Q}\,:\,\gcd(p,q)=1, |q|<d\}.
$$
Denote by $\mathcal{O}(\eps)$ a function whose 
$\mathcal{C}^0$-norm is bounded by $C\eps$ for some 
$C$ independent of $\eps$. 

Define 
\be \label{eq:drift}
E_2(\theta,r)=\Ev(\theta,  r)\,
\partial_\theta S_1(\theta, r)+\E w(\theta,r),\quad 
b(r)=\int E_2(\theta,r)d\theta,
\ee
where $S_1$ is a certain generating function 
defined in (\ref{eq:Genfunction}--\ref{eq:HomEq}).


\begin{thm}\label{thm:normal-form}
Consider the expected map $\Ef$ of the map \eqref{mapthetar}
\begin{equation*}
\E f
\left(\begin{array}{c}\theta\\r\end{array}\right)\longmapsto
\left(\begin{array}{c}\theta+r+\eps \Eu(\theta,r)
+\mathcal O_l(\eps^{1+a})
\\
r+\eps \Ev(\theta,r)+\eps^2 \E w(\theta,r)+
\mathcal O_l(\eps^{2+a})
\end{array}\right).
\end{equation*}
Assume that the functions $\Eu(\theta,r)$, $\Ev(\theta,r)$ and $\E w(\theta,r)$ are $\mathcal{C}^l$, $l\geq3$ and $\gamma>0$
small. Then there exists $K>0$ independent of $\eps$ 
and a canonical change of variables:
\begin{eqnarray*}
\Phi:\mathbb{T}\times\mathbb{R}&\rightarrow&\mathbb{T}\times\mathbb{R},\\
(\tilde \theta,\tilde r)&\mapsto&(\theta,r),
\end{eqnarray*}
such that:
\begin{itemize}
\item If $|\tilde r-p/q|\geq \gamma$ for all $p/q\in\mathcal{R}$, 
then:
\begin{equation}\label{NFfar}
\beal 
\Phi^{-1}\circ\Ef\circ\Phi(\tilde\theta,\tilde r)=\qquad 
\qquad \qquad \qquad \qquad \qquad \qquad \\
\left(\begin{array}{c}\tilde \theta+\tilde r+\eps \Eu(\theta,r)+\eps E_1(\theta,r)+\mathcal O_l(\eps^{1+a})
\\ 
\tilde r+\eps^2E_2(\tilde\theta,\tilde r)+\mathcal{O}_l(\eps^{2+a})\end{array}\right),
\enal 
\end{equation}
where $E_1$ and $E_2$ are some $\mathcal{C}^{l-1}$ functions.
Moreover, $E_2$ verifies $\|E_2\|_{\mathcal{C}^0}\leq K$  
and it verifies:
\begin{equation}\label{averageE2}
\beal 
b(r):=\int_0^1E_2(\tilde\theta,\tilde r)d\tilde\theta=
\qquad \qquad \qquad \qquad \qquad \qquad 
\\
\int_0^1\left[\partial_r\Ev(\tilde\theta, \tilde r)\partial_\theta S_1(\tilde\theta,\tilde r)-\partial_\theta^2S_1(\tilde\theta,\tilde r)\left(\Eu(\tilde\theta,\tilde r)-\Ev(\tilde\theta,\tilde r)\right)\right]d\tilde\theta.
\enal
\end{equation}
In particular, $b(r)$ satisfies $\|b\|_{\mathcal{C}^0}\leq K$  
and in the area-preserving case 
(when $\Eu(\theta,r)=\Ev(\theta,r)=\Ev(\theta)$), $b(r)\equiv 0$.

\item If $|\tilde r-p/q|\leq 2\gamma$ for a given 
$p/q\in\mathcal{R}$, then:
{\small \begin{equation}\label{NFnear}
\Phi^{-1}\circ\Ef\circ\Phi(\tilde\theta,\tilde r)=
\end{equation}
\begin{equation}
\left(\begin{array}{c}\tilde\theta+\tilde r+\eps\left[\Eu(\tilde\theta,p/q)-\Ev(\tilde\theta,p/q)+\Ev_{p,q}(\tilde\theta,p/q)+ E_3(\tilde\theta)\right]+\mathcal{O}_l(\eps^{1+a})\\
\tilde r+\eps\Ev_{p,q}(\tilde\theta,\tilde r)+
\eps^2 E_4(\tilde\theta,\tilde r)+\mathcal{O}_l(\eps^{2+a}) \end{array}\right),\nonumber 
\end{equation}}
where $\Ev_{p,q}$ is the $\mathcal{C}^l$ function defined as: 
\begin{equation}\label{defEvpq}
 \Ev_{p,q}(\tilde\theta,\tilde r)=\sum_{k\in \mathcal{R}_{p,q}}\Ev^k(\tilde r)e^{2\pi ik\tilde\theta},\nonumber
\end{equation}
and $E_3$ is the $\mathcal{C}^{l-1}$ function: 
\begin{equation}\label{defE3}
E_3(\tilde\theta)=-\sum_{k\not\in\mathcal{R}_{p,q}}\frac{i(\Ev^k)'(p/q)}{2\pi k}e^{2\pi i k\tilde\theta},\nonumber
\end{equation}
where $\mathcal{R}_{p,q}=\{k\in\mathbb{Z}\,:\,k\neq0,\,|k|<d,\, kp/q\in\mathbb{Z}\}.$

Moreover, $E_4$ is a $\mathcal{C}^{l-1}$ function verifying 
$\|E_4\|_{\mathcal{C}^0}\leq K$.
\end{itemize}
Also, $\Phi$ is $\mathcal{C}^2$-close to the identity. 
More precisely, there exists a constant $M$ independent 
of $\eps$ such that:
\begin{equation}\label{normPhi-Id}
\|\Phi-\textup{Id}\|_{\mathcal{C}^2}\leq M\eps.
\end{equation}
\end{thm}
\begin{proof}
For each $p/q\in\mathcal{R}$ we will perform a different change. 
Since the procedure is the same for all $p/q\in\mathcal{R}$, from 
now on we fix $p/q\in\mathcal{R}$. The procedure for the rest is 
analogous.

We will consider the canonical change defined implicitly by 
a given generating function 
$S(\theta,\tilde r)=\theta\tilde r+\eps S_1(\theta,\tilde r)$, that is:
\begin{equation} 
\begin{array}{rcl}\label{eq:change}
	r&=&\partial_\theta S(\theta, \tilde r)=\tilde r+\eps \partial_\theta S_1(\theta,\tilde r)\\
	\tilde \theta&=&\partial_{\tilde r} S(\theta, \tilde r)=\theta+\eps \partial_{\tilde r}S_1(\theta,\tilde r).\nonumber
\end{array}
\end{equation}
We shall start by writing explicitly the first orders of 
the $\eps$-series of $\Phi^{-1}\circ\Ef\circ\Phi$.
If $(\theta,r)=\Phi(\tilde\theta,\tilde r)$ is the change given by the generating function $S$, then one has:
\begin{eqnarray}\label{Phi}
\beal 
 \Phi(\tilde\theta,\tilde r)=\qquad \qquad \qquad \qquad
\qquad \qquad \qquad \qquad \qquad \qquad \qquad \qquad\\
\left(\begin{array}{c}\tilde\theta-\eps\partial_{\tilde r}S_1(\tilde\theta,\tilde r)+\eps^2\partial_\theta\partial_{\tilde r}S_1(\tilde\theta,\tilde r)\partial_{\tilde r}S_1(\tilde\theta,\tilde r)+\mathcal{O}(\eps^3\|\partial_\theta^2\partial_{\tilde r}S_1(\partial_{\tilde r}S_1)^2\|_{\mathcal{C}^0})\\
 \tilde r+\eps\partial_\theta S_1(\tilde\theta,\tilde r)-\eps^2\partial_\theta^2S_1(\tilde\theta,\tilde r)\partial_{\tilde r}S_1(\tilde\theta,\tilde r)+\mathcal{O}(\eps^3\|\partial_\theta^3S_1(\partial_{\tilde r}S_1)^2\|_{\mathcal{C}^0})
\end{array}\right).
\enal 
\end{eqnarray}
Its inverse is given by:
\begin{eqnarray}\label{Phiinverse}
\beal
 \Phi^{-1}(\theta,r)=\qquad \qquad \qquad \qquad
\qquad \qquad \qquad \qquad \qquad \qquad \qquad \qquad
\\
\left(\begin{array}{c}\theta+\eps\partial_{\tilde r}S_1(\theta,r)-\eps^2\partial^2_{\tilde r}S_1(\theta,r)\partial_\theta S_1(\theta,r)+\mathcal{O}(\eps^3\|\partial^3_{\tilde r}S_1(\partial_\theta S_1)^2\|_{\mathcal{C}^0})
\\
 r-\eps\partial_\theta S_1(\theta,r)+\eps^2\partial_\theta\partial_{\tilde r}S_1(\theta,r)\partial_\theta S_1(\theta,r)+\mathcal{O}(\eps^3\|\partial_\theta\partial_{\tilde r}^2S_1(\partial_\theta S_1)^2\|_{\mathcal{C}^0})
\end{array}\right).
\enal 
\end{eqnarray}
Now, first we compute $\Ef\circ\Phi(\tilde\theta,\tilde r)$. One can see that:
\begin{equation}\label{EfPhi}
\Ef\circ\Phi(\tilde\theta,\tilde r)=\left(\begin{array}{c}
\tilde\theta + r \eps A_1+\eps^2A_2+\eps^3A_3\\ r+\eps B_1+\eps^2B_2+\eps^3B_3\end{array}\right),\nonumber
\end{equation}
where:
\begin{eqnarray}\nonumber
 A_1&=&\Eu(\tilde\theta,\tilde r)-\partial_{\tilde r}S_1(\tilde\theta,\tilde r)+\partial_\theta S_1(\tilde\theta,\tilde r)\label{defA1}\\
 A_2&=&-\partial_\theta\Eu(\tilde\theta,\tilde r)\partial_{\tilde r}S_1(\tilde\theta,\tilde r)+\partial_r\Eu(\tilde\theta,\tilde r)\partial_\theta S_1(\tilde\theta,\tilde r)+\partial_\theta\partial_{\tilde r}S_1(\tilde\theta,\tilde r)\partial_{\tilde r}S_1(\tilde\theta,\tilde r)\nonumber\\
 && -\partial_\theta^2S_1(\tilde\theta,\tilde r)\partial_{\tilde r}S_1(\tilde\theta,\tilde r),\label{defA2}\nonumber\\
 A_3&=&\mathcal{O}(\|\partial_\theta^2\partial_{\tilde r}S_1(\partial_{\tilde r} S_1)^2\|_{\mathcal{C}^0})+\mathcal{O}(\|\partial_\theta^3 S_1(\partial_{\tilde r} S_1)^2\|_{\mathcal{C}^0})\nonumber\\
 &&+\mathcal{O}(\|\Eu\|_{\mathcal{C}^1}\|\partial_\theta S_1\|_{\mathcal{C}^1}\|\partial_{\tilde r}S_1\|_{\mathcal{C}^0}) +\mathcal{O}(\|\Eu\|_{\mathcal{C}^2}(\|\partial_\theta S_1\|_{\mathcal{C}^0}+\|\partial_{\tilde r}S_1\|_{\mathcal{C}^0})^2),\nonumber
\end{eqnarray}
and:
\begin{eqnarray}
 B_1&=&\Ev(\tilde\theta,\tilde r)+\partial_\theta S_1(\tilde\theta,\tilde r),\label{defB1} \nonumber\\
 B_2&=&-\partial_\theta\Ev(\tilde\theta,\tilde r)\partial_{\tilde r}S_1(\tilde\theta,\tilde r)+\partial_r\Ev(\tilde\theta,\tilde r)\partial_\theta S_1(\tilde\theta,\tilde r)\nonumber\\&&-\partial^2_\theta S_1(\tilde\theta,\tilde r)\partial_{\tilde r}S_1(\tilde\theta,\tilde r),\label{defB2}
\\
 B_3&=&\mathcal{O}(\|\partial_\theta^3S_1(\partial_{\tilde r}S_1)^2\|_{\mathcal{C}^0})+\mathcal{O}(\|\Ev\|_{\mathcal{C}^1}\|\partial_\theta S_1\|_{\mathcal{C}^1}\|\partial_{\tilde r}S_1\|_{\mathcal{C}^0})\nonumber\\
 &&+\mathcal{O}(\|\Ev\|_{\mathcal{C}^2}(\|\partial_\theta S_1\|_{\mathcal{C}^0}+\|\partial_{\tilde r}S_1\|_{\mathcal{C}^0})^2).\label{defB3}\nonumber
\end{eqnarray}
Then, using \eqref{Phiinverse} one can see that:
\begin{equation}\label{PhiinverseEfPhi}
\Phi^{-1}\circ\Ef\circ\Phi(\tilde\theta,\tilde r)=\left(\begin{array}{c}\tilde \theta+r+\eps\hat A_1+\eps^2\hat A_2\\r+\eps\hat B_1+\eps^2\hat B_2+\eps^3\hat B_3\end{array}\right),
\end{equation}
where:
\begin{eqnarray}
 \hat A_1&=&A_1+\partial_{\tilde r}S_1(\tilde\theta+\tilde r,\tilde r),\label{defA1hat}\nonumber\\
 \hat A_2&=&A_2+\eps A_3+\mathcal{O}(\|\partial_\theta\partial_{\tilde r}S_1A_1\|_{\mathcal{C}^0})+\mathcal{O}(\|\partial_{\tilde r}^2S_1B_1\|_{\mathcal{C}^0})\nonumber\\
 &&+\mathcal{O}(\|\partial_{\tilde r}^2S_1\partial_\theta S_1\|_{\mathcal{C}^0}),\label{defA2hat}\nonumber
\end{eqnarray}
and:
\begin{eqnarray} \nonumber
 \hat B_1&=&B_1-\partial_\theta S_1(\tilde \theta+\tilde r,\tilde r)\label{defB1hat}\\
 \hat B_2&=&B_2-\partial_\theta^2S_1(\tilde\theta+\tilde r,\tilde r)A_1-\partial_{\tilde r}\partial_\theta S_1(\tilde\theta+\tilde r,\tilde r)B_1\nonumber\\
 &&+\partial_\theta\partial_{\tilde r}S_1(\tilde \theta+\tilde r,\tilde r)\partial_\theta S_1(\tilde\theta+\tilde r,\tilde r),\label{defB2hat}
\\
 \hat B_3&=&B_3+\mathcal{O}(\|\partial_\theta\partial_{\tilde r}^2S_1(\partial_\theta S_1)^2\|_{\mathcal{C}^0})\nonumber \\ &&+\mathcal{O}(\|\partial_\theta^2 S_1 (A_2+\eps A_3)\|_{\mathcal{C}^0}+\|\partial_\theta\partial_{\tilde r} S_1 B_2\|_{\mathcal{C}^0})\nonumber\\
 &&+\mathcal{O}(\|\partial_\theta^3S_1 A_1^2\|_{\mathcal{C}^0}+\|\partial_\theta^2\partial_{\tilde r} S_1A_1B_1\|_{\mathcal{C}^0}+\|\partial_\theta\partial_{\tilde r}^2S_1 B_1^2\|_{\mathcal{C}^0})\nonumber\\
 &&+\mathcal{O}(\|\partial_\theta^2\partial_{\tilde r}S_1A_1\partial_\theta S_1\|_{\mathcal{C}^0}+\|\partial_\theta\partial_{\tilde r}^2 S_1B_1\partial_\theta S_1\|_{\mathcal{C}^0})\nonumber\\
 &&+\mathcal{O}(\|\partial_\theta\partial_{\tilde r}S_1\partial_\theta^2S_1 A_1\|_{\mathcal{C}^0}+\|(\partial_\theta\partial_{\tilde r}S_1)^2B_1\|_{\mathcal{C}^0})\label{defB3hat}.\nonumber
\end{eqnarray}

%
%
%

Now that we know the terms of order $\eps$ and $\eps^2$ of $\Phi^{-1}\circ\Ef\circ\Phi$, we shall proceed to find a suitable $S_1(\theta,\tilde r)$ such that these terms are as simple as possible. More precisely, we want to simplify the second component of \eqref{PhiinverseEfPhi}. Ideally we would like that $\hat B_1=0$. 
Namely, we want to solve the following equation whenever it is possible:
$$\partial_{\theta}S_1(\tilde\theta,\tilde r)+\Ev(\tilde\theta,\tilde r)-\partial_{\theta}S_1(\tilde\theta+\tilde r,\tilde r)=0.$$
One can easily find a solution of this equation by solving the corresponding equation for the Fourier coefficients. To that aim, we write $S_1$ and $\Ev$ 
in their Fourier series:
\begin{equation}
\label{eq:Genfunction}
S_1(\theta,\tilde r)=
\sum_{k\in\mathbb{Z}}S_1^k(\tilde r)e^{2\pi ik\theta},
\end{equation}
$$\Ev(\theta,r)=\sum_{\substack{k\in\mathbb{Z}\\0<|k|\leq d}}\Ev^k(r)e^{2\pi ik\theta}.$$
It is obvious that for $k>d$ and $k=0$ we can take $S_1^k(\tilde r)=0$. For $0<k\leq d$ we obtain the following homological equation for $S_1^k(\tilde r)$:
\begin{equation} \label{eq:HomEq}
2\pi ikS_1^k(\tilde r)\left(1-e^{2\pi ik\tilde r}\right)+\Ev^k(r)=0.
\end{equation}
Clearly, this equation cannot be solved if $e^{2\pi ik\tilde r}=1$, i.e. if $k\tilde r\in\mathbb{Z}$. We note that there exists a constant $L$, independent of $\eps$, $L<d^{-1}$, such that for all $0<k\leq d$, if $\tilde r\neq p/q$ satisfies:
$$0<|\tilde r-p/q|\leq L$$
then $k\tilde r\not \in\mathbb{Z}$. Thus, restricting ourselves to the domain $|\tilde r-p/q|\leq L$, we have that if $kp/q\not\in\mathbb{Z}$ equation \eqref{eq:HomEq} always has a solution, and if $kp/q\in\mathbb{Z}$ this equation has a solution except at $\tilde r=p/q$. Moreover, in the case that the solution exists, it is equal to:
$$S_1^k(\tilde r)=\frac{i\Ev^k(r)}{2\pi k\left(1-e^{2\pi ik\tilde r}\right)}.$$
We will modify this solution slightly to make it well defined also at $\tilde r=p/q$. To this end, let us consider a $\mathcal{C}^\infty$ function 
$\mu(x)$ such that:
\begin{equation*}
\mu(x)=\left\{\begin{array}{rcl}
1 &\textrm{ if }& |x|\leq1,\\
0 &\textrm{ if }& |x|\geq2,
\end{array}\right.
\end{equation*}
and $0<\mu(x)<1$ if $x\in(1,2)$. Then we define:
$$\mu_k(\tilde r)=\mu\left(\frac{1-e^{2\pi ik\tilde r}}{2\pi k
\gamma
}\right),$$
and take:
\begin{equation}\label{defS1k}
 S_1^k(\tilde r)=\frac{i\Ev^{k}(r)(1-\mu_k(\tilde r))}{2\pi k(1-e^{2\pi ik\tilde r})}.\end{equation}
We note that this function is well defined since the numerator is identically zero in a neighbourhood of $\tilde r=p/q$, the unique zero 
of the denominator (if it is a zero indeed, that is, if 
$k\in\mathcal{R}_{p,q}$). More precisely, we claim that:
\begin{equation}\label{valuesmuk}
 \mu_k(\tilde r)=\left\{\begin{array}{ccl}1&\textrm{ if }& k\in\mathcal{R}_{p,q} \textrm{ and }|\tilde r-p/q|\leq 
\gamma
/2,\\
0&\textrm{ if }& k\in\mathcal{R}_{p,q} \textrm{ and }|\tilde r-p/q|\geq 3\gamma
,
\\
0&\textrm{ if }&k\not\in\mathcal{R}_{p,q}.
\end{array}\right.
\end{equation}
Indeed if $k\in\mathcal{R}_{p,q}$ there exists a constant $M$ 
independent of $\tilde r$ and $\eps$ such that:
$$
\frac{1}{
\gamma
}|\tilde r-p/q|(1-M|\tilde r-p/q|)
\leq
\left|\frac{1-e^{2\pi ik\tilde r}}{2\pi k
\gamma
}\right|\leq\frac{1}{
\gamma
}|\tilde r-p/q|(1+M|\tilde r-p/q|).
$$
Then, on the one hand, if $k\in\mathcal{R}_{p,q}$ and $|\tilde r-p/q|\leq \gamma
/2$ we have:
$$
\left|\frac{1-e^{2\pi ik\tilde r}}{2\pi k
\gamma
}\right|\leq\frac{1}{2}+\frac{M}{4} \gamma
<1,$$
for $\eps$ sufficiently small, and thus $\mu_k(\tilde r)=1$. 
On the other hand, if $|\tilde r-p/q|\geq 3\gamma
$ then:
$$
\left|\frac{1-e^{2\pi ik\tilde r}}{2\pi k\gamma
}\right|\geq 3-9M
\gamma
>2,$$
for $\eps$ sufficiently small, and thus $\mu_k(\tilde r)=0$. 
Finally, if $k\not\in\mathcal{R}_{p,q}$ then:
$$\left|\frac{1-e^{2\pi ik\tilde r}}{2\pi k \gamma
}\right|\geq M\gamma
>2$$
for $\gamma$ sufficiently small and then we also have 
$\mu_k(\tilde r)=0$.

Now we proceed to check that the first order terms of \eqref{PhiinverseEfPhi} take the form \eqref{NFfar} if 
$|\tilde r-p/q|\geq 3\gamma$
and \eqref{NFnear} if 
$|\tilde r-p/q|\leq \gamma$.
On the one hand, by definitions \eqref{defS1k} of 
the coefficients $S_1^k(\tilde r)$ and \eqref{defB1hat} 
of $\hat B_1$, we have:
$$
\hat B_1=\sum_{0<|k|\leq d}
\mu_k(\tilde r)\Ev^k(\tilde r)e^{2\pi i k\tilde\theta}.$$
Then, recalling \eqref{valuesmuk} we obtain:
\begin{equation}\label{valuesB1hat}
 \hat B_1=\left\{\begin{array}{lcl}0&\quad
 \textrm{ if }&|\tilde r-p/q|\geq \gamma
 \\ 
 \displaystyle\sum_{k\in\mathcal{R}_{p,q}}
 \Ev^k(\tilde r)e^{2\pi i k\tilde\theta}= 
 \Ev_{p,q}(\tilde\theta,\tilde r)&\quad
 \textrm{ if }&|\tilde r-p/q|\leq \gamma/2
 .\end{array}\right. 
\end{equation}
where we have used the definition \eqref{defEvpq} of 
$\Ev_{p,q}(\tilde\theta,\tilde r)$. On the other hand, from 
the definition \eqref{defS1k} of $S_1^k(\tilde r)$ one can check that:
\begin{eqnarray*}
 &&-\partial_{\tilde r}S_1(\tilde\theta,\tilde r)+
 \partial_{\tilde r}S_1(\tilde\theta+\tilde r,\tilde r)\\
 &&=-\partial_\theta S_1(\tilde\theta+\tilde r,\tilde r)-\sum_{0<|k|<d}\frac{i(\Ev^k)'(\tilde r)(1-\mu_k(\tilde r))+
 i\Ev^k(\tilde r)\mu_k'(\tilde r)}{2\pi k}e^{2\pi i k\tilde\theta}.
\end{eqnarray*}
Recalling definitions \eqref{defA1hat} of $\hat A_1$ and 
\eqref{defB1hat} of $\hat B_1$, this implies that:
\begin{eqnarray}\label{hatA1rewrite}
\hat A_1=\Eu(\tilde\theta,\tilde r)-\Ev(\tilde\theta,\tilde r)+\hat B_1
\qquad \qquad \qquad \qquad \qquad \\
-\sum_{0<|k|<d}\frac{i(\Ev^k)'(\tilde r)(1-\mu_k(\tilde r))+i\Ev^k(\tilde r)\mu_k'(\tilde r)}{2\pi k}e^{2\pi i k\tilde\theta}.
\end{eqnarray}
Then we use \eqref{valuesB1hat} and \eqref{valuesmuk} again, noting that $\mu'_k(\tilde r)=0$ in both regions $|\tilde r-p/q|\geq 3\gamma$
and $|\tilde r-p/q|\leq \gamma/2$.
Moreover, we note that for $|\tilde r-p/q|\leq \gamma/2$.
$$
\Ev_{p,q}(\tilde\theta,\tilde r)=
\Ev_{p,q}(\tilde\theta,p/q) +\mathcal{O}(\gamma
),$$
$$(\Ev^k)'(\tilde r)=(\Ev^k)'(p/q)+\mathcal{O}(
\gamma
).$$
Define 
\begin{equation}\label{defE1}
E_1(\tilde\theta,\tilde r)=-\sum_{0<|k|<d}
\frac{i(\Ev^k)'(\tilde r)}{2\pi k}e^{2\pi i k\tilde\theta}.
\end{equation}
Then  the same holds for $\Eu(\tilde\theta,\tilde r)$ and $\Ev(\tilde\theta,\tilde r)$: recalling definition 
\eqref{defE3} of $E_3$, equation \eqref{hatA1rewrite} yields:
{\small 
\begin{equation}\label{valuesA1hat}
 \hat A_1=\left\{\begin{array}{lcl}\Eu(\tilde\theta,\tilde r)-\Ev(\tilde\theta,\tilde r)+E_1(\tilde\theta,\tilde r)&\,\textrm{ if }&|\tilde r-p/q|\geq 3\gamma
 ,\\ \Delta \E(\tilde\theta,p/q) +\Ev_{p,q}(\tilde\theta)+E_3(\tilde\theta)+\mathcal{O}(\eps^{1/6}) &\,\textrm{ if }&|\tilde r-p/q|\leq \gamma/2
 ,\end{array}\right.
\end{equation}}
where $\Eu(\tilde\theta,p/q)-\Ev(\tilde\theta,p/q)=\Delta \E(\tilde\theta,p/q)$. 
In conclusion, by \eqref{valuesA1hat} and \eqref{valuesB1hat} 
we obtain that the first order terms of \eqref{Phiinverse} coincide 
with the first order terms of \eqref{NFfar} and \eqref{NFnear} in 
each region.

For the $\eps^2-$terms we rename $\hat B_2$ in the following way:
\begin{eqnarray} 
E_2(\tilde\theta,\tilde r)&=&\hat B_2|_{\{|r-p/q|\geq3\gamma
\}},\label{defE2} \\
E_4(\tilde\theta,\tilde r)&=&\hat B_2|_{\{|r-p/q|\leq\gamma/2
\}}.\label{defE4}
\end{eqnarray}
Now we shall see that $E_2$ verifies \eqref{averageE2}. To avoid 
long notation, in the following we do not write explicitly that 
expressions $A_i$, $B_i$, $\hat A_i$ and $\hat B_i$ are restricted 
to the region $\{|r-p/q|\geq 3 \gamma
\}$.  We note that since in 
this region we have $\hat B_1=0$ by \eqref{valuesB1hat}, recalling 
the definition \eqref{defB1hat} of $\hat B_1$ it is clear that 
$B_1=\partial_\theta S_1(\tilde\theta+\tilde r,\tilde r)$. Hence, from 
definition \eqref{defB2hat} of $\hat B_2$ it is straightforward to see that:
\begin{equation}\label{hatB2simple}
\hat B_2=B_2-\partial_\theta^2S_1(\tilde\theta+\tilde r,\tilde r)A_1.
\end{equation}
Now we recall that $\hat A_1=A_1+\partial_{\tilde r}S_1(\tilde\theta+\tilde r,\tilde r)$. Then, using \eqref{valuesA1hat}, for $|\tilde r-p/q|\geq 3\eps^{1/6}$ we obtain straightforwardly: 
\begin{equation}\label{hatA1simple}
A_1=\Eu(\tilde\theta,\tilde r)-\Eu(\tilde\theta,\tilde r)+E_1(\tilde\theta,\tilde r)-\partial_{\tilde r}S_1(\tilde\theta+\tilde r,\tilde r).
\end{equation}
Using this and the definition \eqref{defB2} of $B_2$ in expression \eqref{hatB2simple} one obtains: 
\begin{eqnarray}\label{B2hatequality}
\hat B_2|_{\{|r-p/q|\geq3\eps^{1/6}\}}&=&-\partial_\theta\Ev(\tilde\theta,\tilde r)\partial_{\tilde r}S_1(\tilde\theta,\tilde r)+\partial_r\Ev(\tilde\theta,\tilde r)\partial_\theta S_1(\tilde\theta,\tilde r) \\
&&-\partial^2_\theta S_1(\tilde\theta,\tilde r)\partial_{\tilde r}S_1(\tilde\theta,\tilde r)\nonumber \\
&&-\partial_\theta^2S_1(\tilde\theta+\tilde r,\tilde r)[\Eu(\tilde\theta,\tilde r)-\Eu(\tilde\theta,\tilde r)\nonumber\\
&&+E_1(\tilde\theta,\tilde r)-\partial_{\tilde r}S_1(\tilde\theta+\tilde r,\tilde r)].\nonumber 
\end{eqnarray}
Show that this expression has the claimed average \eqref{averageE2}. 
On the one hand, it is clear that $\partial^2_\theta S_1(\tilde\theta,\tilde r)\partial_{\tilde r}S_1(\tilde\theta,\tilde r)$ and $\partial^2_\theta S_1(\tilde\theta+\tilde r,\tilde r)\partial_{\tilde r}S_1(\tilde\theta+\tilde r,\tilde r)$ have the same average, so:
$$\int_0^1-\partial^2_\theta S_1(\tilde\theta,\tilde r)\partial_{\tilde r}S_1(\tilde\theta,\tilde r)+\partial_\theta^2S_1(\tilde\theta+\tilde r,\tilde r)\partial_{\tilde r}S_1(\tilde\theta+\tilde r,\tilde r)d\tilde\theta=0.$$
On the other hand, writing explicitly the zeroth Fourier coefficient of the product, one can see that:
$$\int_0^1-\partial_\theta\Ev(\tilde\theta,\tilde r)\partial_{\tilde r}S_1(\tilde\theta,\tilde r)-\partial_\theta^2S_1(\tilde\theta+\tilde r,\tilde r)E_1(\tilde\theta,\tilde r)d\tilde\theta=0.$$
Thus, recalling \eqref{defE2} and using these two facts in equation \eqref{B2hatequality} we obtain:
$$\int_0^1E_2(\tilde\theta,\tilde r)d\tilde\theta=
$$
$$
\int_0^1\left[\partial_r\Ev(\tilde\theta, \tilde r)\partial_\theta S_1(\tilde\theta,\tilde r)-\partial_\theta^2S_1(\tilde\theta,\tilde r)\left(\Eu(\tilde\theta,\tilde r)-\Ev(\tilde\theta,\tilde r)\right)\right]d\tilde\theta,$$
so that \eqref{averageE2} is proved.

We note that, from the definition \eqref{defS1k} of the Fourier coefficients of $S_1$, it is clear that $S_1$ is $\mathcal{C}^l$ with respect to $r$. Since it just has a finite number of nonzero coefficients, it is analytic with respect to $\theta$. Then, from the definitions \eqref{defE2} of $E_2$ and \eqref{defE4} of $E_4$ and the expression $\eqref{defB2hat}$ of $\hat B_2$, it is clear that both $E_2$ and $E_4$ are $\mathcal{C}^{l-1}$.
Finally we shall bound the $\mathcal{C}^0$-norms of the functions $E_2$, $b(r)$ and $E_4$ and also the error terms. To that aim, first let us bound the $\mathcal{C}^l$ norms of $S_1$ and its derivatives. We will use Lemma \ref{lemClnorms} and proceed similarly as in \cite{BKZ}. We note that:
\begin{enumerate}
	\item If $\mu_k(\tilde r)\neq 1$ we have $|1-e^{2\pi ik\tilde r}|>M 
	\gamma
	|k|$, and thus:
	$$\left|\frac{1}{1-e^{2\pi ik\tilde r}}\right|<M^{-1}
		\gamma^{-1}
		|k|^{-1}.$$
	\item Then, using that $\|f\circ g\|_{\mathcal{C}^l}\leq C\|f_{|_{\textrm{Im}(g)}}\|_{\mathcal{C}^l}\left(1+\|g\|_{\mathcal{C}^l}^l\right)$, we get that:
	$$\left\|\frac{1}{1-e^{2\pi ik\tilde r}}\right\|_{\mathcal{C}^l}\leq M
	\gamma
	^{-(l+1)/5}|k|^{-l-1},$$
	for some constant $M$.
	\item Using the rule for the norm of the composition again and 
	the fact that $\|\mu\|_{\mathcal{C}^l}$ is bounded independently 
	of $\eps$, we get:
	$$\|\mu_k(\tilde r)\|_{\mathcal{C}^l}\leq M\gamma
	^{-l/5}|k|^{-l},$$
	for some constant $M$, and the same bound is obtained for $\|1-\mu_k(\tilde r)\|_{\mathcal{C}^l}$.
\end{enumerate}
Using items $2$ and $3$ above and the fact that $\|\Ev^k\|_{\mathcal{C}^l}$ are bounded, we get that:
{\small \begin{eqnarray*}
\left\|\partial_{\tilde r^\alpha}\left[\frac{1-\mu_k(\tilde r)i\Ev^k(\tilde r)}{2\pi k(1-e^{2\pi ik\tilde r})}\right]\right\|_{\mathcal{C}^0}&\leq&M_1\sum_{\alpha_1+\alpha_2=\alpha}\frac{1}{2\pi|k|}\|1-\mu_k(\tilde r)\|_{\mathcal{C}^{\alpha_1}}\left\|\frac{1}{1-e^{2\pi ik\tilde r}}\right\|_{\mathcal{C}^{\alpha_2}}\\
&\leq& M_2\gamma
^{-(\alpha+1)/5}|k|^{-\alpha-2}.
\end{eqnarray*}}
Then, by item $2$ of Lemma \ref{lemClnorms}, we obtain:
$$ \|S_1\|_{\mathcal{C}^l}\leq M\gamma
^{-(l+1)/6}.$$
One can also see that $\|\partial_{\tilde r} S_1\|_{\mathcal{C}^l}\leq M\|S_1\|_{\mathcal{C}^{l+1}}$ and $\|\partial_\theta S_1\|_{\mathcal{C}^l}\leq M\|S_1\|_{\mathcal{C}^l}$. In general, one has:
\begin{equation}\label{boundderivsS1}
 \|\partial_\theta^n\partial_{\tilde r}^m S_1\|_{\mathcal{C}^l}\leq M\gamma
 ^{-(l+m+1)/6}.
\end{equation}

Now, recalling definitions \eqref{defE2} of $E_2$ and \eqref{defE4} of $E_4$, and using either expression \eqref{B2hatequality} or simply \eqref{defB2hat}, bound \eqref{boundderivsS1} yields that there exists some $K$ such that:
$$\|E_2\|_{\mathcal{C}^0}\leq K\gamma
^{-1/2},\qquad \|E_4\|_{\mathcal{C}^0}\leq K\gamma
^{-1/2}.$$
To bound $b(r)$ we use again \eqref{boundderivsS1} and 
recall that 
$\|\Ev\|_{\mathcal{C}^0}\leq K$, 
$\|\partial_r\Ev\|_{\mathcal{C}^0}\leq K$, 
$\|\Eu\|_{\mathcal{C}^0}\leq K$. Then from 
its definition \eqref{averageE2} it is clear that:
$$
\|b\|_{\mathcal{C}^0}\leq K\gamma
.$$
Similarly, one can easily bound the error terms in 
the equation for $\tilde r$:
\begin{equation}\label{errorr}
 \eps^3\hat B_3=\mathcal{O}(\eps^{2+a}),
\end{equation}
and the error terms for the equation of $\tilde \theta$:
\begin{equation}\label{errortheta}
 \eps^2\hat A_2=\mathcal{O}(\eps^{4/3}).
\end{equation}
This finishes the proof for the normal forms \eqref{NFfar} and \eqref{NFnear} (in the latter case, we have to take into account 
the extra error term of order $\mathcal{O}(\eps^{1+a})$ caused 
by the $\gamma$
--error term in \eqref{valuesA1hat}).

To prove \eqref{normPhi-Id}, we just need to recall \eqref{Phi} and 
use \eqref{boundderivsS1}. Then one obtains:
$$\|\Phi-\textrm{Id}\|_{\mathcal{C}^2}\leq M'\eps\|S_1\|_{\mathcal{C}^3}.
$$
\end{proof}

From now on we will consider that our deterministic system is 
in the normal form, and drop tildes.

\section{Analysis of the Martingale problem in the strips
of each type}
\subsection{The Totally Irrational case}\label{sec:TI-case}

First of all we note that in this case, as in the IR case, after performing the change to normal form, the $n$-th iteration of our map can be written as:
\begin{equation}
\begin{array}{rcl}\label{eq:NRmap-n}
\theta_n&=&\displaystyle\theta_0+nr_0+\mathcal{O}(n\eps),\\
r_n&=&\displaystyle r_0+\eps\sum_{k=0}^{n-1}\omega_k[v(\theta_k,r_k)+\eps v_2(\theta_k,r_k)]+\eps^2\sum_{k=0}^{n-1}E_2(\theta_k,r_k)+\mathcal{O}(n\eps^{\textb{1+a}}),
\end{array}
\end{equation}
where $v_2(\theta,r)$ is a given function which can be written explicitly in terms of $v(\theta,r)$ and $S_1(\theta,r)$.

Recall that $I_\beta$ is a  totally irrational segment if $p/q\in I_\beta$,
then $|q|>\eps^{-b},$ where $0<2b<\beta$. 

We recall that we define $b=(\beta-\rho)/2$ for a certain $0<\rho<\beta$. In the following we shall assume that $\rho$ satisfies an extra condition, which will ensure that certain inequalities are satisfied. This inequalities involve the degree of differentiability of certain $\mathcal{C}^l$ functions. We assume that $l\geq 12$. 
Then we have that:
$\frac{l-11}{l-1}>0,\ \lim_{l\to\infty}\frac{l-11}{l-1}=1.$
Thus, there exists a constant $R>0$ such that:
\begin{equation}\label{constantR}
 \frac{l-11}{l-1}>R>0,\qquad \textrm{ for all }l\geq12.
\end{equation}
Given $\beta$, satisfying:
\begin{equation}\label{conditionbeta}
0<\beta\le 1/5,
\end{equation}
then we will take $\rho$ satisfying:
\begin{equation}\label{conditionrho}
 0<\rho<R\beta.
\end{equation}

\begin{lem}\label{lemmasigma2}
Let $g$ be a $\mathcal{C}^l$ function, $l\ge 12$.
Suppose $r^*$ satisfies the following condition if for some 
rational $p/q$ we have $|r^*-p/q|<\eps^\beta$, then $|q|>\eps^{-b}$. 
Then for any $A$ such that 
$$2\beta<A<(l-1)b-\beta, \quad  
0<\tau=\textb{1-2A}\le \min\{A-2\beta,(l-1)b-A-\beta,\beta\}
$$ 
and 
$\eps$ small enough there is 
$N\le \eps^{-A}$
such that for some $K$ independent of $\eps$ and any $\theta^*$ we have:
$$\left|N\,\int_0^1g(\theta,r^*)d\theta-\sum_{k=0}^{N-1}g(\theta^*+kr^*,r^*)\right|\leq K\eps^{\tau+\beta}.$$
\textb{In particular, one can choose any 
$0<\beta\le 1/5,\ A=7\beta/3,\ \tau=A-2\beta=\beta/3,\ b=\beta/3$. }
\end{lem}
\begin{proof} Denote $g_0(r)=\int_0^1g(\theta,r)d\theta$. 
Expand $g(\theta,r)$ in its Fourier series, i.e.:
$$
g(\theta,r)=g_0(r)+\sum_{m\in \Z\setminus \{0\}} g_m(r) e^{2\pi im\theta}
$$
for some $g_m(r):\mathbb{R}\to \mathbb{C}$. Then we have:
\begin{eqnarray}\label{rewritesum}
 &&\sum_{k=0}^{N-1}(g(\theta^*+kr^*,r^*)-g_0(r^*))=
\sum_{k=0}^{N-1}\sum_{m\in \Z\setminus \{0\}}g_m(r^*) e^{2\pi im(\theta^*+kr^*,r^*))}\nonumber\\
&=&\sum_{k=0}^{N-1}\sum_{1\leq|m|\leq[\eps^{-b}]}g_m(r^*) e^{2\pi im(\theta^*+kr^*)}+\sum_{k=0}^N\sum_{|m|\geq[\eps^{-b}]}g_m(r^*) e^{2\pi im(\theta^*+kr^*)}\nonumber\\
&=&\sum_{1\leq|m|\leq[\eps^{-b}]}g_m(r^*)e^{2\pi im\theta^*}\sum_{k=0}^{N-1} e^{2\pi imkr^*}+\sum_{k=0}^{N-1}
\sum_{|m|\geq[\eps^{-b}]}g_m(r^*) e^{2\pi im(\theta^*+kr^*)}\nonumber\\
&=&\sum_{1\leq|m|\leq[\eps^{-b}]}g_m(r^*)e^{2\pi im\theta^*}
\frac{e^{2\pi iNmr^*}-1}{e^{2\pi imr^*}-1}+
\sum_{k=0}^{N-1}\sum_{|m|\geq[\eps^{-b}]}
g_m(r^*) e^{2\pi im(\theta^*+kr^*)}.\nonumber\\
\end{eqnarray}
To bound the first sum in \eqref{rewritesum} we distinguish into the following cases:
\begin{itemize}
\item If $r^*$ is rational $p/q$, we know that $|q|>\eps^{-b}$. 
\begin{itemize}
\item 
If $|q|\le \eps^{-A}$, then pick $N=|q|$ and the first sum vanishes. 

\item If $|q|>\eps^{-A}$, then by definition of $r^*$ for 
any $s/m$ with $|m|<\eps^{-b}$ we have
or $|m r^*-s|>\eps^\beta$. By the pigeon hole principle 
there exist integers $0<N=\tilde q<\eps^{-A}$ and $\tilde p$ such that
$|\tilde qr^*-\tilde p|\le 2 \eps^{A}$. 
\end{itemize}
\item If $r^*$ is irrational, consider a continuous fraction expansion 
$p_n/q_n\to r^*$ as $n\to \infty$. Choose $p'/q'=p_n/q_n$ with
$n$ such that $q_{n+1}>\eps^{-A}$. This implies that 
$|q'r^*-p'|<1/|q_{n+1}|\le \eps^A$. The same argument as above 
shows that for any $|m|<\eps^{-b}$ we have
$|m r^*-s|>\eps^\beta$.
\end{itemize}


Let $N$ be as above. Then 
$$\left|\sum_{1\leq|m|\leq[\eps^{-b}]}\ g_m(r^*)\ e^{2\pi im\theta^*}\ \frac{e^{2\pi iNmr^*}-1}{e^{2\pi imr^*}-1}\right|\leq 2 \eps^{A-\beta}\sum_{1\leq|m|\leq[\eps^{-b}]}|g_m(r^*)|.$$
We point out that since $g(\theta,\textr{r})$ is $\mathcal{C}^l$, then its Fourier coefficients satisfy $|g_m(r^*)|\le C|m|^{-l},\ m\ne 0$. Thus we can bound the first sum in \eqref{rewritesum} by:
\begin{eqnarray}\label{firstsum}
\left|\sum_{1\leq|m|\leq[\eps^{-b}]}\ g_m(r^*)\ e^{2\pi im\theta^*}\ \frac{e^{2\pi iNmr^*}-1}{e^{2\pi imr^*}-1}\right|&& \\ 
\leq\eps^{A-\beta}\sum_{1\leq|m|\leq[\eps^{-b}]}|g_m(r^*)|
&\leq&C\eps^{A-\beta}\sum_{1\leq|m|\leq[\eps^{-b}]}\frac{1}{m^2}
\leq K\eps^{A-\beta}.\nonumber\\
\end{eqnarray}

To bound the second sum we use again the bound for the Fourier coefficients $g_m(r^*)$:
\begin{equation}\label{secondsum}
\beal
\left|\sum_{k=0}^N\sum_{|m|\geq[\eps^{-b}]}g_m(r^*) e^{2\pi im(\theta+kr^*)}\right|\leq N\sum_{|m|\geq[\eps^{-b}]}\frac{1}{m^l}\le
\qquad \qquad \qquad
 \\
 KN\eps^{(l-1)b}\le K \eps^{(l-1)b-A}. 
\enal
\end{equation}

Clearly, taking $\tau=\textb{1-2A}\le 
\min\{A-2\beta,(l-1)b-A-\beta,\beta\}$, 
and substituting \eqref{firstsum} and \eqref{secondsum} in \eqref{rewritesum} we obtain the claim of the lemma.
\end{proof}

\begin{lem}\label{lemsumaverages}
Let $\beta$ satisfy \eqref{conditionbeta}, and $b=(\beta-\rho)/2$ with $\rho$ satisfying \eqref{conditionrho}. Let $n_\beta$ be an exit time of the process $(\theta_n,r_n)$ defined by \eqref{eq:NRmap-n} from some bounded domain $I_\beta$. Let $\delta>0$ be small enough. 
Suppose that $n_\beta\geq\eps^{-\textb{2}(1-\beta)+\delta}$. For all $l\geq 8$ the following holds:
\begin{enumerate}
\item Given two $\mathcal{C}^l$ functions $h:\mathbb{R}\to\mathbb{R}$ and $g:\mathbb{T}\times\mathbb{R}\to\mathbb{R}$, there exists a constant $d>0$ such that:
$$\eps^2\sum_{k=0}^{n_\beta-1}e^{-\lambda\eps^2k}h(r_k)(g(\theta_k,r_k)-g_0(r_k))=\mathcal{O}(\eps^{d}),$$
where $g_0(r)=\int_0^1g(\theta,r)d\theta$. 

\item If $n_\beta<\eps^{-\textb{2(1-\beta)+\delta}}$, then given 
a $\mathcal{C}^l$ function $g: \mathbb{T}\times\mathbb{R} \to\mathbb{R}$ and a collection of functions $h_k:\mathbb{R}\to\mathbb{R}$, with $\|h_k\|_{\mathcal{C}^0}\leq M$ and $\|h_{k+1}-h_k\|_{\mathcal{C}^0}\leq M\eps^2$ for all $k$, there exists a constant $d>0$ such that
$$\eps^2\sum_{k=0}^{n_\beta-1}h_k(r_k)(g(\theta_k,r_k)-g_0(r_k))=\mathcal{O}(\eps^{2\beta+d}),$$
where $g_0(r)=\int_0^1g(\theta,r)d\theta$.
\end{enumerate}

\end{lem}
\begin{proof}
 We shall prove both claims using Lemma \ref{lemmasigma2}. To that aim fix $2\beta<A<\min\{(l-1)b-\beta,(1-\beta)/2\}$. We note that \eqref{conditionbeta} and \eqref{conditionrho} ensure that $2\beta<\min\{(l-1)b-\beta,(1-\beta)/2\}$, so there always exists such $A$.
 
Since we have $n_\beta\geq\eps^{-(1-\beta)+\delta}$ and we have taken $A<(1-\beta)/2<1-\beta-\delta$ (the second inequality being satisfied because $\delta$ is small), we have $\eps^{-A}<n_\beta$.
Choose $N<\eps^{-A}$ from Lemma \ref{lemmasigma2} 
and write $n_\beta=P_\beta N+Q_\beta$, for some integers $P_\beta$ and  $0\le Q_\beta<N$. Then:
{\small 
\begin{eqnarray}\label{rewritePbQb}
\beal 
&&\eps^2\left|\sum_{k=0}^{n_\beta-1}e^{-\lambda\eps^2k}h(r_k)(g(\theta_k,r_k)-g_0(r_k))\right|
\qquad \qquad \qquad \qquad \qquad \qquad \\
&\leq&\eps^2\left|\sum_{k=0}^{P_\beta-1}\sum_{j=0}^{N-1}e^{-\lambda\eps^2(kN+j)}h(r_{kN+j})(g(\theta_{kN+j},r_{kN+j})-g_0(r_{kN+j}))\right| \quad \\
&+&\eps^2\left|\sum_{j=0}^{Q_\beta-1}e^{-\lambda\eps^2(P_\beta N+j)}h(r_{P_\beta N+j})(g(\theta_{P_\beta N+j},r_{P_\beta N+j})-g_0(r_{P_\beta N+j}))\right|.
\enal
\end{eqnarray}}

Let us prove item 1. We shall bound the two terms in the right hand side of \eqref{rewritePbQb} in a different way. Recall that 
in the normal form (\ref{NFfar}) we have 
\be
\beal 
f_{\omega}
\left(\begin{array}{c}\theta\\r\end{array}\right)\mapsto
\left(\begin{array}{c}\theta+r+\eps \Eu(\theta,r)+ 
\eps \omega u(\theta,r)+\textb{\mathcal O_l(\eps^{1+a})}
\\
r+\eps^2 E_2(\theta,r)+\textb{\mathcal O_l(\eps^{2+a})}.
\end{array}\right).
\enal 
\ee
On the one hand we have that for all $k\leq P_\beta$, and 
all $j\leq N$:
\begin{equation}\label{decompr}
r_{kN+j}=r_{kN}+\mathcal{O}(\textb{N}\eps^2),
\end{equation}
and:
\begin{equation}\label{decomptheta}
\theta_{kN+j}=\theta_{kN}+jr_{kN}+\mathcal{O}(\textb{N^2}\eps).
\end{equation}
Hence:
\begin{eqnarray*}
 &&e^{-\lambda\eps^2(kN+j)}
h(r_{kN+j})(g(\theta_{kN+j},r_{kN+j})-g_0(r_{kN+j}))\\
&=&e^{-\lambda\eps^2kN}
h(r_{kN})(g(\theta_{kN}+jr_{kN},r_{kN})-g_0(r_{kN}))+\mathcal{O}(e^{-\lambda\eps^2kN}\textb{N^2}\eps).
\end{eqnarray*}
Then:
\begin{eqnarray*}
 && \eps^2\left|\sum_{k=0}^{P_\beta-1}\sum_{j=0}^{N-1}e^{-\lambda\eps^2(kN+j)}
h(r_{kN+j})(g(\theta_{kN+j},r_{kN+j})-g_0(r_{kN+j}))\right|\\
&\leq&\eps^2\left|\sum_{k=0}^{P_\beta-1}e^{-\lambda\eps^2kN}
h(r_{kN})\sum_{j=0}^{N-1}(g(\theta_{kN}+jr_{kN},r_{kN})-g_0(r_{kN}))\right|\\
&&+K\textb{N^3}\eps^3\sum_{k=0}^{P_\beta-1}e^{-\lambda\eps^2kN}.
\end{eqnarray*}

Thus, using Lemma \ref{lemmasigma2} we obtain:
\begin{eqnarray*}
&&\eps^2\left|\sum_{k=0}^{P_\beta-1}e^{-\lambda\eps^2kN}
h(r_{kN})\sum_{j=0}^{N-1}(g(\theta_{kN}+jr_{kN},r_{kN})-g_0(r_{kN}))\right|\nonumber\\
&\leq&K\eps^{2+\tau+\beta}\sum_{k=0}^{P_\beta-1}e^{-\lambda\eps^2kN}|h(r_{kN})|\leq \tilde K\eps^{\tau+\beta},
\end{eqnarray*}
for some constants $K,\tilde K>0$. Moreover, we have:
$$K\textb{N^3}\eps^{3}\sum_{k=0}^{P_\beta-1}e^{-\lambda\eps^2kN}\leq \tilde K\eps^{1\textb{+\beta-2}A}.$$
Thus:
\begin{equation}\label{boundPbeta}
\beal 
 \eps^2\left|\sum_{k=0}^{P_\beta-1}\sum_{j=0}^{N-1}e^{-\lambda\eps^2(kN+j)}
h(r_{kN+j})(g(\theta_{kN+j},r_{kN+j})-g_0(r_{kN+j}))\right|
\\ \leq K(\eps^{\tau+\beta}+\eps^{1\textb{+\beta-2}A}). \qquad \qquad 
 \qquad \qquad  \qquad \qquad 
\enal 
\end{equation}

On the other hand we have:
\begin{eqnarray}\label{boundQbeta}
&&\eps^2\left|\sum_{j=0}^{Q_\beta-1}
e^{-\lambda\eps^2(P_\beta N+j)}h(r_{P_\beta N+j})(g(\theta_{P_\beta N+j},r_{P_\beta N+j})-g_0(r_{P_\beta N+j}))\right|\nonumber\\
&&\leq \eps^2K\sup_{(\theta,r)\in I_\beta}|h(r)(g(\theta,r)-g_0(r))|Q_\beta \leq \tilde K \eps^{2-A}. 
\end{eqnarray}

In conclusion, using \eqref{boundPbeta} and \eqref{boundQbeta} in equation \eqref{rewritePbQb} we obtain:
$$\eps^2\left|\sum_{k=0}^{n_\beta-1}e^{-\lambda\eps^2k}h(r_k)(g(\theta_k,r_k)-g_0(r_k))\right|\leq K(\eps^{2-A}+\eps^{1\textb{+\beta-2}A}+\eps^{\tau+\beta}).$$
Denoting $d=\min\{2-A,1\textb{+\beta-2}A,\tau+\beta\}$, and 
\textb{letting $1-2A=\tau\ge \beta/3$}.
the proof of item 1 is finished.

Now let us prove item 2. The proof is very similar to item 1. We can use the same formula \eqref{rewritePbQb} substituting $e^{-\lambda\eps^2k}h(r_k)$ by $h_k(r_k)$. Again, using \eqref{decompr} and \eqref{decomptheta} we can write:
\begin{eqnarray*}
 &&h_{kN+j}(r_{kN+j})(g(\theta_{kN+j},r_{kN+j})-g_0(r_{kN+j}))\\
&=&h_{kN}(r_{kN})(g(\theta_{kN}+jr_{kN},r_{kN})-g_0(r_{kN}))+\mathcal{O}(\textb{N^2}\eps).
\end{eqnarray*}
Then:
\begin{eqnarray*}
 && \eps^2\left|\sum_{k=0}^{P_\beta-1}\sum_{j=0}^{N-1}h_{kN+j}(r_{kN+j})(g(\theta_{kN+j},r_{kN+j})-g_0(r_{kN+j}))\right|\\
&\leq&\eps^2\left|\sum_{k=0}^{P_\beta-1}h_{kN}(r_{kN})\sum_{j=0}^{N-1}(g(\theta_{kN}+jr_{kN},r_{kN})-g_0(r_{kN}))\right|+K\textb{N^3}\eps^{3}P_\beta.
\end{eqnarray*}
Thus, using Lemma \ref{lemmasigma2} we obtain:
$$\eps^2\left|\sum_{k=0}^{P_\beta-1}h_{kN}(r_{kN})\sum_{j=0}^{N-1}(g(\theta_{kN}+jr_{kN},r_{kN})-g_0(r_{kN}))\right|\leq
$$
$$ K\eps^{2+\tau+\beta}P_\beta\leq K\eps^{\tau+2\beta},
$$
where we have used that $P_\beta\leq n_\beta\leq\eps^{-2+2\beta-\delta}\leq\eps^{-2+\beta}$. 
For the same reason we have 
$K\eps^{3-2A}P_\beta\leq K\eps^{1-2A+\beta}$.Thus:
\begin{equation}\label{boundPbeta-item2}
\beal 
 \eps^2\left|\sum_{k=0}^{P_\beta-1}\sum_{j=0}^{N-1}h_{kN+j}(r_{kN+j})(g(\theta_{kN+j},r_{kN+j})-g_0(r_{kN+j}))\right|\leq \\ K(\eps^{\tau+2\beta}+\eps^{1-2A+\beta}).  \qquad \qquad  \qquad \qquad  \qquad \qquad 
\enal 
\end{equation}

On the other hand we have:
\begin{eqnarray}\label{boundQbeta-item2}
&&\eps^2\left|\sum_{j=0}^{Q_\beta-1}h_{P_\beta N+j}(r_{P_\beta N+j})(g(\theta_{P_\beta N+j},r_{P_\beta N+j})-g_0(r_{P_\beta N+j}))\right|\nonumber\\
&&\leq \eps^2KQ_\beta \leq \tilde K \eps^{2-A}.
\end{eqnarray}

In conclusion, using \eqref{boundPbeta-item2} and \eqref{boundQbeta-item2} in equation \eqref{rewritePbQb} we obtain:
$$\eps^2\left|\sum_{k=0}^{n_\beta-1}e^{-\lambda\eps^2k}h_k(r_k)(g(\theta_k,r_k)-g_0(r_k))\right|\leq K(\eps^{2-A}+\eps^{1-2A+\beta}+\eps^{\tau+2\beta}).$$
Choosing $\tau=1-2A>0$ the last two terms are the same.
In particular, $A<1/2$ and $2-A>3/2$. Therefore, the first
term is negligible. 
\end{proof}

Let $r_0$ belong to the TI case. Consider an interval $I_\beta=\{(\theta,r)\in\mathbb{T}\times\mathbb{R}\,:\,|r-r_0|\leq\eps^\beta\}$, for some $0<\beta\le 1/5$. Denote $n_\beta\in\mathbb{N}$ the exit time from $I_\beta$, that is the first number such that $(\theta_{n_\beta},r_{n_\beta})\not\in I_\beta$. 

\begin{lem}\label{lemmaexpectation}
Let $\beta$ satisfy \eqref{conditionbeta}, and $b=(\beta-\rho)/2$ with $\rho$ satisfying \eqref{conditionrho}. Take $f:\mathbb{R}\rightarrow\mathbb{R}$ be any $\mathcal{C}^l$ function with $l\ge 12$. Then there exists $d>0$ such that for all $\lambda>0$ one has:
\begin{eqnarray*}
&&\E\left(e^{-\lambda\eps^2n_\beta}f(r_{n_\beta})+ \right. \\
&& \left. \eps^2
\sum_{k=0}^{n_\beta-1}e^{-\lambda\eps^2k}\left[\lambda f(r_k)-\left(b(r_k)f'(r_k)+\frac{\sigma^2(r_k)}{2}f''(r_k)\right)\right]\right)\\
&&\qquad \qquad \qquad \qquad \qquad  \qquad \qquad 
- f(r_0)=\mathcal{O}(\eps^{2\beta+d}),
\end{eqnarray*}
where for $E_2(\theta,r)$, defined in (\ref{eq:drift}), we have 
$$b(r)=\int_0^1E_2(\theta,r)d\theta,\qquad \sigma^2(r)=\int_0^{1}v^2(\theta,r)d\theta.$$
\end{lem}

\begin{proof}
 Let us denote:
\begin{equation}\label{defeta}
\eta=e^{-\lambda\eps^2n_\beta}f(r_{n_\beta})+\eps^2
\sum_{k=0}^{n_\beta-1}e^{-\lambda\eps^2k}\left[\lambda f(r_k)-\left(b(r_k)f'(r_k)+\frac{\sigma^2(r_k)}{2}f''(r_k)\right)\right].
\end{equation}
First of all we shall use the law of total expectation. 
Fix a small enough $\delta>0$. Then we have:
\begin{eqnarray*}
\E\left(\eta\right)&=&\E\left(\eta\,|\,\eps^{-\textb{2}(1-\beta)+\delta}\leq n_\beta\leq \eps^{-2(1-\beta)-\delta}\right)
\Prob\{\eps^{-\textb{2}(1-\beta)+\delta}\leq n_\beta\leq \eps^{-2(1-\beta)-\delta}\}\\
&+&\E\left(\eta\,|\,n_\beta< \eps^{-\textb{2}(1-\beta)+\delta}\right)
\Prob\{n_\beta<\eps^{-\textb{2}(1-\beta)+\delta}\}\\
&+&\E\left(\eta\,|\,n_\beta> \eps^{-2(1-\beta)-\delta}\right)
\Prob\{n_\beta>\eps^{-2(1-\beta)-\delta}\}.
\end{eqnarray*}
By Lemma \ref{main lemma} for $\eps$ sufficiently small
and $c>0$ independent of $\eps$ we have 
\begin{equation}\label{Prob0}
\Prob\{n_\beta<\eps^{-\textb{2}(1-\beta)+\delta}\}
\le \exp\left( -\frac{c}{\eps^{2\delta}} \right).
\end{equation}

Now we write:
$$e^{-\lambda\eps^2n_\beta}f(r_{n_\beta})=f(r_0)+
\sum_{k=0}^{n_\beta-1}\left(e^{-\lambda\eps^2(k+1)}f(r_{k+1})-e^{-\lambda\eps^2k}f(r_k)\right).$$
Doing the Taylor expansion in each term inside the sum we get:
\begin{eqnarray*}
e^{-\lambda\eps^2n_\beta}f(r_{n_\beta})&=&f(r_0)+\sum_{k=0}^{n_\beta-1}\left[-\lambda\eps^2 e^{-\lambda\eps^2k}f(r_k)+e^{-\lambda\eps^2k}f'(r_k)(r_{k+1}-r_k)\right.\\
&&\left.+\frac{1}{2}e^{-\lambda\eps^2k}f''(r_k)(r_{k+1}-r_k)^2+\mathcal{O}(e^{-\lambda\eps^2k}\eps^3)\right].
\end{eqnarray*}
Substituting this in \eqref{defeta} we get:
\begin{eqnarray}\label{eta-version2}
 \eta&=&f(r_0)+\sum_{k=0}^{n_\beta-1}
\left[e^{-\lambda\eps^2k}f'(r_k)(r_{k+1}-r_k)+
\frac{1}{2}e^{-\lambda\eps^2k}f''(r_k)(r_{k+1}-r_k)^2\right]
\nonumber
\\ 
&&-\eps^2\sum_{k=0}^{n_\beta-1}
e^{-\lambda\eps^2k}\left[b(r_k)f'(r_k)+\frac{\sigma^2(r_k)}{2}f''(r_k)\right]+\sum_{k=0}^{n_\beta-1}\mathcal{O}(e^{-\lambda\eps^2k}\eps^3).
\end{eqnarray}
We note that using \eqref{eq:NRmap-n} we can write:
$$ r_{k+1}-r_k=\eps\omega_k[v(\theta_k,r_k)+\eps v_2(\theta_k,r_k)]+\eps^2E_2(\theta_k,r_k)
+\mathcal{O}(\eps^{2+a}),$$
and also:
$$(r_{k+1}-r_k)^2=\eps^2v^2(\theta_k,r_k)+\mathcal{O}(\eps^3).
$$
Thus we can rewrite \eqref{eta-version2} as:
\begin{eqnarray}\label{eta-version3}
 \eta&=&f(r_0)+\eps\sum_{k=0}^{n_\beta-1}
e^{-\lambda\eps^2k}f'(r_k)\omega_k\left[v(\theta_k,r_k)
+\eps v_2(\theta_k,r_k)\right]\nonumber\\
 &&+\eps^2\sum_{k=0}^{n_\beta-1}
e^{-\lambda\eps^2k}f'(r_k)\left[E_2(\theta_k,r_k)-b(r_k)\right]\nonumber\\
 &&+\frac{\eps^2}{2}\sum_{k=0}^{n_\beta-1}e^{-\lambda\eps^2k}f''(r_k)
\left[v^2(\theta_k,r_k)-\sigma^2(r_k)\right]\nonumber\\
 &&+\sum_{k=0}^{n_\beta-1}\mathcal{O}(e^{-\lambda\eps^2k}\eps^{2+a}).
\end{eqnarray}

Now we distinguish between the case $\eps^{-2(1-\beta)+\delta}\leq n_\beta\leq \eps^{-2(1-\beta)-\delta}$ and $n_\beta> \eps^{-2(1-\beta)-\delta}$. Consider the former case. First, 
we show that the last term  in \eqref{eta-version3} is 
$\mathcal{O}(\eps^{\beta+d})$ for some $d>0$. Indeed, 
\begin{equation}\label{errortermsum1-small}
\left|\sum_{k=0}^{n_\beta-1}\mathcal{O}(e^{-\lambda\eps^2k}\eps^{2+a})
\right|\leq K\eps^{2+a}n_\beta\leq K\eps^{2\beta+d}, 
\end{equation}
where $d=a-\delta>0$ due to smallness of $\delta$, and $K$ is some positive constant. Now we use item 2 of Lemma \ref{lemsumaverages} in \eqref{eta-version3} twice. First we take $h_k(r)=e^{-\lambda\eps^2k}f'(r)$ and $g(\theta,r)=E_2(\theta,r)$, and after we take $h_k(r)=e^{-\lambda\eps^2k}f''(r)$ and $g(\theta,r)=v^2(\theta,r)$. Then, recalling also \eqref{errortermsum1-small}, equation \eqref{eta-version3} for $\eps^{-2(1-\beta)+\delta}\leq n_\beta\leq \eps^{-2(1-\beta)-\delta}$ yields:
\begin{equation}\label{eta-version4-small}
 \eta=f(r_0)+\eps\sum_{k=0}^{n_\beta-1}e^{-\lambda\eps^2k}f'(r_k)\omega_k\left[v(\theta_k,r_k)+\eps v_2(\theta_k,r_k)\right]+ \mathcal{O}(\eps^{2\beta+d}).
\end{equation} 

Now we focus on the case $n_\beta>\eps^{-2(1-\beta)-\delta}$. The last term in \eqref{eta-version3} can be bounded by:
\begin{equation}\label{errortermsum1-large}
\left|\sum_{k=0}^{n_\beta-1}\mathcal{O}(e^{-\lambda\eps^2k}\eps^{2+a})\right|\leq K\eps^{2+a}\sum_{k=0}^{n_\beta-1}e^{-\lambda\eps^2k}= K\eps^{2+a}
\frac{1-e^{-\lambda\eps^2n_\beta}}{1-e^{-\lambda\eps^2}} 
\leq K_\lb\eps^{a}, 
\end{equation}
for some positive constants $K$ and $K_\lambda$. Similarly 
to the previous case, we use item 1 of Lemma \ref{lemsumaverages} in \eqref{eta-version3} twice. First we take $h(r)=f'(r)$ and $g(\theta,r)=E_2(\theta,r)$, and after we take $h(r)=f''(r)$ and $g(\theta,r)=v^2(\theta,r)$. Using this and bound \eqref{errortermsum1-large} in equation \eqref{eta-version3}, we obtain the following bound for $n_\beta> \eps^{-2(1-\beta)-\delta}$:
\begin{equation}\label{eta-version4-large}
 \eta=f(r_0)+\eps\sum_{k=0}^{n_\beta-1}e^{-\lambda\eps^2k}f'(r_k)\omega_k\left[v(\theta_k,r_k)+\eps v_2(\theta_k,r_k)\right]+ \mathcal{O}(\eps^d).
\end{equation} 

Now we just need to note that since $\omega_k$ is independent of $r_k$ and $\theta_k$, we have for all $k\in\mathbb{N}$:
\begin{eqnarray*}
\E(\omega_kf'(r_k)[v(\theta_k,r_k)+\eps v_2(\theta_k,r_k)])&=&\\
\E(\omega_k)\E(f'(r_k)[v(\theta_k,r_k)+\eps v_2(\theta_k,r_k)])&=&0,
\end{eqnarray*}
because $\E(\omega_k)=0$. Thus, denoting 
$n_\eps=[\eps^{-\textb{2}(1-\beta)+\delta}]$, 
if we take expectations in \eqref{eta-version4-small} 
and \eqref{eta-version4-large} and use 
the total expectation formula, it is clear that:
\begin{eqnarray}\label{expzero}
 &&\E(\eta)-f(r_0) \nonumber\\
  =&\eps&\sum_{\substack{n\in\mathbb{N}\\n\geq n_\eps}}\E\left(\sum_{k=0}^{n-1}e^{-\lambda\eps^2k}f'(r_k)\omega_k\left[v(\theta_k,r_k)+\eps v_2(\theta_k,r_k\right])\right)\Prob\{n_\beta=n\}\ \ \ \nonumber\\
 +&\mathcal{O}&(\eps^{2\beta+d})
\Prob\{\eps^{-2(1-\beta)+\delta}\leq n_\beta\leq \eps^{-2(1-\beta)-\delta}\}+\mathcal{O}(\eps^d)\Prob\{n_\beta> \eps^{-2(1-\beta)-\delta}\}\nonumber\\
 =&\eps&\sum_{\substack{n\in\mathbb{N}\\n\geq n_\eps}}
\left(\sum_{k=0}^{n-1}e^{-\lambda\eps^2k}
\E(f'(r_k)\omega_k\left[v(\theta_k,r_k)+\eps v_2(\theta_k,r_k)\right])\right)\Prob\{n_\beta=n\} \ \ \ \\
 +&\mathcal{O}&(\eps^{2\beta+d})\Prob\{\eps^{-2(1-\beta)+\delta}\leq n_\beta\leq \eps^{-2(1-\beta)-\delta}\}+\mathcal{O}(\eps^d)\Prob\{n_\beta> \eps^{-2(1-\beta)-\delta}\}\nonumber\\
=&\mathcal{O}&(\eps^{2\beta+d})\Prob\{\eps^{-2(1-\beta)+\delta}\leq n_\beta\leq \eps^{-2(1-\beta)-\delta}\}+\mathcal{O}(\eps^d)\Prob\{n_\beta> \eps^{-2(1-\beta)-\delta}\}
 \nonumber
\end{eqnarray}

By Lemma \ref{main lemma} there exists 
a constant $a>0$ such that:
\begin{equation}\label{Prob-expsmall}
\Prob\{n_\beta>\eps^{-2(1-\beta)-\delta}\}=\mathcal{O}\left(e^{-\frac{a}{\eps^\delta}}\right).
\end{equation}
Clearly, if \eqref{Prob-expsmall} is true then $\Prob\{n_\beta>\eps^{-2(1-\beta)-\delta}\}$ is smaller than any order of $\eps$ and then \eqref{expzero} finishes the proof of the lemma.
 
To prove \eqref{Prob-expsmall}, let us define $n_\delta:=[\eps^{-2(1-\beta)-\delta}]$. Define also $n_i:=i[\eps^{-2(1-\beta)}]$, $i=0,\dots,[\eps^{-\delta}]$. Clearly, if $n_\beta>\eps^{-2(1-\beta)-\delta}$, then $|r_{n_{i+1}}-r_{n_i}|<2\eps^{\beta}$ for all $i$. In other words, we have that:
\begin{eqnarray}\label{prodrni}
\Prob\{n_\beta>\eps^{-2(1-\beta)-\delta}\}&\leq&\Prob\{|r_{n_{i+1}}-r_{n_i}|<2\eps^{\beta}\,\textrm{ for all }i=0,\dots,[\eps^{-\delta}]\}\nonumber\\
&=&\prod_{i=0}^{[\eps^{-\delta}]}\Prob\{|r_{n_{i+1}}-r_{n_i}|<2\eps^{\beta}\},
\end{eqnarray}
where in the last equality we have used that $r_{n_{i+1}}-r_{n_i}$ and $r_{n_{j+1}}-r_{n_j}$ are independent if $i\neq j$.

Now, take any $i$. Then:
$$r_{n_{i+1}}-r_{n_i}=\eps\sum_{i=n_i}^{n_{i+1}-1}\omega_kv(\theta_k,r_k)+\mathcal{O}(\eps^2(n_i-n_{i+1})).$$
Note that $\eps^2(n_i-n_{i+1})=\eps^2[\eps^{-2(1-\beta)}]\leq\eps^{2\beta}$. Thus:
\begin{equation}\label{sumdiff}
\eps\left|\sum_{i=n_i}^{n_{i+1}-1}\omega_kv(\theta_k,r_k)\right|-\mathcal{O}(\eps^{2\beta})\leq|r_{n_{i+1}}-r_{n_i}|.
\end{equation}
As a consequence, if $|r_{n_{i+1}}-r_{n_i}|\leq2\eps^{\beta}$ then $\eps\left|\sum_{i=n_i}^{n_{i+1}-1}\omega_kv(\theta_k,r_k)\right|\leq3\eps^{\beta}$. Indeed, if this latter inequality does not hold, 
then:
$$
\eps\left|\sum_{i=n_i}^{n_{i+1}-1}\omega_kv(\theta_k,r_k)\right|-\mathcal{O}(\eps^{2\beta})> 3\eps^{\beta}(1-\mathcal{O}(\eps^\beta))\geq 2\eps^{\beta}\geq|r_{n_{i+1}}-r_{n_i}|,
$$
which is a contraditciton with \eqref{sumdiff}. In other words:
$$
\Prob\{|r_{n_{i+1}}-r_{n_i}|<2\eps^\beta\}\leq 
\Prob\left\{\eps\left|\sum_{i=n_i}^{n_{i+1}-1}\omega_kv(\theta_k,r_k)\right|\leq3\eps^\beta\right\}.
$$
Now by Lemma \ref{main lemma} that $(n_{i+1}-n_i)^{-1/2}\sum_{k=n_i}^{n_{i+1}-1}\omega_kv(\theta_k,r_k)$ converges in distribution to $\xi\sim \mathcal{N}(0,c^2)$ for some $c>0$.

Thus, using that $n_{i+1}-n_i=[\eps^{-2(1-\beta)}]$, as $\eps\to0$ we obtain:
$$\Prob\left\{\eps\left|\sum_{i=n_i}^{n_{i+1}-1}\omega_kv(\theta_k,r_k)\right|\leq3\eps^\beta\right\}= \Prob\{\left|\xi\right|\leq3\}+o(1)\leq\rho<1,$$
for some constant $\rho>0$. Then, using this in \eqref{prodrni} we get:
$$\Prob\{n_\beta>\eps^{-2(1-\beta)-\delta}\}\leq\rho^{1/\eps^{\delta}}.$$
Defining $a=-\log\rho>0$ (because $\rho<1$) we obtain claim \eqref{Prob-expsmall}.
\end{proof}

\subsection{The Imaginary Rational case}\label{sec:IR-case}

In this section we deal with the imaginary rational case. The ideas are basically the same as in the TI case. Recall that  after performing the change to normal form the $n$-th iteration of our map can be written as:
\begin{equation}
\begin{array}{rcl}\label{eq:NRmap-n-IR}
\theta_n&=&\displaystyle\theta_0+nr_0+\mathcal{O}(n\eps),\\
r_n&=&\displaystyle r_0+\eps\sum_{k=0}^{n-1}\omega_k[v(\theta_k,r_k)+\eps v_2(\theta_k,r_k)]\\
&&+\eps^2\sum_{k=0}^{n-1}E_2(\theta_k,r_k)+\mathcal{O}(n\eps^{\textb{2+a}}),
\end{array}
\end{equation}
where $v_2(\theta,r)$ is a given function which can be written explicitly in terms of $v(\theta,r)$ and $S_1(\theta,r)$.

We also recall that given an imaginary rational strip $I_\beta$ there exists a unique $r^*\in I_\beta$, with $r^*=p/q$ and $|q|<\eps^{-b}$. Moreover, for all $r_0\in I_\beta$ we have $|r_0-r^*|\leq\eps^\beta.$ Then by \eqref{eq:NRmap-n-IR} for any $n\le n_\beta$ we have:
\begin{eqnarray}\label{approxIR}
\beal 
\theta_n&=&\theta_0+nr^*+\mathcal{O}(n\eps^\beta),
\\
r_n&=&r^*+\mathcal{O}(\eps^\beta).\qquad 
\enal
\end{eqnarray}
Define 
\begin{equation}\label{defg0-IR}
g_0(\theta,r)=\frac{1}{q}\sum_{i=0}^{q-1}g(\theta+ir,r).
\end{equation}

\begin{lem}\label{lemsumaverages-IR}
Let $\beta$ satisfy \eqref{conditionbeta}, and $b=(\beta-\rho)/2$ with $\rho$ satisfying \eqref{conditionrho}. Let $n_\beta$ be an exit time of the process $(\theta_n,r_n)$ defined by \eqref{eq:NRmap-n} from some bounded domain $I_\beta$. For all $l\geq1$ the following holds:
\begin{enumerate}
\item Given two $\mathcal{C}^l$ functions $h:\mathbb{R}\to\mathbb{R}$ and $g:\mathbb{T}\times\mathbb{R}\to\mathbb{R}$, there exists a constant $d>0$ such that:
$$\eps^2\sum_{k=0}^{n_\beta-1}e^{-\lambda\eps^2k}h(r_k)(g(\theta_k,r_k)-g_0(\theta_k,r_k))=\mathcal{O}(\eps^{d}).$$

\item If $n_\beta<\eps^{-2+\beta}$, then given a $\mathcal{C}^l$ function $g:\mathbb{T}\times\mathbb{R}\to\mathbb{R}$ and a collection of functions $h_k:\mathbb{R}\to\mathbb{R}$, with $\|h_k\|_{\mathcal{C}^0}\leq M$ and $\|h_{k+1}-h_k\|_{\mathcal{C}^0}\leq M\eps^2$ for all $k$, there exists a constant $d>0$ such that:
$$\eps^2\sum_{k=0}^{n_\beta-1}h_k(r_k)(g(\theta_k,r_k)-g_0(\theta_k,r_k))=\mathcal{O}(\eps^{2\beta+d}).
$$
\end{enumerate}
\end{lem}
\begin{proof}

Let us start with item 1. Write $n_\beta=P_\beta q+Q_\beta$, for some integers $P_\beta$ and  $0\le Q_\beta<q$. Then:
\begin{eqnarray}\label{rewritePbQb-IR}
\beal 
&&\eps^2\left|\sum_{k=0}^{n_\beta-1}e^{-\lambda\eps^2k}h(r_k)(g(\theta_k,r_k)-g_0(\theta_k,r_k))\right| 
\qquad \qquad \qquad \qquad \\
&\leq&\eps^2\left|\sum_{k=0}^{P_\beta-1}\sum_{j=0}^{q-1}
e^{-\lambda\eps^2(kq+j)}h(r_{kq+j})(g(\theta_{kq+j},r_{kq+j})-g_0(\theta_{kq+j},r_{kq+j}))\right|\qquad \\
&+&\eps^2\left|\sum_{j=0}^{Q_\beta-1}e^{-\lambda\eps^2(P_\beta q+j)}h(r_{P_\beta q+j})(g(\theta_{P_\beta q+j},r_{P_\beta q+j})-g_0(\theta_{P_\beta q+j},r_{P_\beta q+j}))\right|.\\
\enal 
\end{eqnarray}
On the one hand, we note that by \eqref{approxIR}
and $j\leq q<\eps^{-b}$ we have 
$$\theta_{kq+j}=\theta_{kq}+jr^*+\mathcal{O}(\eps^{\beta-b}),$$
$$r_{kq+j}=r^*+\mathcal{O}(\eps^\beta).
$$
Then for all $k\leq P_\beta$:
\begin{eqnarray*}
&&e^{-\lambda\eps^2(kq+j)}
h(r_{kq+j})(g(\theta_{kq+j},r_{kq+j})-g_0(\theta_{kq+j},r_{kq+j}))\\
&=&e^{-\lambda\eps^2kq}h(r_{kq})(g(\theta_{kq}+jr^*,r^*)-g_0(\theta_{kq}+jr^*,r^*))
+\mathcal{O}(e^{-\lambda\eps^2kq}\eps^{1-b}).
\end{eqnarray*}
Then:
\begin{eqnarray}\label{boundPbeta-IR-preliminary}
 && \eps^2\left|\sum_{k=0}^{P_\beta-1}\sum_{j=0}^{q-1}e^{-\lambda\eps^2(kq+j)}
h(r_{kq+j})(g(\theta_{kq+j},r_{kq+j})-g_0(\theta_{kq+j},r_{kq+j}))\right|\nonumber\\
&\leq&\eps^2\left|\sum_{k=0}^{P_\beta-1}e^{-\lambda\eps^2kq}
h(r_{kN})\sum_{j=0}^{q-1}(g(\theta_{kq}+jr^*,r^*)-g_0(\theta_{kq}+jr^*,r^*))\right|\nonumber\\
&&+K\eps^{3-2b}\sum_{k=0}^{P_\beta-1}e^{-\lambda\eps^2kq}.
\end{eqnarray}
Now, recalling that $r^*=p/q$, by the definition \eqref{defg0-IR} of $g_0(\theta,r)$ for all $k<P_\beta$ we have:
\begin{equation}
\sum_{j=0}^{q-1}(g(\theta_{kq}+jr^*,r^*)-g_0(\theta_{kq}+jr^*,r^*))=\sum_{j=0}^{q-1}g(\theta_{kq},r^*)-qg_0(\theta_{kq},r^*)=0.\nonumber
\end{equation}
Moreover:
\begin{equation}
K\eps^{3-2b}\sum_{k=0}^{P_\beta-1}e^{-\lambda\eps^2kq}\leq K_\lambda \ \eps^{1-2b}. \nonumber 
\end{equation}
Using these estimates 
\eqref{boundPbeta-IR-preliminary} yields:
\begin{equation}\label{boundPbeta-IR}
\beal 
\eps^2\left|\sum_{k=0}^{P_\beta-1}
\sum_{j=0}^{q-1}e^{-\lambda\eps^2(kq+j)}h(r_{kq+j})(g(\theta_{kq+j},r_{kq+j})-g_0(\theta_{kq+j},r_{kq+j}))\right|\leq
\\ K\,\eps^{1-2b}.\qquad \qquad \qquad \qquad 
\qquad \qquad \qquad 
\enal 
\end{equation}
Note that $1-2b=1-\beta+\rho>0$, since $\beta<1$ and $\beta-2b=\rho>0$.

On the other hand, we have:
\begin{eqnarray}\label{boundQbeta-IR}
&&\eps^2\left|\sum_{j=0}^{Q_\beta-1}
e^{-\lambda\eps^2(P_\beta N+j)}h(r_{P_\beta N+j})(g(_{P_\beta N+j},r_{P_\beta N+j})-g_0(r_{P_\beta N+j}))\right|\nonumber\\
&&\leq \eps^2K\sup_{(\theta,r)\in I_\beta}|h(r)(g(\theta,r)-g_0(\theta,r))|Q_\beta \leq \tilde K \eps^{2-b}. 
\end{eqnarray}
Clearly, $2-b>0$. Substituting \eqref{boundPbeta-IR} and \eqref{boundQbeta-IR} in \eqref{rewritePbQb-IR} yields item 1 of the Lemma.

Now let us consider item 2. If we take equation \eqref{boundPbeta-IR-preliminary} and substitute $e^{-\lambda\eps^2k}h(r_k)$ by $h_k(r_k)$, we can write for all $k\leq P_\beta$:
\begin{eqnarray}\label{boundPbeta-IR-preliminary-item2}
 && \eps^2\left|\sum_{k=0}^{P_\beta-1}\sum_{j=0}^{q-1}h_{kq+j}(r_{kq+j})(g(\theta_{kq+j},r_{kq+j})-g_0(\theta_{kq+j},r_{kq+j}))\right|\\
&\leq&\eps^2\left|\sum_{k=0}^{P_\beta-1}h_{kq}(r_{kq})\sum_{j=0}^{q-1}(g(\theta_{kq}+jr^*,r^*)-g_0(\theta_{kq}+jr^*,r^*))\right|+K\eps^{3-2b}P_\beta.\nonumber
\end{eqnarray}
Again:
$$\sum_{j=0}^{q-1}(g(\theta_{kq}+jr^*,r^*)-g_0(\theta_{kq}+jr^*,r^*))=0,$$
and since $P_\beta<n_\beta<\eps^{-2(1-\beta)-\delta}$
and $b=(\beta-\rho)/2$
$$K\,\eps^{3-2b} P_\beta\leq 
K\,\eps^{1+2\beta-2b}\leq 
K\eps^{1+\rho+\beta}.
$$
Then we have:
\begin{equation}\label{boundPbeta-IR-item2}
\beal 
\eps^2\left|\sum_{k=0}^{P_\beta-1}\sum_{j=0}^{q-1}e^{-\lambda\eps^2(kq+j)}h(r_{kq+j})(g(\theta_{kq+j},r_{kq+j})-g_0(\theta_{kq+j},r_{kq+j}))\right|\leq \\
 K\,\eps^{1+\rho+\beta}.
\qquad \qquad \qquad \qquad \qquad \qquad 
\enal 
\end{equation}

On the other hand we have:
\begin{eqnarray}\label{boundQbeta-IR-item2}
&&\eps^2\left|\sum_{j=0}^{Q_\beta-1}
h_{P_\beta N+j}(r_{P_\beta N+j})(g(_{P_\beta N+j},r_{P_\beta N+j})-g_0(r_{P_\beta N+j}))\right|\nonumber\\
&&\leq \eps^2KQ_\beta \leq \tilde K \eps^{2-b}. 
\end{eqnarray}
We note that $2-b-2\beta>0$ for $b=(\beta-\rho)/2$ and $\beta<4/5$ and that $1-\beta>0$ if $\beta<1$. Thus, taking $\beta<4/5$ bounds \eqref{boundPbeta-IR-item2} and \eqref{boundQbeta-IR-item2} prove item 2 of the Lemma with $d=\min\{2-b-2\beta,1-\beta,\rho\}>0$.
\end{proof}

Let $r_0$ belong to the IR case. Consider an interval $I_\beta=\{(\theta,r)\in\mathbb{T}\times\mathbb{R}\,:\,|r-r_0|\leq\eps^\beta\}$, for some $0<\beta<4/5$. Denote $n_\beta\in\mathbb{N}$ the exit time from $I_\beta$, that is the first number such that $(\theta_{n_\beta},r_{n_\beta})\not\in I_\beta$. 

\begin{lem}\label{lemmaexpectation-IR}
Let $f:\mathbb{R}\rightarrow\mathbb{R}$ be any $\mathcal{C}^l$ function with $l\ge \textr{3}$. Then, there exists $d>0$ such that 
for all $\lambda>0$ one has:
\begin{eqnarray*}
 &&\E\left(e^{-\lambda\eps^2n_\beta}f(r_{n_\beta})+\right.
\\ & & \left. \eps^2
\sum_{k=0}^{n_\beta-1}e^{-\lambda\eps^2k}\left[\lambda f(r_k)-\left(b(\theta_k,r_k)f'(r_k)+\frac{\sigma^2(\theta_k,r_k)}{2}f''(r_k)\right)\right]\right)\\
&&-f(r_0)=\mathcal{O}(\eps^{2\beta+d}),
\end{eqnarray*}
where:
$$b(\theta,r)=\frac{1}{q}\sum_{i=0}^{q-1}E_2(\theta+ir,r),\qquad \sigma^2(\theta,r)=\frac{1}{q}\sum_{i=0}^{q-1}v^2(\theta+ir,r).$$
\end{lem}

\begin{proof}
Let us fix any $0<\delta<1/6$. Again, denoting:
\begin{equation}\label{defeta-IR}
\beal 
\eta=e^{-\lambda\eps^2n_\beta}f(r_{n_\beta})&+&
\\
\eps^2
\sum_{k=0}^{n_\beta-1}e^{-\lambda\eps^2k}&&\left[\lambda f(r_k)-\left(b(\theta_k,r_k)f'(r_k)+\frac{\sigma^2(\theta_k,r_k)}{2}f''(r_k)\right)\right]
\enal 
\end{equation}
and using the law of total expectation, we have:
{\small
\begin{eqnarray}\label{lawtotalexp}
\beal 
\E\left(\eta\right)&=
\E\left(\eta\,|\,\eps^{-(1-\beta)+\delta}\leq n_\beta\leq \eps^{-2(1-\beta)-\delta}\right)
\Prob\{\eps^{-\textb{2}(1-\beta)+\delta}\leq n_\beta\leq  \eps^{-2(1-\beta)-\delta}\}\\
&+\E\left(\eta\,|\,n_\beta< \eps^{-(1-\beta)+\delta}\right)
\Prob\{n_\beta<
\eps^{-\textb{2}(1-\beta)+\delta}\}\\
&+\E\left(\eta\,|\,n_\beta> \eps^{-2(1-\beta)-\delta}\right)
\Prob\{n_\beta>\eps^{-2(1-\beta)-\delta}\}.
\enal
\end{eqnarray}}
As in the proof of Lemma \ref{lemmaexpectation}, for $\eps$ sufficiently small by Lemma \ref{main lemma} and some 
$C>0$ independent of $\eps$ we have 
$$
\Prob\{n_\beta<\eps^{-2(1-\beta)+\delta}\}\le 
\exp\left(-\frac{C}{\eps^{2\delta}}\right),
$$
and, thus, \eqref{lawtotalexp} yields:
\begin{eqnarray}\label{lawtotalexp-v2}
\E\left(\eta\right)=
\E(\eta\,|\,\eps^{-(1-\beta)+\delta}\leq n_\beta\leq \eps^{-2(1-\beta)-\delta})
\Prob\{\eps^{-(1-\beta)+\delta}\leq n_\beta\leq \eps^{-2(1-\beta)-\delta}\} \nonumber \\
+\E\left(\eta\,|\,n_\beta> \eps^{-2(1-\beta)-\delta}\right)
\Prob\{n_\beta>\eps^{-2(1-\beta)-\delta}\}.\qquad \qquad 
\qquad \qquad \qquad 
\end{eqnarray}

Now, we can write:
\begin{eqnarray*}
e^{-\lambda\eps^2n_\beta}f(r_{n_\beta})&=&f(r_0)+\sum_{k=0}^{n_\beta-1}\left[-\lambda\eps^2 e^{-\lambda\eps^2k}f(r_k)+e^{-\lambda\eps^2k}f'(r_k)(r_{k+1}-r_k)\right.\\
&&\left.+\frac{1}{2}e^{-\lambda\eps^2k}f''(r_k)(r_{k+1}-r_k)^2+\mathcal{O}(e^{-\lambda\eps^2k}\eps^3)\right],
\end{eqnarray*}
and then \eqref{defeta-IR} can be rewritten as:
\begin{eqnarray}\label{eta-version2-IR}
\beal
 \eta=f(r_0)+&&
\\ 
\sum_{k=0}^{n_\beta-1}&&\left[e^{-\lambda\eps^2k}f'(r_k)(r_{k+1}-r_k)+\frac{1}{2}e^{-\lambda\eps^2k}f''(r_k)(r_{k+1}-r_k)^2\right]\\\
 -\eps^2\sum_{k=0}^{n_\beta-1}
e^{-\lambda\eps^2k}&&\left[b(\theta_k,r_k)f'(r_k)+
\frac{\sigma^2(\theta_k,r_k)}{2}f''(r_k)\right]+\sum_{k=0}^{n_\beta-1}\mathcal{O}(e^{-\lambda\eps^2k}\eps^3).\\
\enal
\end{eqnarray}
Now, using \eqref{eq:NRmap-n-IR} we have:
$$ r_{k+1}-r_k=\eps\omega_k[v(\theta_k,r_k)+\eps v_2(\theta_k,r_k)]+\eps^2E_2(\theta_k,r_k)+
\mathcal{O}(\eps^{2+a}),$$
and:
$$(r_{k+1}-r_k)^2=\eps^2v^2(\theta_k,r_k)+\mathcal{O}(\eps^3).$$
Thus, \eqref{eta-version2-IR} writes out as:
\begin{eqnarray}\label{eta-version3-IR}
 \eta&=&f(r_0)+\sum_{k=0}^{n_\beta-1}e^{-\lambda\eps^2k}f'(r_k)\eps\omega_k\left[v(\theta_k,r_k)+\eps v_2(\theta_k,r_k)\right]\nonumber\\
 &&+\eps^2\sum_{k=0}^{n_\beta-1}e^{-\lambda\eps^2k}f'(r_k)\left[E_2(\theta_k,r_k)-b(\theta_k,r_k)\right]\nonumber\\
 &&+\frac{\eps^2}{2}\sum_{k=0}^{n_\beta-1}e^{-\lambda\eps^2k}f''(r_k)\left[v^2(\theta_k,r_k)-\sigma^2(\theta_k,r_k)\right]\nonumber\\
 &&+\sum_{k=0}^{n_\beta-1}\mathcal{O}(e^{-\lambda\eps^2k}\eps^{\textb{2+a}}).
\end{eqnarray}

Now we distinguish between the cases 
$\eps^{-2(1-\beta)+\delta}\leq n_\beta\leq \eps^{-2(1-\beta)-\delta}$ and $n_\beta>\eps^{-2(1-\beta)-\delta}$. 
We shall start assuming 
that $\eps^{-2(1-\beta)+\delta}\leq n_\beta\leq \eps^{-2(1-\beta)-\delta}$. As in the proof of Lemma \ref{lemmaexpectation} we have: 
\begin{equation}\label{errortermsum1-IR}
\left|\sum_{k=0}^{n_\beta-1}\mathcal{O}(e^{-\lambda\eps^2k}\eps^{\textb{2+a}})\right|\leq K\eps^{\textb{2+a}}n_\beta\leq K\eps^{2\beta+\textb{a}-\delta}.
\end{equation}
We note that $1/6-\delta>0$ since we have taken $\delta<1/6$. Now we use item 2 of Lemma \ref{lemsumaverages-IR} in \eqref{eta-version3-IR} twice, taking first $h_k(r)=e^{-\lambda\eps^2k}f'(r)$ and $g(\theta,r)=E_2(\theta,r)$, and after $h_k(r)=e^{-\lambda\eps^2k}f''(r)$ and $g(\theta,r)=v^2(\theta,r)$. Then, using also \eqref{errortermsum1-IR}, equation \eqref{eta-version3-IR} yields:
\begin{equation}\label{eta-version4-IR}
 \eta=f(r_0)+\sum_{k=0}^{n_\beta-1}e^{-\lambda\eps^2k}f'(r_k)\eps\omega_k\left[v(\theta_k,r_k)+\eps v_2(\theta_k,r_k)\right]+ \mathcal{O}(\eps^{2\beta+d}),
\end{equation}
for some suitable $d>0$. 

Now turn to the case $n_\beta>\eps^{-2(1-\beta)-\delta}$. The last term in \eqref{eta-version3-IR} can be bounded by:
\begin{equation}\label{errortermsum1-large-IR}
\left|\sum_{k=0}^{n_\beta-1}\mathcal{O}(e^{-\lambda\eps^2k}\eps^{2+a})\right|\leq K\eps^{\textb{1+a}}\sum_{k=0}^{n_\beta-1}
e^{-\lambda\eps^2k}= K\eps^{\textb{1+a}}\frac{1-e^{-\lambda\eps^2n_\beta}}{1-e^{-\lambda\eps^2}} \leq K_\lb\eps^{\textb{a}}, 
\end{equation}
for some positive constants $K$ and $K_\lambda$. Then, using item 1 of Lemma \ref{lemsumaverages-IR} twice (first with $h(r)=f'(r)$ and $g(\theta,r)=E_2(\theta,r)$, and later with $h(r)=f''(r)$ and $g(\theta,r)=v^2(\theta,r)$), we obtain the following bound for $n_\beta> \eps^{-2(1-\beta)-\delta}$:
\begin{equation}\label{eta-version4-large-IR}
 \eta=f(r_0)+\eps\sum_{k=0}^{n_\beta-1}e^{-\lambda\eps^2k}f'(r_k)\omega_k\left[v(\theta_k,r_k)+\eps v_2(\theta_k,r_k)\right]+ \mathcal{O}(\eps^d).
\end{equation} 

To finish the proof, we follow the same steps as in the proof of Lemma \ref{lemmaexpectation} and \eqref{lawtotalexp-v2} yields:
\begin{eqnarray*}
 &&\E(\eta)-f(r_0)\\
  &=&\sum_{n\in\mathbb{N}}\E\left(\sum_{k=0}^{n-1}e^{-\lambda\eps^2k}f'(r_k)\eps\omega_k\left[v(\theta_k,r_k)+\eps v_2(\theta_k,r_k\right])\right)\Prob\{n_\beta=n\}\\
 &+&\mathcal{O}(\eps^{2\beta+d})\Prob\{\eps^{-(1-\beta)+\delta}\leq n_\beta\leq \eps^{-2(1-\beta)-\delta}\}+\mathcal{O}(\eps^d)\Prob\{n_\beta<\eps^{-(1-\beta)+\delta}\}\\
&=&\mathcal{O}(\eps^{2\beta+d}),
\end{eqnarray*}
where by Lemma \ref{main lemma} for some $C>0$ 
independent of $\eps$ 
$$\Prob\{n_\beta>\eps^{-2(1-\beta)-\delta}\}=\exp
\left(-\frac{C}{\eps^{2\delta}}\right),
$$
so that it is smaller than any power of $\eps$.
\end{proof}

\subsection{The Real Rational case}\label{sec:RR-case}

Here we study the system in the RR case, in the subdomain:
$$|r-p/q|\leq K_1\eps^{1/2}.$$

\begin{figure}[h]
  \begin{center}
  \includegraphics[width=10cm]{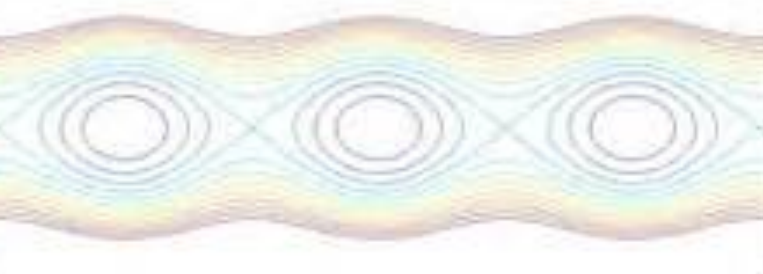}
  \end{center}
  \caption{Level sets of the pendulum $p/q=1/3$.}
  \label{fig:rotor-pendulum}
\end{figure}

From the Normal Form Theorem, in the Real Rational strips the system takes the following form:
\begin{equation}\label{thetarRR-pre}
\beal 
\theta_1&=&\theta_0+r_0+\eps\left[\Eu(\theta_0,p/q)-  \Ev(\theta_0,p/q)+\Ev_{p,q}(\theta_0,p/q)+E_3(\theta_0)\right] \\ 
&+&\eps\omega_0u(\theta_0,p/q)+\mathcal{O}(\eps^{1+a}),
\qquad \qquad \qquad \qquad \qquad \qquad \qquad \qquad 
\\
 r_1&=&r_0+\eps\Ev_{p,q}(\theta_0,r_0) +  
\eps\omega_0v(\theta_0,r_0)\ 
\qquad \qquad \qquad \qquad \qquad \qquad \quad 
 \\
&+&\eps^{3/2}\omega_0v_2(\theta_0,r_0)+
\eps^{3/2}E_4(\theta_0,r_0)+\mathcal{O}(\eps^{2+a}),
\qquad \qquad 
\qquad \qquad \quad 
\enal
\end{equation}
where $v_2(\theta,r)$ can be written explicitly in terms of $v(\theta,r)$ and $S_1(\theta,r)$. The function $E_4$ 
is such that $\|E_4\|_{\mathcal{C}^0}\leq K.$ We point out that it is a rescaled version of the function $E_4$ appearing in the Normal Form Theorem. 

Recall also that:
$$\max \{\|\Eu\|_{\mathcal{C}^0},\ \|\Ev\|_{\mathcal{C}^0},\ 
\|\Ev_{p,q}\|_{\mathcal{C}^0},\ \|E_3\|_{\mathcal{C}^0},\ 
\|u\|_{\mathcal{C}^0},\ \|v\|_{\mathcal{C}^0}\}\ \leq K.
$$
Moreover, we have defined $v_2$ also in such a way that 
$\|v_2\|_{\mathcal{C}^0}\leq K.$

First, we switch to the resonant variable:
$$ 
\hat r=r-p/q,\qquad 0\leq|\hat r|\leq K_1\eps^{1/2}.
$$
With this new variable, system \eqref{thetarRR-pre} writes out as:
{\small 
\begin{equation}\label{thetarRR}
\beal 
 \theta_1&=&\theta_0+p/q+\eps\left[\Eu(\theta_0,p/q)-\Ev(\theta_0,p/q)+\Ev_{p,q}(\theta_0,p/q)+E_3(\theta_0)\right]\\
 &+&\hat r_0+\eps\omega_0u(\theta_0,p/q)
+\mathcal{O}(\eps^{1+a}), \quad \qquad \qquad  
\qquad \qquad \qquad \qquad \qquad  \\
 \hat r_1&=&\hat r_0+\eps\hatEv_{p,q}(\theta_0,\hat r_0)+\eps^{3/2}\hat E_4(\theta_0,\hat r_0)+\eps\omega_0\hat v(\theta_0,\hat r_0)\qquad \qquad \qquad \quad \\
 &+&\eps^{3/2}\omega_0\hat v_2(\theta_0,\hat r_0)+\mathcal{O}(\eps^{2+a}), \qquad \qquad \qquad \qquad \qquad \qquad \qquad 
\qquad 
\enal
\end{equation}}
where:
\begin{equation*}
\begin{array}{rclcrcl}
\hat v(\theta_0,\hat r_0)&=&v(\theta_0,\hat r_0+p/q),&\quad&\hat v_2(\theta_0,\hat r_0)&=&v_2(\theta_0,\hat r_0+p/q),\\
\hatEv_{p,q}(\theta_0,\hat r_0)&=&\Ev_{p,q}(\theta_0,\hat r_0+p/q),&\quad& \hat E_4(\theta_0,\hat r_0)&=&E_4(\theta_0,\hat r_0+p/q),
\end{array}
\end{equation*}
From now on, we will abuse notation and drop all hats. We are interested in the $q-$th iteration of map \eqref{thetarRR}, which is given by:
\begin{equation*}
\begin{array}{rcl}
 \theta_q&=&\displaystyle\theta_0+qr_0+\\
&+&\displaystyle
\eps\sum_{k=0}^{q-1}\left[\Eu(\theta_k,p/q)-\Ev(\theta_k,p/q)+
(q-k)\Ev_{p,q}(\theta_k,p/q)+E_3(\theta_k)\right]\medskip\\
 \displaystyle
&+&\eps\sum_{k=0}^{q-1}\omega_k\left[u(\theta_k, p/q)+v(\theta_k, p/q)\right]+\mathcal{O}(\eps^{1+a}),\bigskip\\
  r_q&=& \displaystyle r_0+\eps\sum_{k=0}^{q-1}\left[\Ev_{p,q}(\theta_k, r_k)+\eps^{1/2} E_4(\theta_k, r_k)\right]\medskip\\
\displaystyle
&+&
\eps\sum_{k=0}^{q-1}\omega_k\left[v(\theta_k, r_k)+\eps^{1/2} v_2(\theta_k, r_k)\right]+\mathcal{O}(\eps^{2+a}).
 \end{array}
\end{equation*}
Taking into account that $q$ is bounded for $0\ge i\leq q$ we have:
$$\theta_i=\theta_0+i(p/q+r_0)+\mathcal{O}(\eps),\qquad r_i=r_0+\mathcal{O}(\eps),$$
we can rewrite the last system as:
\begin{equation}\label{thetarRR-q}
\beal
 \theta_q&=&\displaystyle\theta_0+qr_0+
\eps\Eu^{(q)}(\theta_0)+\eps u^{(q)}(\theta_0,\omega_0^q)
+\mathcal{O}(\eps^{1+a}),\qquad \qquad \qquad \qquad \bigskip\\
  r_q&=& \displaystyle r_0+\eps \Ev^{(q)}(\theta_0,r_0,\eps)+\eps v^{(q)}(\theta_0,r_0,\omega_0^q)+\eps^{3/2} v_2^{(q)}(\theta_0,r_0,\omega_0^q)+\mathcal{O}(\eps^2).
\enal
\end{equation}
where $\omega_k^q=(\omega_{qk},\dots,\omega_{qk+q-1})$ and:
\begin{eqnarray*}
\Eu^{(q)}(\theta)&=&\sum_{i=0}^{q-1}\left[\Eu(\theta+ip/q,p/q)-
\Ev(\theta+ip/q,p/q)+\right.\\
&+&\left.(q-i)\Ev_{p,q}(\theta+ip/q,p/q)+E_3(\theta+ip/q)\right],\\
 u^{(q)}(\theta,\omega_k^q)&=&
\sum_{i=0}^{q-1}(q-i)\omega_{qk+i}\left[u(\theta+ip/q,p/q)
+v(\theta+ip/q,p/q)\right],\\
 \Ev^{(q)}(\theta,r,\eps)&=&\sum_{i=0}^{q-1}\left[\Ev_{p,q}(\theta+i(p/q+r),r)+\eps^{1/2} E_4(\theta+i(p/q+r),r)\right],\\
 v^{(q)}(\theta,r,\omega_k^q)&=&\sum_{i=0}^{q-1}
\omega_{qk+i}v(\theta+i(p/q+r),r),\\
 v_2^{(q)}(\theta,r,\omega_k^q)&=&\sum_{i=0}^{q-1}
\omega_{qk+i}v_2(\theta+i(p/q+r),r).
\end{eqnarray*}

Introduce a rescaled variable $r=R\sqrt \eps$. Then $(\theta_1,R_1)$ are defined using system \eqref{thetarRR-q} with the corresponding rescaling. This can be rewritten in the following way, where we just keep the necessary $\eps$-dependent terms:
\begin{equation}\label{thetarRR-q-Rvar}
\begin{array}{rcl}
 \theta_q&=&\displaystyle\theta_0+qR_0\sqrt{\eps}+
\eps\Eu^{(q)}(\theta_0)+\eps u^{(q)}(\theta_0,\omega_0^q)
+\mathcal{O}(\eps^{1+a}),\bigskip\\
  R_q&=& \displaystyle R_0+\eps^{1/2} \Ev^{(q)}(\theta_0,R_0\sqrt{\eps},0)+\eps^{1/2} v^{(q)}(\theta_0,R_0\sqrt{\eps},\omega_0^q)
  \\&&+\eps v_2^{(q)}(\theta_0,0,\omega_0^q)+\mathcal{O}(\eps^{3/2}).
 \end{array}
\end{equation}
Consider the Hamiltonian:
\be \label{eq:RR-Hamiltonian}
H(\theta,R)=\frac{R^2}{2}-
\frac{1}{q}\int_0^\theta\Ev^{(q)}(s,R \sqrt \eps,0)ds.
\ee
Let $H_n:=H(\theta_{qn},R_{qn})$. We study the process:
$(\theta_{qn},H_n):=(\theta_n, H(\theta_{qn},R_{qn}),$
where $(\theta_{qn},R_{qn})$ is the process obtained 
iterating \eqref{thetarRR-q-Rvar} $n$ times. One can see that:
\begin{equation}\label{defH1-RR-rescaled}
\beal 
H_1=H_0&+&\sqrt \eps R_0\,
v^{(q)}(\theta_0,R_0\sqrt \eps,\omega_0^q)
\qquad \qquad \qquad \qquad \qquad \quad \ 
\\
&+&\eps F(\theta_0,R_0,\eps)+
\eps G(\theta_0,R_0,\omega_0^q,\eps)
+\mathcal{O}(\eps^{1+a}),\qquad 
\enal 
\end{equation}
where $F$ and $G$ are:
\begin{eqnarray*}
F(\theta,R)&=&
\frac{1}{q}\Ev^{(q)}(\theta,0,0)\ 
\Eu^{(q)}(\theta)\\
&-&
\frac{1}{q}\Ev^{(q)}(\theta,0,0)
\int_0^{\theta}\Ev^{(q)}(s,0,0)ds\\
&+&\frac{q}{2} R^2 
\partial_\theta\Ev^{(q)}(\theta,0,0)+
\frac{1}{2}\left(\Ev^{(q)}(\theta,0,0)\right)^2\\
&+&\frac{1}{2}\sum_{i=0}^{q-1}
v^2(\theta+ip/q,0),\\
G(\theta,R,\omega_k^q)&=&
\frac{1}{q}\Ev^{(q)}(\theta,0,0)\ 
u^{(q)}(\theta,\omega_k^q)\\
&-&\frac{1}{q}v^{(q)}(\theta,0,\omega_k^q,0)\int_0^{\theta}\partial_r
\Ev^{(q)}(s,0,0)ds\\
+\frac{1}{2}\sum_{\substack{i,j=0\\i\neq j}}^{q-1}
\omega_{qk+i}&\omega_{qk+j}&
v(\theta+ip/q,0)
v(\theta+jp/q,0)\\
&+&
\Ev^{(q)}(\theta,0,0)
v^{(q)}(\theta,0,\omega_k^q)\\
&+&
R\sum_{i=0}^{q-1}
\omega_{qk+i}v_2(\theta+ip/q),0).
\end{eqnarray*}
We note that since $|R|\leq K_1$ we have:
$$\|F\|_{\mathcal{C}^0}\leq K,\qquad \|G\|_{\mathcal{C}^0}\leq K.$$
Moreover, one has that for all $k\ge0$:
$$
\E(G(\theta_{qk},R_{qk},\omega_k^q))=0.
$$

In terms of the variable $H$, the $\{|R|\leq K_1\}$ 
region can be written as:
\begin{equation}\label{defHstrips-RR}
 I_{RR}:=\{H\in\mathbb{R}\,:\,|H|\leq K_3\},
\end{equation}
for some constant $K_3$. Denote by $n^*$ 
the stopping time of leaving $I_{RR}$ for the first time.  

\begin{lem}\label{lemmaexpectation-RR}
Let $f:\mathbb{R}\rightarrow\mathbb{R}$ be any $\mathcal{C}^l$ function with $l\ge 3$. Then 
there exists $d>2/5$ such that for all $\lambda>0$ one has:
\begin{eqnarray*}
 &&\E\left(e^{-\lambda\eps n^*}f(H_{n^*})+ \right. \\
&& \left. 
\eps\sum_{k=0}^{n^*-1}e^{-\lambda \eps k}
\left[\lambda f(H_k)-b(\theta_{qk},R_{qk})f'( H_k) -\frac{\sigma^2(\theta_{qk},R_{qk})}{2}f''(H_k)\right]\right)\\
&&\qquad \qquad \qquad \qquad \qquad \qquad 
\qquad \qquad 
- f(H_0)=\mathcal{O}(\eps^d),
\end{eqnarray*}
where:
\be \label{eq:RR-drift-variance}
b(\theta,R)=F(\theta,R),\quad 
\sigma^2(\theta,R)= R_{qk}^2\sum_{i=0}^{q-1}
v^2(\theta+ip/q,0).\ee
\end{lem}
\begin{proof} Depending on properties of the averaged 
potential $\Ev^{(q)}(\theta,R \sqrt \eps,0)$ 
we decompose the region $\{|H|\le K_1\}$ 
into several domains as follows. 

In the interval $[-K_1,K_1]$ there are finitely many 
critical value of $H$ due to the fact that $\E v^{(q)}$ is 
a trigonometric polynomial in $\theta$. By assumption 
[{\bf H5}] all critical points are nondegenerate. Therefore, 
there are finitely many of them.  Consider 
a $\sqrt \eps$-neighborhood of each one of them. 

There are level sets that are not intersecting 
these neighbourhoods of the critical points. Call 
these level sets {\it regular}. The others levels intersect 
these neighbourhoods. Call these level sets 
{\it nearly critical}. Notice that the regular level sets are 
no necessarily connected (see Fig. \ref{fig:potential-graph}). 

Each connected component of a family of ovals gives rise 
to a segment and the union of such segments gives rise to 
a diffusion on a graph as in Freidlin-Wentzell \cite{FW}, Sect. 8. 

Now we study separately three regimes: 
\begin{itemize}
\item a $\sqrt \eps$-near a critical point;

\item iterates with an initial condition on 
a near critical level set;  

\item iterates with an initial condition on 
a regular level set;
\end{itemize}

In the first case note that $n^*<\infty$ almost surely. 
Indeed, on the one hand if $(\theta^*,0)$ is a critical point of 
the dominant part of the deterministic system, i.e. $\Eu^{(q)}(\theta^*,0,0)=\Ev^{(q)}(\theta^*,0,0)=0$, then the drift of the process is given by:
$$b(\theta^*,0)=\frac{1}{2}
\sum_{i=0}^{q-1}v^2(\theta^*+ip/q,0)>0.
$$
On the other hand, if $(\theta,0)$ is not a critical point, 
then the drift of the $R$--component of the process $(\theta_{qn},R_{qn})$ defined in \eqref{thetarRR-q-Rvar} 
is precisely, which is $\Ev^{(q)}(\theta,0,0)\neq0$. In conclusion, 
the process does not get stuck at $R=0$. Since away from 
this zone the diffusion coefficient $\sigma(\theta,R)$ is nonzero, 
$n^*$ must be finite with probability one.

In the second case denote:
\begin{equation}\label{defeta-RR}
\beal 
\eta=e^{-\lambda\eps n^*}f(H_{n^*})+ \qquad \qquad 
\qquad \qquad \qquad \qquad \qquad \qquad \qquad \qquad 
\\
\eps\sum_{k=0}^{n^*-1}e^{-\lambda\eps k}\left[\lambda f(H_k)-b(\theta_{qk},R_{qk})f'(H_k)-\frac{\sigma^2(\theta_{qk},R_{qk})}{2}f''(H_k)\right].
\enal 
\end{equation}
Now we write:{\small 
\begin{eqnarray*}
e^{-\lambda\eps n^*}f(H_{n^*})&=&f(H_0)+\sum_{k=0}^{n^*-1}\left[-\lambda e^{-\lambda\eps k}f(H_k)+e^{-\lambda\eps k}f'(H_k)(H_{k+1}-H_k)\right.\\
&&\left.+\frac{1}{2}e^{-\lambda\eps k}f''(H_k)(H_{k+1}-H_k)^2+\mathcal{O}(e^{-\lambda\eps k}\eps^{3/2})\right],
\end{eqnarray*}}
so that \eqref{defeta-RR} writes out as:
{\small 
\begin{eqnarray}\label{eta-version2-RR}
\beal \eta=&&f(H_0)+\qquad \qquad\qquad\qquad
\qquad \qquad \qquad \qquad\qquad\qquad
\qquad \qquad \\
\sum_{k=0}^{n^*-1}&&\left[e^{-\lambda\eps k}f'(H_k)(H_{k+1}-H_k)+\frac{1}{2}e^{-\lambda\eps k}f''(H_k)(H_{k+1}-H_k)^2\right]\qquad \qquad\\
 -\eps\sum_{k=0}^{n^*-1} && e^{-\lambda\eps k}\left[b(\theta_{qk},R_{qk})f'(H_k)+\frac{\sigma^2(\theta_{qk},R_{qk})}{2}f''(H_k)\right]+\sum_{k=0}^{n^*-1}\mathcal{O}(e^{-\lambda\eps k}\eps^{3/2}).\\
\enal 
\end{eqnarray}}
Now, using \eqref{defH1-RR-rescaled} it is clear that:
\begin{equation}\label{Hkdiff-RR}
\beal 
 H_{k+1}-H_k=\sqrt{\eps}R_{qk}v^{(q)}(\theta_{qk},R_{qk}\sqrt{\eps},\omega_k^q)+\qquad\qquad\qquad\qquad \\ 
\eps F(\theta_{qk},R_{qk})+
\eps G(\theta_{qk},R_{qk},\omega_k^q)
+\mathcal{O}(\eps^{1+a}).
\enal 
\end{equation}
Moreover, we have:
{\small 
\begin{eqnarray}\label{Hkdiff2-RR}
\beal 
&&(H_{k+1}-H_k)^2=\eps R_{qk}^2(v^{(q)}(\theta_{qk},0,\omega_k^q))^2+\mathcal{O}(\eps^{3/2})\qquad \qquad \qquad \\
&&=\eps R_{qk}^2\sum_{i=0}^{q-1}
v^2(\theta_{qk}+ip/q,0)+\eps G_0(\theta_{qk},R_{qk},\omega_k^q)+\mathcal{O}(\eps^{3/2}),\\
\enal 
\end{eqnarray}}
where:
{\small\begin{eqnarray*}
&&G_0(\theta_{qk},R_{qk},\omega_k^q)\\
&&=2R_{qk}^2\sum_{l=0}^{q-1}\sum_{j=l+1}^{q-1}\omega_{qk+l}\omega_{qk+j}v(\theta+lp/q,0)v(\theta+jp/q,0).
\end{eqnarray*}}
We note that $\E(G_0(\theta_{qk},R_{qk},\omega_k^q))=0$. 

Using \eqref{Hkdiff-RR} and \eqref{Hkdiff2-RR}, equation \eqref{eta-version2-RR} writes out as:
{\small 
\begin{eqnarray}\label{eta-version3-RR}
\eta&=&f(H_0)+\sqrt{\eps}\sum_{k=0}^{n^*-1}e^{-\lambda\eps k}f'(H_k)R_{qk}v^{(q)}(\theta_{qk},\sqrt{\eps}R_{qk},\omega_k^q)\qquad \qquad \qquad \qquad \qquad \qquad \qquad  \nonumber\\
 &&+\eps\sum_{k=0}^{n^*-1}e^{-\lambda\eps k}f'(H_k)\left[F(\theta_{qk},R_{qk})-b(\theta_{qk},R_{qk})\right]
\qquad \qquad \qquad \qquad \qquad \qquad \qquad  \nonumber\\
 &&+\frac{\eps}{2}\sum_{k=0}^{n^*-1}e^{-\lambda\eps k}f''(H_k)\left[R_{qk}^2
\sum_{i=0}^{q-1}v^2(\theta_{qk}+ip/q,0)-
\sigma^2(\theta_{qk},R_{qk})\right]\qquad \nonumber\\
 &&+\eps\sum_{k=0}^{n^*-1}e^{-\lambda\eps k}\left[f'(H_k)G(\theta_{qk},R_{qk},\omega_k^q)+\frac{f''(H_k)}{2}G_0(\theta_{qk},R_{qk},\omega_k^q)\right]\nonumber\\
 &&+\sum_{k=0}^{n^*-1}\mathcal{O}(e^{-\lambda\eps k}\eps^{1+a}). \nonumber
\end{eqnarray}}
Now, on the one hand by definition of $b(\theta,R)$ and $\sigma^2(\theta,R)$ it is clear that:
\begin{eqnarray}\label{sigma0-RR}
\beal 
F(\theta_{qk},R_{qk})-b(\theta_{qk},R_{qk})&=&0\label{b0-RR}\\
R_{qk}^2\sum_{i=0}^{q-1}v^2(\theta_{qk}+ip/q,0)-\sigma^2(\theta_{qk},R_{qk})&=&0.\enal 
\end{eqnarray} 
On the other hand, the last term in \eqref{eta-version3-RR} can be bounded by:
\begin{eqnarray}\label{errortermsum1-RR}
\left|\sum_{k=0}^{n^*-1}
\mathcal{O}(e^{-\lambda\eps k}\eps^{1+a})\right|&\leq& K\eps^{1+a}\sum_{k=0}^{n^*-1}e^{-\lambda\eps k}= K\eps^{1+a}\frac{1-e^{-\lambda\eps n^*}}{1-e^{-\lambda\eps}}\nonumber\\
&\leq& K_\lb\eps^{a}, 
\end{eqnarray}
for some positive constants $K$ and $K_\lambda$.
Using \eqref{b0-RR}, \eqref{sigma0-RR} and 
\eqref{errortermsum1-RR} in equation \eqref{eta-version3-RR} 
we have:
\begin{eqnarray}\label{eta-version4-RR} 
\eta&=&f(H_0)+
\sqrt{\eps}\sum_{k=0}^{n^*-1}e^{-\lambda\eps k}
f'(H_k)R_{qk}v^{(q)}(\theta_{qk},\sqrt{\eps}R_{qk},\omega_k^q)\\
 &+&\eps\sum_{k=0}^{n^*-1}
e^{-\lambda\eps k}\left[f'(H_k)G(\theta_{qk},R_{qk},\omega_k^q)+\frac{f''(H_k)}{2}G_0(\theta_{qk},R_{qk},\omega_k^q)\right] +\mathcal{O}(\eps^a).\nonumber
\end{eqnarray}
Thus, to finish the proof, we just need to use that:
$$\E\left(v^{(q)}(\theta_{qk},\sqrt{\eps}R_{qk},\omega_k^q)\right)=\E\left(G_0(\theta_{qk},R_{qk},\omega_k^q)\right)=\E\left(G(\theta_{qk},R_{qk},\omega_k^q)\right)=0,$$
and that $R_{qk}$ and $H_k$ are independent of $\omega_k^q$. Using these facts in \eqref{eta-version4-RR}, one obtains straightforwardly:
$$ \E(\eta)-f(H_0)=\mathcal{O}(\eps^a),$$
and the proof is finished.
\end{proof}

\subsection{Transition Zones}\label{sec:TZ-case}

Here we study the system in the RR case, in the subdomain:
$$K_1\eps^{1/2}\leq|r-p/q|\leq K_2\eps^{1/6},$$
for certain constants $K_1$ and $K_2$. From the Normal Form Theorem, in the Real Rational strips the system takes the following form:
\begin{equation}\label{thetarTZ-pre}
\begin{array}{rcl}
 \theta_1&=&\theta_0+r_0+\mathcal{O}(\eps),\\
 r_1&=&r_0+\eps\Ev_{p,q}(\theta_0,r_0)+\eps\omega_0v(\theta_0,r_0)+\mathcal{O}(\eps^{3/2}).
 \end{array}
\end{equation}
Recall that:
$$
\|\Ev_{p,q}\|_{\mathcal{C}^0}\leq K,\qquad\|v\|_{\mathcal{C}^0}\leq K.$$

For our purposes, it will be more convenient to work with the variable:
$$ \hat r=r-p/q,\qquad K_1\eps^{1/2}\leq|\hat r|\leq K_2\eps^{1/6}.$$
With this new variable, system \eqref{thetarTZ-pre} writes out as:
\begin{equation}\label{thetarTZ}
\begin{array}{rcl}
 \theta_1&=&\theta_0+p/q+\hat r_0+\mathcal{O}(\eps),\\
 \hat r_1&=&\hat r_0+\eps\hatEv_{p,q}(\theta_0,\hat r_0)+\eps\omega_0\hat v(\theta_0,\hat r_0)+\mathcal{O}(\eps^{3/2}),
 \end{array}
\end{equation}
where:
\begin{equation*}
\begin{array}{rcl}
\hat v(\theta_0,\hat r_0)=v(\theta_0,\hat r_0+p/q),\qquad 
\hatEv_{p,q}(\theta_0,\hat r_0)=\Ev_{p,q}(\theta_0,\hat r_0+p/q).
\end{array}
\end{equation*}
From now on, we will abuse notation and drop all hats. We are interested in the $q-$th iteration of map \eqref{thetarTZ}, which is given by:
\begin{equation*}
\begin{array}{rcl}
 \theta_q&=&\displaystyle\theta_0+qr_0+\mathcal{O}(\eps),\medskip\\
  r_q&=& \displaystyle r_0+\eps\sum_{k=0}^{q-1}\Ev_{p,q}(\theta_k, r_k)+\eps\sum_{k=0}^{q-1}\omega_kv(\theta_k, r_k)+\mathcal{O}(\eps^{3/2}).
 \end{array}
\end{equation*}
Note that we have used that $|q|$ is bounded. Moreover, taking into account that for $i\leq q$, we have:
$$\theta_i=\theta_0+i(p/q+r_0)+\mathcal{O}(\eps),\qquad r_i=r_0+\mathcal{O}(\eps),$$
so that we can rewrite the last system as:
\begin{equation}\label{thetarTZ-q}
\begin{array}{rcl}
 \theta_q&=&\displaystyle\theta_0+qr_0+\mathcal{O}(\eps),\bigskip\\
  r_q&=& \displaystyle r_0+\eps \Ev^{(q)}(\theta_0,r_0)+\eps v^{(q)}(\theta_0,r_0,\omega_0^q)+\mathcal{O}(\eps^{3/2}),
 \end{array}
\end{equation}
where we use the notation $\omega_k^q=(\omega_{kq},\dots,\omega_{kq+q-1})$ and:
\begin{eqnarray*}
 \Ev^{(q)}(\theta,r)&=&\sum_{i=0}^{q-1}\Ev_{p,q}(\theta+i(p/q+r),r),\\
 v^{(q)}(\theta,r,\omega_k^q)&=&\sum_{i=0}^{q-1}\omega_{kq+i}v(\theta+i(p/q+r),r).
\end{eqnarray*}
We point out that for any $k\ge0$:
$$\E\left(v^{(q)}(\theta_{qk},r_{qk},\omega_k^q)\right)=0.$$

Now let us consider the following function:
$$H(\theta,r)=\frac{r^2}{2}-\frac{\eps}{q}\int_0^\theta\Ev^{(q)}(s,r)ds.$$
In this section we will study the process:
$$H_n:=H(\theta_{qn},r_{qn}),$$
where $(\theta_{qn},r_{qn})$ is the process obtained iterating \eqref{thetarTZ-q} $n$ times. One can easily see that:
\begin{eqnarray}\label{defH1}
H_1&=&H_0+\eps r_0v^{(q)}(\theta_0,r_0,\omega_0^q)+\mathcal{O}(\eps r_0^2).
\end{eqnarray}

Now, given a constant $0<\rho<1/6$, we want to study the process $H_n$ in the following $H-$strips:
\begin{equation}\label{defHstrips}
 I_\rho(r_0)=\{H\in\mathbb{R}\,:\,|H-H_0|\leq |r_0|\eps^{1-\rho}\}.
\end{equation}
We stress out that the width of these strips depends on the initial condition $r_0$ and $\eps$. To avoid this, we can define the process:
$$\bar H_n:=|r_0|^{-1}\eps^{-1+\rho}H_n.$$
The process $\bar H_n$ is defined through:
\begin{eqnarray}\label{defH1-rescaled}
\bar H_1&=&\bar H_0+\eps^{\rho}v^{(q)}(\theta_0,r_0,\omega_0^q)+\mathcal{O}(\eps^{\rho}|r_0|).
\end{eqnarray}
Clearly, if we denote by $n_\rho$ the first exit time of the process $H_n$ from the strip $ I_\rho(r_0)$, it is also the first exit time of the process $\bar H_n$ of the strip:
\begin{equation}\label{defHstrips-rescaled}
 I=\{H\in\mathbb{R}\,:\,|H-\bar H_0|\leq 1\}.
\end{equation}

\begin{lem}\label{lemmaexpectation-TZ}
Let $f:\mathbb{R}\rightarrow\mathbb{R}$ be any $\mathcal{C}^l$ function with $l\ge3$. Then there exists $d>0$ such that for all $\lambda>2/5$ one has:
\begin{eqnarray*}
 &&\E\left(e^{-\lambda\eps^{2\rho}n_\rho}f(\bar H_{n_\rho})+\eps^{2\rho}
\sum_{k=0}^{n_\rho-1}e^{-\lambda\eps^{2\rho}k}\left[\lambda f(\bar H_k)-\frac{\sigma^2(\theta_{qk},r_{qk})}{2}f''(\bar H_k)\right]\right)\\
&&- f(\bar H_0)=\mathcal{O}(\eps^d),
\end{eqnarray*}
where:
$$\sigma^2(\theta,r)=\sum_{i=0}^{q-1}v^2(\theta+i(p/q+r),r).$$
\end{lem}

\begin{proof}
Let us denote:
\begin{equation}\label{defeta-TZ}
\eta=e^{-\lambda\eps^{2\rho}n_\rho}f(r_{n_\rho})+\eps^{2\rho}
\sum_{k=0}^{n_\rho-1}e^{-\lambda\eps^{2\rho}k}\left[\lambda f(\bar H_k)-\frac{\sigma^2(\theta_{qk},r_{qk})}{2}f''(\bar H_k)\right].
\end{equation}
%
Now we write:
\begin{eqnarray*}
e^{-\lambda\eps^{2\rho}n_\rho}f(\bar H_{n_\rho})&=&
\\
f(\bar H_0)&+&\sum_{k=0}^{n_\rho-1}\left[-\lambda\eps^{2\rho} e^{-\lambda\eps^{2\rho}k}f(\bar H_k)+e^{-\lambda\eps^{2\rho}k}f'(\bar H_k)(\bar H_{k+1}-\bar H_k)\right.\\
&&\left.+\frac{1}{2}e^{-\lambda\eps^{2\rho}k}f''(\bar H_k)(\bar H_{k+1}-\bar H_k)^2+\mathcal{O}(e^{-\lambda\eps^2k}\eps^{3\rho})\right],
\end{eqnarray*}
so that \eqref{defeta-TZ} writes out as:
\begin{eqnarray}\label{eta-version2-TZ}
 \eta&=&f(\bar H_0)+  \\
\sum_{k=0}^{n_\rho-1} &&
\left[e^{-\lambda\eps^{2\rho}k}f'(\bar H_k)(\bar H_{k+1}-\bar H_k)+\frac{1}{2}e^{-\lambda\eps^{2\rho}k}f''(\bar H_k)(\bar H_{k+1}-\bar H_k)^2\right]\nonumber\\
 &&-\eps^{2\rho}\sum_{k=0}^{n_\rho-1}e^{-\lambda\eps^{2\rho}k}\frac{\sigma^2(\theta_{qk},r_{qk})}{2}f''(\bar H_k)+\sum_{k=0}^{n_\rho-1}\mathcal{O}(e^{-\lambda\eps^2k}\eps^{3\rho}).\nonumber
\end{eqnarray}
Now, using \eqref{defH1-rescaled} it is clear that:
\begin{equation}\label{Hkdiff}
 \bar H_{k+1}-\bar H_k=\eps^{\rho}v^{(q)}(\theta_{qk},r_{qk},\omega_k^q)+\mathcal{O}(\eps^\rho |r_{qk}|).
\end{equation}
Moreover, we have:
\begin{eqnarray}\label{Hkdiff2}
(\bar H_{k+1}-\bar H_k)^2&=&\eps^{2\rho}(v^{(q)}(\theta_{qk},r_{qk},\omega_k^q))^2+\mathcal{O}(\eps^{2\rho} |r_{qk}|)\nonumber\\
&=&\eps^{2\rho}\sum_{i=0}^{q-1}v^2(\theta_{qk}+i(p/q+r_{qk}),r_{qk})\nonumber\\
&&+\eps^{2\rho}G_0(\theta_{qk},r_{qk},\omega_k^q)+\mathcal{O}(\eps^{2\rho}|r_{qk}|)
\end{eqnarray}
where:
$$G_0(\theta,r,\omega_k^q)=2\sum_{l=0}^{q-1}\sum_{j=l+1}^{q-1}\omega_{kq+l}\omega_{kq+j}v(\theta+l(p/q+r),r)v(\theta+j(p/q+r),r).$$
We note that $\E(G_0(\theta_{qk},r_{qk},\omega_k^q))=0$. 

Using \eqref{Hkdiff} and \eqref{Hkdiff2} and noting that $\eps^{2\rho}|r_{qk}|\leq\eps^\rho |r_{qk}|$, equation \eqref{eta-version2-TZ} writes out as:
\begin{eqnarray}\label{eta-version3-TZ}
 \eta&=&f(\bar H_0)+\eps^\rho\sum_{k=0}^{n_\rho-1}e^{-\lambda\eps^{2\rho}k}f'(\bar H_k)v^{(q)}(\theta_{qk},r_{qk},\omega_k^q)\nonumber\\
 &&+\frac{\eps^{2\rho}}{2}\sum_{k=0}^{n_\rho-1}e^{-\lambda\eps^{2\rho}k}f''(\bar H_k)\left[\sum_{i=0}^{q-1}v^2(\theta_{qk}+i(p/q+r_{qk}),r_{qk})-\sigma^2(\theta_{qk},r_{qk})\right]\nonumber\\
 &&+\frac{\eps^{2\rho}}{2}\sum_{k=0}^{n_\rho-1}e^{-\lambda\eps^{2\rho}k}f''(\bar H_k)G_0(\theta_{qk},r_{qk},\omega_k^q)+\sum_{k=0}^{n_\rho-1}\mathcal{O}(e^{-\lambda\eps^{2\rho}k}\eps^{\rho}|r_{qk}|).
\end{eqnarray}
Now, on the one hand by definition of $\sigma^2(\theta,r)$ it is clear that:
\begin{eqnarray}
\sum_{i=0}^{q-1}v^2(\theta_{qk}+i(p/q+r_{qk}),r_{qk})-\sigma^2(\theta_{qk},r_{qk})=0.\label{sigma0}
\end{eqnarray} 
On the other hand, the last term in \eqref{eta-version3-TZ} can be bounded by:
\begin{eqnarray}\label{errortermsum1-TZ}
\left|\sum_{k=0}^{n_\rho-1}\mathcal{O}(e^{-\lambda\eps^{2\rho}k}\eps^{\rho}|r_{qk}|)\right|&\leq& K\eps^{\rho+1/6}\sum_{k=0}^{n_\rho-1}e^{-\lambda\eps^{2\rho}k}= K\eps^{\rho+1/6}\frac{1-e^{-\lambda\eps^{2\rho}n_\rho}}{1-e^{-\lambda\eps^{2\rho}}}\nonumber\\
&\leq& K_\lb\eps^{1/6-\rho}, 
\end{eqnarray}
for some positive constants $K$ and $K_\lambda$, where we have used that $|r_{qk}|\leq\eps^{1/6}$ if $0\leq qk\leq n_\rho$. 
Using \eqref{sigma0} and \eqref{errortermsum1-TZ} in equation \eqref{eta-version3-TZ}, and denoting $d=1/6-\rho>0$, we have:
\begin{eqnarray}\label{eta-version4-TZ}
 \eta&=&f(\bar H_0)+\eps^\rho\sum_{k=0}^{n_\rho-1}e^{-\lambda\eps^{2\rho}k}f'(\bar H_k)v^{(q)}(\theta_{qk},r_{qk},\omega_k^q)\nonumber\\
 &&+\frac{\eps^{2\rho}}{2}\sum_{k=0}^{n_\rho-1}e^{-\lambda\eps^{2\rho}k}f''(\bar H_k)G_0(\theta_{qk},r_{qk},\omega_k^q)+\mathcal{O}(\eps^d).
\end{eqnarray}
Thus, to finish the proof, we just need to use that:
$$\E\left(v^{(q)}(\theta_{qk},r_{qk},\omega_k^q)\right)=\E\left(G_0(\theta_{qk},r_{qk},\omega_k^q)\right)=0.$$
Indeed, using this fact in \eqref{eta-version4-TZ}, one obtains straightforwardly:
$$ \E(\eta)-f(\bar H_0)=\mathcal{O}(\eps^d),$$
and the proof is finished.
\end{proof}

\appendix

\section{Measure of the domain covered by RR and IR intervals}\label{sec:measure-IR-RR}
In this section we show that, with the right choice of $b$, the measure of the the union of all strips of RR and IR type inside any compact set:
$$
A_\beta=\cup_k I_\beta^k\subset\mathbb{T}\times B\qquad I_\beta^k\ \textrm{ strips of width }\ 2\eps^\beta
$$
goes to zero as $\eps\to 0$. 

In fact, we will do the proof for $A=[0,1]$. The general case is completely analogous. Let us consider:
\begin{equation}\nonumber
 \mathcal{R}=\{p/q\in\mathbb{Q}\,:\,p<q,\,\gcd(p,q)=1,\,q<\eps^{-b}\}=\cup_{q=1}^{q_\textrm{max}}\mathcal{R}_q \subset[0,1],
\end{equation}
where $q_{\textrm{max}}=[\eps^{-b}]$ and:
\begin{equation}\nonumber
 \mathcal{R}_q=\{p/q\in\mathbb{Q}\,:\,p<q,\,\gcd(p,q)=1\}.
\end{equation}
Finally we denote:
\begin{equation}\nonumber
 I_\mathcal{R}=\{I_\beta^k\subset[0,1]\,:\, \exists p/q\in\mathcal{R}\cap I_\beta\}.
\end{equation}

\begin{lem}\label{lem:im-rational}
Let $\rho$ be fixed, $0<\rho<\beta$, and define $b=(\beta-\rho)/2$. 
Then, for each $I_\beta$ such that there is at most one rational 
$p/q$ satisfying $|q|\leq\eps^{-b}$ the union $I_{\mathcal R}$
has the Lebesgue measure $\mu(I_\mathcal{R}) \le \eps^\rho$
and, therefore,  as $\eps \to 0$:
$$
\mu(I_\mathcal{R})\to 0,
$$
where $\mu$ denotes the Lebesgue measure.
\end{lem}
\begin{proof}
On the one hand, suppose that $p/q\in I_\beta$, $q\leq\eps^{-b}$. Then, for all $p'/q'\in I_\beta$, with $p'$ and $q'$ relatively prime and $p'/q'\neq p/q$, we have:
$$\eps^\beta\geq|p/q-p'/q'|\geq\frac{1}{qq'}\geq\frac{\eps^b}{q'}.$$
Therefore:
$$q'\geq\eps^{-\beta+b}=\eps^{-b-3\rho/2}>\eps^{-b},$$
so the first part of the claim is proved.

On the other hand we note that, if $q_1\neq q_2$, then $\mathcal{R}_{q_1}\cap\mathcal{R}_{q_2}=\emptyset$. Moreover, it is clear that $\#\mathcal{R}_q\leq q-1$ (and if $q$ is prime then $\#\mathcal{R}_q=q-1$, so that the bound is optimal). Therefore we have:
$$\#\mathcal{R}\leq\sum_{q=1}^{q_\textrm{max}}\#\mathcal{R}_q\leq\sum_{q=1}^{q_\textrm{max}}q-1=\frac{q_\textrm{max}^2}{2}<\eps^{-2b}.$$
Since $\mu(I_\beta)=\eps^{\beta}$, one has:
$$0\leq\mu(I_\mathcal{R})=\eps^{\beta}\#\mathcal{R}<\eps^\beta\eps^{-2b}=\eps^{\rho},$$ 
so that the second claim of the lemma is also clear.
\end{proof}

\section{Sufficient condition for weak convergence
and auxiliary lemmas}
\label{sec:auxiliaries}

In order to prove that the $r$-component exhibits a diffusion 
process we need to adapt several lemmas from Ch. 8 
sec. 3 \cite{FW}. We recall some terminology and 
notations (see Ch. 1 sec. 1 \cite{FW} for more details).

In the notations of section \ref{sec:diffusion-generators}
we have 

\begin{lem} \label{lem:FW-suff-cond}
(see Lm. 3.1, \cite{FW}) Let $M$ be a metric space, 
$Y$ a continuous mapping $M \mapsto Y(M)$,  
$Y(M)$ being a complete separable metric space. 
Let $(X^\eps_t,P^\eps_x)$ be a family of Markov 
processes in $M$; suppose that the process 
$Y (X^\eps_t)$ has continuous trajectories. 
Let $(y_t, P_y$) be a Markov process with 
continuous paths in $Y (M)$ whose infinitesimal 
operator is $A$ with domain of definition $D_A$. 
Suppose that the space $C[0,\infty)$ of 
continuous functions on $[0,\infty)$ with values 
in $\Gamma$ is taken as the sample space, so 
that the distribution of the process in the space 
of continuous functions is simply $P_y$. Let 
$\Psi$ be a subset of the space $C(Y (M))$ 
such that for measures $\mu_1,\ \mu_2$ on 
$Y (M)$ the equality $\int f d\mu_1 =\int f d\mu_2$ for all 
$f\in \Psi$ implies $\mu_1 = \mu_2$. Let $D$ be a subset of 
$D_A$ such that for every $f\in \Psi$ and 
$\lb>0$ the equation $\lb F − AF = f$ has 
a solution $F \in D$.

Suppose that for every $x \in M$ the family of distributions 
$Q^\eps_x$ of $Y (X^\eps_\bullet)$ in the space $C[0,\infty)$ 
corresponding to the probabilities $P^\eps_x$ for all $\eps$ 
is tight; and that for every compact $K \subset Y (M)$, 
for every $f \in D$ and every $\lb > 0$,
\[
\E^\eps_x \int_0^\infty 
\exp(-\lb t)\left[ \lb f(Y(X^\eps_t)) -A f(Y(X_t^\eps))
\right]dt \to f(Y(x))
\]
as $\eps \to 0$ uniformly in $x\in Y^{-1}(K)$.

Then $Q^\eps_x$ converges weakly as $\eps \to 0$
to the probability measure $P_{Y(x)}$.  
\end{lem}

In our case $Y(M)$ is the real line. 
We use a discrete version of this lemma in our proof. 

Similarly, to Lemma 3.2 \cite{FW} one can show that 
the family of distributions $Q^\eps_x$ (those of 
$Y (X^\eps_\bullet$) with respect to the probability 
measures $P^\eps_x$ in the space $C[0,\infty)$) with 
small nonzero $\eps$ is tight. Indeed,  in our case
speed of change of $I$ is bounded. Denote 
$H(X)=H(r,\theta)=r^2/2$. Then  
\begin{itemize}
\item for every $T>0$ and $\delta>0$ there exists $H_0$ 
such that 
\[
\mathbb P_x^\eps\{\max_{0<t<T} |H(X^\eps_t)|>H_0\}<
\delta. 
\] 

\item for every compact subset $K  \subset \A$ and 
for every sufficiently small $\rho > 0$ there exists 
a constant $A_\rho$ such that for every $a \in K$ 
there exists a function $f^a_\rho(y)$ on $Y(\A)$ such 
that $f^a_\rho(a)\equiv 1, f^a_\rho(y)\equiv 0$ for 
$\rho(y,a) \ge \rho,\ 0\le f^a_\rho(y)\le 1$ everywhere, 
and $f^a_\rho(Y(X^\eps_t))+A_\rho t$ is a submartingale 
for all $\eps$ (see Stroock and Varadhan \cite{SV}).
\end{itemize}

In the proof we need an auxiliary lemmas. 
We study the random sums 
\begin{align}\label{random sum}
S_n=\sum_{k=1}^n v_k\om_k,\ \ n\ge 1,
\end{align}
where $\{\om_k\}_{k\ge 1}$ is a sequence of 
independent random variables with equal $\pm 1$ with 
equal probability $1/2$ each and $\{v_k\}_{k\ge 1}$ is 
a sequence such that 
$$
\lim_{n\to \infty} \frac{\sum_{k=1}^n v_k^2}{n}=\sigma.
$$ 
Here is a standard 

\begin{lem}\label{main lemma}
$\{S_n/n^{1/2}\}_{n\ge 1}$ converges in distribution to 
the normal dirtribution $\mathcal N(0,\sigma^2)$.
\end{lem}

Recall that a characteristic function of a random 
variable $X$ is a function $\phi_X:\R \to \mathbb C$ given by 
$\phi_X(t)=\E \exp (itX)$. Notice that it satisfies the 
following two properties:
\begin{itemize}
\item If $X,Y$ are independent random variables, 
then $\varphi_{X+Y}=\varphi_X\cdot\varphi_Y$.
\item $\varphi_{aX}(t)=\varphi_X(at)$.
\end{itemize}

A sufficient condition to prove convergence in distribution
is as follows.  
\begin{thm}[Continuity theorem \cite{Br}]
\label{continuity theorem}
Let $\{X_n\}_{n\ge 1},Y$ be random variables. 
If $\{\varphi_{X_n}(t)\}_{n\ge 1}$ converges to $\varphi_Y(t)$
for every $t\in\mathbb R$, then $\{X_n\}_{n\ge 1}$ converges 
in distribution to $Y$.
\end{thm}

A direct calculation shows that 
$$
\lim_{n\to \infty} \log \phi_{S_n/\sqrt n}(t)=-\dfrac{t^2}{2\sigma^2}
\qquad \text{ for all }\ \  t\in \R.   
$$ 
This way of proof was communicated to the authors by 
Yuri Lima. 

\

{\bf Acknowledgement:} The authors warmly thank Leonid Koralov 
for numerious envigorating discussions of various topics involving 
stochatic processes. Communications with Dmitry Dolgopyat, 
Yuri Bakhtin, Jinxin Xue were useful for the project 
and gladly acknowledged by the authors. The second author 
acknowledges partial support of the NSF grant DMS-1402164.

\end{document}